\documentclass[oneside,reqno,11pt]{amsart}

\usepackage{graphicx}
\usepackage{epstopdf}
\usepackage{amsmath}
\usepackage{amssymb}
\usepackage[top=1in, bottom=1.25in, left=1.25in, right=1.25in]{geometry}
\usepackage{textcomp}
\usepackage{mathrsfs}
\usepackage{tikz}
\usepackage{tikz-cd}
\usepackage[utf8]{inputenc}
\usepackage{mathabx}
\usepackage{bm}
\usepackage{amsthm}
\usepackage{amsfonts}
\usepackage[colorlinks=true,linkcolor=blue]{hyperref}
\usepackage{pb-diagram}
\usepackage{lipsum}
\usepackage{secdot} 
\usepackage{color}
\usepackage[normalem]{ulem}
\usepackage{enumitem}
\usepackage{dsfont}
\usepackage{appendix} 
\usepackage{color}
\usepackage{calc}
\usepackage{slantsc}
\usepackage[sc]{mathpazo}
\usepackage{indentfirst}
\usepackage{mathtools}
\usepackage{hyperref}
\usepackage{amscd}
\usepackage{caption}
\usepackage{subcaption}
\usepackage{fancybox}
\usepackage{wrapfig}
\usepackage[all]{xy}
\usepackage{verbatim}
\usepackage{stmaryrd}

\theoremstyle{plain}
\newtheorem{theorem}{Theorem}[section]
\newtheorem{lemma}[theorem]{Lemma}
\newtheorem{corollary}[theorem]{Corollary}

\newtheorem{prop}[theorem]{Proposition}
\newtheorem{defn}[theorem]{Definition}
\newtheorem{example}[theorem]{Example}
\newtheorem{remark}[theorem]{Remark}

\numberwithin{equation}{section}

\newcommand{\Mat}{\operatorname{Mat}}
\newcommand{\Image}{\operatorname{Im}}

\newcommand{\Gr}{\operatorname{Gr}}
\newcommand{\bi}{\mathbf{i}}

\newcommand{\twobar}{/\kern-0.5em/}
\newcommand{\threebar}{/\kern-0.5em/\kern-0.5em/}

\newcommand{\Z}{\mathbb{Z}}

\newcommand{\R}{\mathbb{R}}
\newcommand{\C}{\mathbb{C}}
\newcommand{\bS}{\mathbb{S}}
\newcommand{\bP}{\mathbb{P}}

\newcommand{\cD}{\mathcal{D}}

\newcommand{\cV}{\mathcal{V}}
\newcommand{\cE}{\mathcal{E}}

\newcommand{\Hom}{\mathrm{Hom}}

\newcommand{\End}{\mathrm{End}}

\newcommand{\GL}{\mathrm{GL}}

\newcommand{\cM}{\mathcal{M}}

\newcommand{\cO}{\mathcal{O}}

\newcommand{\pt}{\mathrm{pt}}

\begin{document}
	
\title{K\"ahler geometry of quiver varieties and machine learning}

\author{George Jeffreys}
\address{Department of Mathematics and Statistics, Boston University, 111 Cummington Mall, Boston MA 02215, USA}
\email{georgej@bu.edu}

\author{Siu-Cheong Lau}
\address{Department of Mathematics and Statistics, Boston University, 111 Cummington Mall, Boston MA 02215, USA}
\email{lau@math.bu.edu}

\maketitle

\begin{abstract}
	We develop an algebro-geometric formulation for neural networks in machine learning using the moduli space of framed quiver representations.  We find natural Hermitian metrics on the universal bundles over the moduli which are compatible with the GIT quotient construction by the general linear group, and show that their Ricci curvatures give a K\"ahler metric on the moduli.  Moreover, we use toric moment maps to construct  activation functions, and prove the universal approximation theorem for the multi-variable activation function constructed from the complex projective space.
\end{abstract}

\section{Introduction}

Machine learning by artificial neural networks has made exciting developments and has been applied to many branches of science in recent years.  Mathematically, stochastic gradient flow over a matrix space (or called the weight space) is the central tool.  The non-convex nature of the cost function has made the problem very interesting.  Current research has focused on different types of stochastic gradient flows and finding new types of networks, which have brought great improvements of computational efficiency.

In geometry and physics, the applications of gradient flow and Morse theory have a long history and have brought numerous fundamental breakthroughs.   For instance, the gradient flow of the Yang-Mills functional is used to find Hermitian Yang-Mills connections, whose existence in a stable holomorphic vector bundle is proved by Donaldson \cite{Don} and Uhlenbeck-Yau \cite{UY}.  The celebrated Ricci flow found by Hamilton \cite{Hamilton}, which is a crucial tool to solve the three-dimensional Poincar\'e conjecture, is essentially a gradient flow \cite{Perelman,Perelman2}.  Its K\"ahler analog has been an important tool in finding K\"ahler-Einstein metrics on Fano manifolds \cite{Yau-KE,Tian,Donaldson,CSW,CDS1,CDS2,CDS3}.  In these works, GIT quotients and finite-dimensional models have provided important motivations and guidelines \cite{Donaldson-GIT}.  Hamiltonian Floer theory \cite{Floer}, which is essentially Morse theory on the loop space, was invented to solve the Arnold conjecture \cite{FHS,Ono,FO}.  Various versions of Floer theory have been crucial ingredients in the study of mirror symmetry.

In this paper, we would like to develop a foundational \emph{algebro-geometric formulation} for neural networks in machine learning.  The theory of \emph{quiver representations}, which is a well-developed branch of mathematics motivated from Lie theory and has been an important tool in mathematical physics, will be well suited for this purpose.  

A quiver representation assigns to a directed graph $Q$ a bunch of vector spaces for the vertices and a bunch of linear maps for the arrows.  Such a construction is in common with neural networks.  However, in order to use quiver theory to formulate machine-learning neural networks, there are two main differences between these two subjects that needs to be addressed.

\begin{enumerate}
	\item Compactness of moduli space.  A moduli space of quiver representations \cite{King} is defined by identifying \emph{isomorphic} quiver representations using GIT quotients.  As a result, the moduli space is \emph{compact} when the quiver has no oriented cycle.  On the other hand, the matrix space used in neural networks is non-compact.  In machine learning, isomorphic quiver representations may correspond to physically different input or output information and in general cannot be identified.
	\item Non-linearity.  \emph{Activation functions}, which are non-linear maps on the vector spaces over the vertices, serve as a crucial ingredient to achieve machine learning of non-linear functions.  Such non-linearity jumps out of the category of quiver representations.  This is also related to the first point above.  Namely, such non-linear maps are not necessarily equivariant under the group of automorphisms of quiver representations.
\end{enumerate}

For the first point, we shall use framed quiver representations, which were first found by Nakajima \cite{Nakajima-Duke} in the study of affine Lie algebras.  A framed representation assigns to each vertex a vector space together with a choice of `framing' (for instance it is a basis in the simplest situation).  In the applications considered here, such a decoration makes sure that isomorphic framed quiver representations correspond to the same physical state.  Note that framed quiver moduli $\cM$ are also compact when $Q$ has no oriented cycle.  \emph{Compactness} is one of the main advantages of our algebro-geometric formulation, which makes sure the convergence of a gradient flow.

In this formulation, the weight matrices are encoded as morphisms between the \emph{universal vector bundles} (over the framed quiver moduli) associated to the vertices.  The data flow is encoded by sections of the universal bundles, which are sent from one to another bundles by the morphisms associated to the arrows of $Q$.  The cost function, and hence its gradient flow, is defined on the framed quiver moduli $\cM$.  

In particular, the \emph{critical points and the gradient flow are controlled by the topology of $\cM$} (for instance, the Morse inequalities).
The topology of a framed quiver moduli is well-understood by the work of Reineke \cite{Reineke} when $Q$ has no oriented cycle.  $\cM$ is an iterated Grassmann bundle, and its Poincar\'e polynomial is a product of that of the Grassmannians.

For the purpose of gradient flow, one needs to choose a \emph{K\"ahler metric} on $\cM$, and also Hermitian metrics on the universal vector bundles.  As a result, we have found metrics that are defined by explicit beautiful formulae.  These metrics are not just $U_{\vec{d}}$-equivariant so that they descend to symplectic quotients, but are also $\GL_{\vec{d}}$-equivariant and hence compatible with the GIT construction of $\cM$.  Moreover, they are compatible with the iterated Grassmann structure found by Reineke.  In application, such metrics would simplify the actual computational algorithm over the quiver moduli.  They are summarized as follows.




\begin{theorem}[Combining Theorem \ref{thm:metric}, \ref{thm:Ricci},\ref{thm:metric-with-cycles}] \label{thm:metric-int}
	Let $Q$ be an arbitrary quiver.  
	Fix a vertex $i\in Q_0$.  
	Let $\rho$ be the row vector whose entries are 
	$V_\gamma e^{\left(t(\gamma)\right)}$,
	where $\gamma$ is any path whose head $h(\gamma)$ is $i$ (including the trivial path), $t(\gamma)$ denotes its tail, and $V_{\gamma} \in \Hom(\C^{d_{t(\gamma)}},\C^{d_{h(\gamma)}})$ is the representing matrix of $\gamma$.  Then
	$$(\rho_i\rho_i^*)^{-1} = \left(\sum_{h(\gamma)=i} \left(V_{\gamma} e^{\left(t(\gamma)\right)}\right)\left(V_{\gamma} e^{\left(t(\gamma)\right)}\right)^*\right)^{-1}$$ 
	is $\GL_{\vec{d}}$-equivariant, and it descends to a metric on the universal bundle $\cV_i$ over a certain domain of convergence $\cM^\circ$.
	
	When $Q$ has no oriented cycle, $\cM^\circ = \cM$.  Moreover, the Ricci curvature of the induced metric on $\bigotimes_{i\in Q_0} \cV_i$ gives a K\"ahler metric on $\cM$.
\end{theorem}
The precise definition of $M^\circ$ is given in Section \ref{sec:cycle}.

In this paper, we focus on the framed quiver moduli defined over complex numbers.  In actual applications, we can also restrict to real coefficients.  Then the above formula defines a bundle metric over $\cM_\R$, and the Ricci curvature gives a Riemannian metric on $\cM_\R$.


Now let us address the second point.  Namely, we need to introduce \emph{non-linearity} in addition to the usual theory of quiver representations.  By definition, morphisms between universal vector bundles over $\cM$ are linear along fibers.  They correspond to weight matrices in neural networks.  To introduce non-linearity, we shall treat the universal bundles as fiber bundles and construct suitable \emph{fiber-bundle maps} that play the role of activation functions.

One of the commonly-used activation functions is
$$ \frac{e^{2x}}{1+e^{2x}}: \R \to (0,1). $$
We observe that this function also appears in the base of the symplectic trivialization of the open dense toric orbit of $\bP^1$ as a toric variety:
$$ (\C^\times,\omega_{\C\bP^1}) \cong ((0,1) \times \bS^1,\omega_{\mathrm{std}}), $$
or lifted to the universal cover:
$$ (\C,\pi^*\omega_{\C\bP^1}) \cong ((0,1) \times \R,\omega_{\mathrm{std}}).$$
Here, $\omega_{\C\bP^1}$ denotes the Fubini-Study metric of $\C\bP^1$, that is, the standard area form of the unit sphere; $\omega_{\mathrm{std}} = dx \wedge dy$ is the standard symplectic form.

Similarly, another activation function
$$ \frac{z}{\sqrt{1+|z|^2}}: \C \to \C$$
also arises as a symplectic trivialization:
$ (\C,\omega_{\C\bP^1}) \cong (\{w \in \C: |w|<1\},\omega_\C). $

Motivated from these observations, we consider 
$$\sigma(x)=\left(\frac{e^{2x_i}}{1+\sum_{j=1}^ne^{2x_j}}\right)_{i=1}^n: \R^n \to \Delta$$
and
$$ \psi(\vec{z})=\left(\frac{z_i}{\sqrt{1+\|\vec{z}\|^2}}\right)_{i=1}^n: \C^n \to \{\vec{w} \in \C^n: \|\vec{w}\| < 1\}
$$
as multi-variable activation functions, where $\Delta$ denotes the standard simplex with vertices $0$ and $\epsilon_1,\ldots,\epsilon_n$, the standard basis of $\R^n$.  They arise from symplectomorphisms $(\C^n, \omega_{\C\bP^n}) \cong (\{\vec{w} \in \C^n: \|\vec{w}\| < 1\,\omega_{\C^n}\})$.  More generally, these come from moment maps of \emph{toric varieties} \cite{Guillemin,Abreu}.  We note that $\psi$ has an advantage of being $U(n)$-equivariant.

The \emph{universal approximation theorem} (see for instance \cite{Cybenko,Petrushev,MMC,Pinkus})  provides a theoretical foundation for the success of neural networks.  In existing literature, the theorem was proved for single-variable activation functions.  

In this paper, we prove the universal approximation theorem for the above multivariable function $\sigma = \sigma_{\R^n}$.  Note that $\sigma$ is the softmax function $\left(\frac{e^{2x_i}}{\sum_{j=0}^n e^{2x_j}}\right)_{i=0}^n$ restricted to the hyperplane $x_0=0$ and composed with the projection along $x_0$-direction.  We shall restrict to real coefficients in this theorem.

\begin{theorem}[same as Theorem \ref{thm:app}] \label{thm:app-int}
	Let $K$ be a compact set of $\R^{d_1}$, and $f: K \to \R^{d_3}$ a continuous function.  For any $\epsilon>0$, there exists $d_2 > 0$ and $W_1 \in \Mat(d_2,d_1)$, $W_2 \in \Mat(d_3,d_2)$, $b \in \R^{d_2}$ such that $\|f^U_{W_1,W_2,b} - f\|_{L^2(K)} < \epsilon$.  Here, $	f^U_{W_1,W_2,b}(x) = W_2 \cdot \sigma_{\R^{d_1}}(W_1\cdot x + b)$ is the function coming from the $A_3$ quiver.
\end{theorem}
The $A_3$ quiver corresponds to the feed-forward network with one input layer, one middle layer and one output layer.  See Figure \ref{fig:A3}.


\begin{figure}[h]
	\begin{tikzcd}
	a \arrow[r, "\alpha"] & b \arrow[r, "\beta"] & c
	\end{tikzcd}
\caption{The $A_3$ quiver.}
	\label{fig:A3}
\end{figure}

The above theorem is proved by using the tropical limit of the toric manifold $\bP^n$, and a geometric object that we call a centered polyhedral web, which is an analog of a tropical variety in an integral affine manifold.  Since we do not have integral structure in the context here, we need to invent this new notion.

In above, we have focus on explaining non-linearity for a single vector space.  We shall \emph{globalize them as non-linear fiber-bundle morphisms} for the universal bundles over $\cM$.  This can be achieved with the help of Hermitian metrics on the universal bundles, so that the Fubini-Study metric on $\bP^n$ can be globalized as a fiberwise symplectic structure on projective bundles over $\cM$.  Actually, the globalization from a single framing vector space $V$ to the universal fiber bundle over $\cM$ \emph{works for any continuous function} $V\to V$ (and in particular for a symplectomorphism from $V$ to its image).  Combining the ingredients explained above together, we can construct a gradient flow over the framed quiver moduli to achieve machine learning.  The detail is given in Section \ref{sec:nonlinear}.

Such an algebro-geometric formulation has several \emph{advantages}.  First, the gradient flow under consideration runs in a \emph{compact} manifold.  This ensures the existence of absolute extrema, convergence of the flow, and upper bound for the norm of the gradient vector field.  Second, because of compactness, the flow is constrained by \emph{topology} of the manifold due to Morse theory.  See Section \ref{sec:topo}.  Finally, the moduli space has extra \emph{symmetry} coming from framing.   If we use activation functions that respect this symmetry (for instance $\psi$ above enjoys $U(n)$-equivariance), we can perform dimension reduction which improves the effectiveness of the network.  (See Proposition \ref{prop:sym} and \ref{prop:sym2}.)

In summary, from this point of view, the success of neural network is resulted from the interplay between algebraic morphisms and (transcendental) symplectomorphisms.  Interestingly, such an interplay is also an important feature that occurs in the study of complete integrable systems and mirror symmetry for toric manifolds and flag varieties, see for instance \cite{Guillemin,FLTZ12,Abouzaid-toric,CLL,NNU,HKL}.





\subsection*{Some related works}
The relation between neural network and quiver representation was investigated in the recent paper \cite{Armenta-Jodoin}.  Their work considered the quotient space by $(\C^\times)^N$ of pairs $(W,f)$, where $W$ is a quiver representation of $Q$ with the dimension vector $\vec{1}$, and $f$ associates each vertex a function $\C\to\C$ (playing the role of an activation function).  Moreover, in dimension $\vec{1}$ (which is a typical case for machine learning), they invented an interesting way of encoding the data flow as a quiver representation.  (In our work, the data flow is given as sections of universal bundles over the quiver moduli.)

The approach and the goal of this paper is rather different.  We aim at formulating machine learning as a gradient flow over a compact quiver moduli.  In \cite{Armenta-Jodoin}, quiver representations were used in encoding the data in the network; however, the machine learning process was independent of the quiver moduli.  Moreover, `double-framing' was used, and the corresponding moduli space is non-compact.

The map $(W,f) \mapsto W$ gives an infinite-dimensional fiber bundle over the quiver moduli $\cM$, whose fibers are the spaces of choices of activation functions.  In a typical program of machine learning, the activation functions are fixed during the optimization process.  In order to formulate the program as a gradient flow over the compact moduli $\cM$, we found a non-trivial way by equipping intermediate vertices with additional framings and metrics, so that we can lift $f$ to be a well-defined fiber-bundle map over $\cM$.  
Note that $f$ is not equivariant under the group action of $(\C^\times)^N$ ($\GL_{\vec{d}}$ in the higher rank case).  Such a lifting is an important non-trivial step.

Furthermore, we have dealt with representations of general rank $\vec{d}$, and a class of activation functions coming from toric symplectomorphisms.  Different functions (on the same domain and target) are obtained if we deform the toric K\"ahler metric.  To also optimize the activation functions during the learning process (see also \cite{GGL}), we may consider a gradient flow on $\cM\times \mathcal{K}$ where $\mathcal{K}$ denotes the moduli of toric K\"ahler metrics in the same class.  By the celebrated works of \cite{Donaldson-GIT,Semmes}, $\mathcal{K}$ is an infinite-dimensional negatively curved symmetric space.

Recently, there is a rising interest of applying geometric techniques to the study of neural networks.  For instance, in the works \cite{GBH,CYRL}, hyperbolic spaces are applied to machine learning in graphs and achieved great performance.  

Moreover, the applications of symmetry and group equivariance in neural networks were studied and developed in \cite{CW,CGW,CWKW-spherical,CWKW,CAWHCW,dehaan2020natural}.  Overall, these works aim at capturing symmetry of the input data and designing networks that are adapted to such symmetry.  Moreover, homogeneous spaces (in place of vector spaces) have been employed in layers of convolutional neural networks.

In comparison, our paper aims at revealing the geometric nature of neural networks and build a connection with algebraic geometry.  
The resulting framed quiver moduli, which has interesting topology and metrics, is the main geometric object of interest.  Furthermore, we study activation functions that respects the `intrinsic symmetry' over the quiver moduli, which can provide a more effective algorithm by dimension reduction.

In the reverse direction, there are interesting applications of machine learning in frontier geometry and physics.  For instance, \cite{He-Yau} used machine learning to solve problems in computing graph Laplacians, such as recognizing graph Ricci-flatness and predicting the spectral gap.  In physics, \cite{HSTT2,HSTT} used deep learning to study AdS/CFT correspondence by discretizing the equation of motion.  Since we have formulated machine learning using quiver representations, it will be interesting to find direct relations between quiver gauge theory and these problems that can be attacked via machine learning.

\subsection*{Organization of this paper} In Section 2, we will take a quick review on quiver representations and their moduli spaces. In Section 3, we will construct nice Hermitian metrics on universal bundles over the moduli.  For readers who are mainly interested in machine learning, Section 3 can be skipped for the first reading.  Then we give an algebro-geometric formulation of neural network using quiver representations in 
Section 4.  In Section 5, we prove the universal approximation theorem for the multivariable activation function $\sigma$.\\

\section*{Acknowledgment}
We are grateful to Marco Antonio Armenta for informing us about the work \cite{Armenta-Jodoin} and the further useful discussions.
We express our gratitude to Shing-Tung Yau for his generous encouragement.  The work of S.C. Lau in this paper is partially supported by the Simons collaboration grant.

\section{Review of framed quiver moduli}

Let $Q$ be a directed graph.  Denote by $Q_0, Q_1$ the set of vertices and arrows respectively.  A quiver representation $V$ with dimension vector $\vec{d} \in \Z_{\geq 0}^{Q_0}$ associates each arrow $a$ with a matrix $V(a)$ of size $d_{h(a)}\times d_{t(a)}$ (where $h(a), t(a)$ denote the head and tail vertices of $a$ respectively).  The set of complex quiver representations with dimension $\vec{d}$ form a vector space denoted by $R_{\vec{d}}(Q)$.  The moduli space of quiver representations is a GIT quotient of $R_{\vec{d}}(Q)$ by the group of isomorphisms $\GL(\vec{d}) = \prod_{i \in Q_0} \GL(d_i,\C)$ \cite{King}, where $\GL(\vec{d})$ acts on $R_{\vec{d}}(Q)$ via
\begin{equation}
g\cdot (V(a): a\in Q_1) = (g_{h(a)}\cdot V(a) \cdot g_{t(a)}^{-1}: a \in Q_1).
\label{eq:GL(d)}
\end{equation}


In the applications we consider in this paper, since the vector space over the input and output vertices are equipped with fixed basis with physical meanings, we need to use framed quiver representations \cite{Nakajima-Duke,Nakajima-JAMS, Crawley-Boevey, Reineke}.

Let $\vec{d}, \vec{n} \in \Z_{\geq 0}^{Q_0}$.  $\vec{n}$ will be the dimension vector for the framing, which is a linear map $e^{(i)}: \C^{n_i} \to V_i$ at each $i\in Q_0$ (where $V_i = \C^{d_i}$).  Since we will take a quotient by $\GL(\vec{d})$, we shall think of $V_i$ as a vector space without a preferred basis, while $\C^{n_i}$ is equipped with the standard basis.



\begin{defn}
The vector space of framed representations is given by
$$R_{\vec{n},\vec{d}} = R_{\vec{d}} \times \bigoplus_{i\in Q_0}\Hom(\C^{n_i}, \C^{d_i}).$$\\
It carries a natural action of $\GL(\vec{d})$ given by $g \cdot (V, e) = (g \cdot V, (ge^{(i)}: i\in Q_0) )$, where $g \cdot V$ is given by Equation \eqref{eq:GL(d)}.
\end{defn}

We need to remove unstable framed representations from $R_{\vec{n},\vec{d}}$ in order to get a nice quotient by $\GL(\vec{d})$.  

\begin{theorem}[\cite{Nakajima-proc}]
$(V, e) \in R_{\vec{n},\vec{d}}$ is called \textit{stable} if there is no proper subrepresentation $U$ of $V$ which contains $\mathrm{Im} \,e$. The set of all stable points of $R_{\vec{n},\vec{d}}$ is denoted by $R_{\vec{n},\vec{d}}^s$.
Then the quotient $\cM_{\vec{n},\vec{d}} := R_{\vec{n},\vec{d}}^s / \GL(\vec{d})$ is a smooth variety, which is called to be a framed quiver moduli.
\end{theorem}

Actually $\cM_{\vec{n},\vec{d}}$ can be formulated as a GIT quotient \cite{Crawley-Boevey,Reineke}.  Namely, by adding an extra vertex labeled as $\infty$ to the quiver and $n_i$ arrows from the vertex $\infty$ to the vertex $i$, $(V,e)$ can be identified as a usual representation of this bigger quiver with the dimension vector $(\vec{d},1)$.  The above stability condition can be rewritten as slope stability, and hence it is a GIT quotient \cite{King}.  

Since $(d,1)$ is a primitive vector, $\cM_{\vec{n},\vec{d}}$ is a smooth fine moduli.  There are universal vector bundles $\cV_i$ over $\cM_{\vec{n},\vec{d}}$ corresponding to each vertex $i$, with fibers $\cV_i|_{[V,e]} = V_i$.

\begin{example}
	For the quiver with a single vertex and no arrow, and $n>d$, $$\cM_{n,d} = \Gr(n,d) = \{e \in \Hom(\C^n,\C^d): e \textrm{ is surjective}\}/\GL_{\vec{d}}$$ 
	is the (dual) Grassmannian.  We have the tautological bundle $\cV$ over $\Gr(n,d)$.  (Note that this tautological bundle is dual to the one on $\Gr(d,n)\cong \Gr(n,d)$.)
\end{example}

The topology of $\cM_{\vec{n},\vec{d}}$ is well-understood.  Let's make an ordering of the vertices.  Namely the vertices are labeled by $\{1,\ldots,N\}$, such that $i<j$ implies there is no arrow going from $j$ to $i$.  Such a labeling exists if $Q$ has no oriented cycle.


\begin{theorem}[Reineke \cite{Reineke}] \label{thm:Reineke}
Assume $Q$ has no oriented cycle.  Consider the chain of iterated Grassmann bundles
$M^{(N)}\stackrel{p_N}{\to} M^{(N-1)} \stackrel{p_{N-1}}{\to} \ldots \stackrel{p_{2}}{\to} M^{(1)} \stackrel{p_{1}}{\to} \pt$
(where $\pt$ denotes a singleton)
defined by induction:
$$M^{(i)}=\Gr_{M^{(i-1)}}\left(\underline{\C^{n_{i}}}\oplus\bigoplus_{j\to i}p_{i-1}^*\dots p_{j+1}^*(S_j), d_{i}\right) \to M^{(i-1)},$$
where $S_i$ denotes the tautological bundle on $M_i$ (as a Grassmann bundle over $M_{i-1}$).  (The direct sum is over each arrow $j\to i$.)
Then $\cM_{\vec{n},\vec{d}}\cong M^{(N)}$, with universal bundles $\mathcal{V}_i\cong p_N^*\dots p_{i+1}^*S_i$ for all $i\in Q_0$.\\
\end{theorem}

\begin{corollary} [Reineke \cite{Reineke}] \label{cor:Reineke}
The Poincare polynomial of nonempty $\cM_{\vec{n},\vec{d}}$ is given by $$\prod_{i\in Q_0}\binom{n_i+\sum_{j\to i}d_j}{d_i}_{q^2}$$ 
where $$\binom{n}{d}_q=\prod_{k=1}^d\frac{q^{n-d+k}-1}{q^k-1}.$$\\
\end{corollary}

\begin{remark}
	In \cite{Reineke}, the framing $e$ goes in the other direction (from $\C^{d_i}$ to $\C^{n_i}$).  The above theorem is stated in the dual way, which is the convention we take for the rest of this paper.
\end{remark}

\begin{example} \label{ex:A3}
	Consider the $A_3$-quiver which has three vertices $i=1,2,3$ and two arrows $a_1: 1\to 2,a_2: 2 \to 3$.  Suppose $n_1 = d_1$, $n_2 =d_2+1$ and $n_3=d_3$.  Then the iterated Grassmann bundle is $M^{(3)} \to M^{(2)} \to M^{(1)}$, where $M^{(1)} = \Gr(d_1,d_1)=\pt$ (and its tautological bundle is the vector space $\C^{d_1}$); $M^{(2)} = \Gr(d_2+1+d_1,d_2)$ is equipped with the tautological bundle $S_2$ of rank $d_2$; $M^{(3)} = \Gr_{\Gr(d_2+1+d_1,d_2)} (\underline{\C}^{d_3} \oplus S_2,d_3)$ is a Grassmannian bundle over $\Gr(d_2+1+d_1,d_2)$ with fibers $\Gr(d_2+d_3,d_3)$. The corresponding Poincare Polynomial will be $$\left(\prod_{k=1}^{d_3}\frac{q^{2(d_2+k)}-1}{q^{2k}-1}\right)\left(\prod_{k=1}^{d_2}\frac{q^{2(d_1+1+k)}-1}{q^{2k}-1}\right).$$\\
	
	See Figure \ref{fig:A3Grass}.
	
	\begin{figure}[h]
		\begin{tikzcd}
		1 \arrow[d, "a_1"] & Gr(d_1,d_1)=\pt \arrow[l] \\ 2 \arrow[d, "a_2"] & Gr(d_2+1+d_1, d_2) \arrow[l]\arrow[u] \\ 3 & \Gr_{\Gr(d_2+1+d_1,d_2)} (\underline{\C}^{d_3} \oplus S_2,d_3) \arrow[l]\arrow[u]
		\end{tikzcd}
		\caption{The iterated Grassmann bundles associated to the $A_3$ quiver.}
		\label{fig:A3Grass}
	\end{figure}
\end{example}


\section{Hermitian Metric over framed quiver moduli}


In constructing fiber-bundle endomorphisms, it will be crucial to consider K\"ahler metrics on universal bundles. In this section, we find a beautiful formula for the canonical metric on the universal bundle $\cV_i$ over $\cM_{ n, d}$ written in homogeneous coordinates. Using this formula, we then show that the sum of Ricci curvatures over the vertices $i$ give a K\"ahler metric on $\cM_{\vec{n},\vec{d}}$.

First, let us begin by recalling the typical example $\Gr(n,k)$.

\subsection{The Grassmannian} \label{sec:Herm}
Consider
\begin{equation*}
\Gr\left(n,k\right)=\Mat^\C_{n,k}\sslash_{\chi=1} U(k) =\left\{e\in \Mat^\C_{n,k}: ee^{*}=I_{k}\right\}\big/U\left(k\right)
\end{equation*}
for $n\geq k$.
Here we have used the dual description which better matches the frame convention used in this paper.  Namely, $\Gr\left(n,k\right)$ parametrizes $k$-dimensional quotient vector spaces of a fixed $n$-dimensional vector space as opposed to $k$-dimensional subspaces.  The moment map for the standard $U(k)$-action on $\Mat^\C_{n,k}$ is $ee^*: \Mat^\C_{n,k} \to \bi \mathfrak{u}_k$.  We have taken the moment map level $\chi = 1$ in the above symplectic reduction.  Note that $U(k)$ is acting on the left, although in the above expression $U(k)$ appears on the right.

Writing $e=\left(b,p\right)$ where $b \in \Mat^\C_{k,k}$ and $p \in \Mat^\C_{n-k,k}$, the moment-map equation $ee^*=I_k$ becomes
\begin{equation*}
bb^{*}+pp^{*}=I_{k}.
\end{equation*}

We shall consider the chart defined by 
$$U=\{[b,p] \in \Gr(n,k): \det b \not= 0\} \cong \Mat^\C_{n-k,k}$$ 
where the identification is given by the holomorphic coordinates
$$\zeta^h=b^{-1}p \in \Mat^\C_{n-k,k}.$$
We also have the symplectic coordinates
\begin{equation*}
\zeta^u =\left(b^{*}b\right)^{\frac{1}{2}}b^{-1}p\in \Mat^\C_{n-k,k}.
\end{equation*}
The entries of $\zeta^u$ are not meromorphic functions.  On the other hand, $\zeta^u$ has the advantage that it satisfies the moment-map equation
\begin{equation}
b^*b+\zeta^u (\zeta^u)^{*}=I_{k}.
\label{eq:Gr-moment}
\end{equation}
(Note that the first term is $b^*b$ instead of $bb^*$.)


The construction of $\zeta^u$ uses the polar decomposition
$$ b=\left(b\left(b^{*}b\right)^{-\frac{1}{2}}\right) \left(b^{*}b\right)^{\frac{1}{2}}$$
where $\left(b\left(b^{*}b\right)^{-\frac{1}{2}}\right) \in U(k)$ and $\left(b^{*}b\right)^{\frac{1}{2}} \in \bi \mathfrak{u}(k)$ is positive definite.  We obtain the coordinates $\zeta^u$ by observing
$$ [b, p]= \left[\left(b^{*}b\right)^{\frac{1}{2}},\zeta^u\right]$$
using the left-$U(k)$-action, such that the first component $\left(b^{*}b\right)^{\frac{1}{2}}$ is Hermitian, and is determined $\zeta^u$ due to the moment-map equation \eqref{eq:Gr-moment}.

The two coordinate systems are related by
\begin{equation}
\zeta^u =\left(b^{*}b\right)^{\frac{1}{2}}\cdot \zeta ^{h}.
\label{eq:zeta^u-h}
\end{equation}

Let $S$ be the tautological vector bundle whose fibers are the quotient vector spaces.  (This is dual to the tautological bundle of $\Gr(k,n) \cong \Gr(n,k)$.)  It can be written as the quotient of the trivial bundle:
\begin{equation*}
S=\left(\left\{ee^{*}=I_{k}\right\}\times \mathbb{C}^{k}\right)\big/U\left(k\right)
\end{equation*}
where the left action of $U(k)$ on $\C^k$ is the standard one.

We now take the standard metric on $\mathbb{C}^{k}$, which is preserved by $U(k)$ and hence descends to a metric $H$ of $S$.

Denote the standard basis of $\C^k$ by $\epsilon_j$ for $j=1,\ldots,k$.
Under this metric, we have the lifting of a local Hermitian frame over the chart $U=\{\det b\not=0\}$ being $$u_{i}=b\left(b^{*}b\right)^{-\frac{1}{2}}\cdot \epsilon_{i}$$
since
$\left[b, p, b\left(b^{*}b\right)^{-\frac{1}{2}}\epsilon _{i}\right]\sim \left[\left(b^{*}b\right)^{\frac{1}{2}},\zeta^u,\epsilon_{i}\right]$.

We also have the lifting of a local holomorphic frame $$h_{i}=b\epsilon _{i} = b\left(b^{*}b\right)^{\frac{1}{2}}b^{-1}u_{i}$$
since
$\left[b, p, b\epsilon _{i}\right]\sim \left(I_{k},\zeta ^{h},\epsilon _{i}\right)$.  The two frames are related as follows.

\begin{lemma} \label{lem:Gr-h-u}
	$h_{i}=u_{j}a_{i}^{j}$ where
	$\left(a_{i}^{j}\right)=\left(b^{*}b\right)^{\frac{1}{2}}$, $i$ is indexing the colomns and $j$ is indexing the rows.
\end{lemma}

\begin{proof}
	Consider
	\begin{equation*}
	\left(h_{1}\ldots h_{k}\right)=b\left(b^{*}b\right)^{\frac{1}{2}}b^{-1}\left(u_{1}\ldots u_{k}\right)=\left(u_{1}\ldots u_{k}\right)\left(a_{i}^{j}\right).
	\end{equation*}
	Thus $\left(a_{i}^{j}\right)=\left(u_{1}\ldots u_{k}\right)^{-1}b\left(b^{*}b\right)^{\frac{1}{2}}b^{-1}\left(u_{1}\ldots u_{k}\right)$.
	
	\begin{equation*}
	\left(u_{1}\ldots u_{k}\right)=b\left(b^{*}b\right)^{-\frac{1}{2}}
	\end{equation*}
	since $u_{i}=b\left(b^{*}b\right)^{-\frac{1}{2}}\epsilon _{i}$. Result follows.
\end{proof}


\begin{prop}
	The metric $H$ defined above on the tautological bundle $S$ is represented by the matrix
	$\left(I_{k}+\zeta ^{h}\left(\zeta ^{h}\right)^{*}\right)^{-1}$
	in the local holomorphic frame $h_i$ and the local coordinates $\zeta^h$.
\end{prop}

\begin{proof}
	Using Lemma \ref{lem:Gr-h-u},
	$$
	\left(H\left(h_{i},h_{p}\right)\right)
	=\sum_{j,l}\left(H\left(a_{i}^{j}u_{j},a_{p}^{l}u_{l}\right)\right)
	=\sum_{j,l}\left(\overline{a_{i}^{j}}a_{p}^{l}H\left(u_{j},u_{l}\right)\right)
	=\sum_{j} \left(\overline{a_{i}^{j}}a_{p}^{j}\right) 
	= b^*b
	$$
	where $i$ is indexing the rows and $p$ is indexing the columns.
	By the moment map equation \eqref{eq:Gr-moment} and the relation \eqref{eq:zeta^u-h},
	$$
	b^{*}b+\left(b^{*}b\right)^{\frac{1}{2}}\cdot \zeta ^{h}\left(\zeta ^{h}\right)^{*}\left(b^{*}b\right)^{\frac{1}{2}}=I_{k}. 	
	$$
	Then
	$$
	I_{k}+\zeta ^{h}\left(\zeta ^{h}\right)^{*}=\left(b^{*}b\right)^{-1}.
	$$
	Hence
	\begin{equation*}
	H=\left(I_{k}+\zeta ^{h}\left(\zeta ^{h}\right)^{*}\right)^{-1}.
	\end{equation*}
\end{proof}

\begin{example}
	Let's consider the simplest example:
	\begin{equation*}
	\mathbb{P}^{1}=\Gr(2,1)=(\C^2-\{0\})/\C^\times = \mathbb{S}^{3}/\mathrm{U}\left(1\right).
	\end{equation*}
	The tautological bundle for $\Gr(2,1)$ is
	\begin{equation*}
	S = \left(\mathbb{S}^{3}\times \mathbb{C}\right)/U\left(1\right)
	\end{equation*}
	where $U(1)$ acts on $\C$ in the standard way, and it acts on both factors on the left.  (Note that this is dual to the usual notion of the tautological bundle of $\Gr(1,2)=\bP^1$, since we are now considering the family of quotient lines of $\C^2$, which are dual to subspaces of $\C^2$.)
	
	Let's take the standard metric on $\mathbb{C}$.  We have the local Hermitian frame (over $z_{1}\neq 0$) $u$ given by
	\begin{equation*}
	\left(z_{1},z_{2},z_{1}/\left| z_{1}\right| \right) \stackrel{U(1)}{\sim} \left(\left| z_{1}\right| ,\zeta^u ,1\right)
	\end{equation*}
	where $\zeta^u$ is the coordinate of $\bP^1$ which belongs to the open unit disc, and $|z_1|$ is determined by the moment-map equation
	\begin{equation*}
	\left| z_{1}\right| ^{2}+\left| z_{2}\right| ^{2}=\left| z_{1}\right| ^{2}+\left| \zeta^u \right| ^{2}=1.
	\end{equation*}
	We also have the local holomorphic frame $h$ defined by
	\begin{equation*}
	\left(z_{1},z_{2},z_{1}\right) \stackrel{\C^\times}{\sim} \left(1,\zeta^{h},1\right).
	\end{equation*}
	The unitary and holomorphic coordinates are related by
	$\zeta^u =\left| z_{1}\right| \cdot \zeta^{h} = (1-|\zeta^u|^2) \cdot \zeta^{h}$ for $\zeta^h \in \C$.
	The frames are related by
	\begin{equation*}
	h=\left| z_{1}\right| \cdot u.
	\end{equation*}
	The Hermitian frame $u$ always have length one.  Writing the metric in the holomorphic frame $h$:
	\begin{align*}
	\left| h\right| ^{2}&=\left| z_{1}\right| ^{2}=1-\left| \zeta^u \right| ^{2}=1-\left| z_{1}\right| ^{2}\left| \zeta ^{h}\right| ^{2} \\
	&=1-\left(1-\left| z_{1}\right| ^{2}\left| \zeta ^{h}\right| ^{2}\right)\left| \zeta ^{h}\right| ^{2}=\ldots  \\
	&=\frac{1}{1+\left| \zeta ^{h}\right| ^{2}}. 
	\end{align*}
	This is the standard metric on $\cO_{\bP^1}(1)$, whose curvature gives the Fubini-Study metric on $\bP^1$.
\end{example}

\subsection{Metric on framed quiver moduli}

We have seen that the standard metric on the trivial bundle over $\Mat^\C_{n,k}$ descends to give the standard metric on $\Gr(n,k)$.  However, it turns out that for the framed quiver moduli, the standard metric on the trivial bundle over $R_{\vec{n},\vec{d}}$ is not good from the GIT quotient point of view, namely it is not equivariant under $\GL_{\vec{d}}$.  In this section, we find a nice metric over $R_{\vec{n},\vec{d}}$ which is equivariant under $\GL_{\vec{d}}$.



Recall from the last section that $\cM_{\vec{n},\vec{d}} = R_{\vec{n},\vec{d}}^{s}/\GL_{\vec{d}}$.  The universal bundle over the vertex $i$ is given by
$$\cV_i = \left(R_{\vec{n},\vec{d}}^{s}\times \C^{d_i}\right)\big/\GL_{\vec{d}}$$
where $\GL_{\vec{d}}$ acts diagonally on the left, the factor $\GL(d_i,\C)$ of $\GL_{\vec{d}}$ acts on $\C^{d_i}$ in the standard way, and other factors of $\GL_{\vec{d}}$ act trivially on $\C^{d_i}$.

There is an equivalent description of $\cM = \cM_{\vec{n},\vec{d}}$ and the universal bundle $\cV_i$ in terms of symplectic quotient.  Namely, let $\mu: R_{\vec{n},\vec{d}} \to \bi \mathfrak{u}_{\vec{d}} $ be the moment map.  Explicitly, $\mu = (\mu_i)_{i\in Q_0}$ where
$$ \mu_i = e^{(i)} (e^{(i)})^* - \sum_{t(a)=i} V_a^* V_a + \sum_{h(a')=i} V_{a'} V_{a'}^*.$$
Then define
$$ \cM_{\vec{n},\vec{d}} = \mu ^{-1}\left\{-c \right\}\big/U_{\vec{d}}$$
for the following level $c$.

\begin{lemma}
	The slope stability condition $\left(1,\vec{0}\right) \in \C^{\hat{Q}_0}$ corresponds to the moment-map level
	$c = \left(-I_{d_i}\right)_{i\in Q_0} \in \bi \mathfrak{u}_{\vec{d}}$, where $I_k$ denotes the identity matrix of rank $k$.
\end{lemma}

\begin{proof}
	The character taken in King's stability \cite{King} corresponding to $\left(1,\vec{0}\right)$ is
	\begin{equation*}
	\left(1+\Sigma \vec{d} \right)\left(\left(1,\vec {0}\right)-\frac{1}{1+\Sigma \vec{d} }\left(1,\vec {1}\right)\right)=\left(\Sigma \vec{d} ,-1,\ldots ,-1\right)
	\end{equation*}
	where $\Sigma \vec{d} = \sum_{i\in Q_0} d_i$, and the first entry is over the root vertex.  (Note that there is no group action over the root vertex.) Thus we should take $c$ to be $-I_{d_i}$ over each vertex $i$.
\end{proof}

The universal bundle over the vertex $i$ is then given by 
$$ \cV_i = \left(\mu ^{-1}\left\{I_{\vec{d}} \right\}\times \C^{d_i}\right)\big/U_{\vec{d}}.$$



Let's review some very basic definitions about group actions.

\begin{defn}
	Suppose a Lie group $G$ acts on a vector bundle $V \stackrel{\pi}{\to} M$ equivariantly, namely, $g \circ \pi = \pi \circ g$ for all $g \in G$, and the action is fiberwise linear.  A metric $H$ on $V$ is said to be $G$-equivariant if
	\begin{equation*}
	H_{x}\left(v,w\right)=H_{g\cdot x}\left(g\cdot v,g\cdot w\right).
	\end{equation*}
\end{defn}

Writing in matrix form when $G=\GL(n,\C)$, the above equation is
$v^{*}\cdot H_{x}\cdot w=v^{*}\cdot\left(g^{*}\cdot H_{g\cdot x} \cdot g\right)\cdot w,$ that is,
\begin{equation}
(g^*)^{-1} \cdot H_{x} \cdot g^{-1}=H_{g\cdot x}.
\label{eq:equiv}
\end{equation}

The following easily follows from the definition.

\begin{lemma} \label{lem:equiv}
	Suppose a Lie group $G$ acts on a vector bundle $V \stackrel{\pi}{\to} M$ equivariantly and fiberwise linearly, and the action of $G$ on $M$ is free and proper.  A Hermitian form $H$ on $V$ descends to the corresponding bundle over the quotient $M/G$ if and only if $H$ is $G$-equivariant.
\end{lemma}

For framed quiver varieties, we have the framing map $e^{(j)}\colon R_{\vec{n},\vec{d}}^{s}\rightarrow \mathrm{Hom}\left(\C^{n_j},\C^{d_j}\right)$ for each vertex $j$.  Using this, we cook up a $\GL_{\vec{d}}$-invariant Hermitian form on the trivial bundle $\underline{\C^{d_i}}\to R_{\vec{n},\vec{d}}$, which descends to a metric on $\cV_i \to \cM_{\vec{n},\vec{d}}$.

\begin{theorem} \label{thm:metric}
	Suppose $Q$ has no oriented cycle.
	Fix $i\in Q_0$.  
	Let $\rho$ be the row vector whose entries are 
	$V_\gamma e^{\left(t(\gamma)\right)}$,
	where $\gamma$ is any path whose head $h(\gamma)$ is $i$ (including the trivial path), $t(\gamma)$ denotes its tail, and $V_{\gamma} \in \Hom(\C^{d_{t(\gamma)}},\C^{d_{h(\gamma)}})$ is the representing matrix of $\gamma$.  This defines a map 
	$$V_{\gamma} e^{\left(t(\gamma)\right)}: R_{\vec{n},\vec{d}} \to \Hom(\C^{n_{t(\gamma)}},\C^{d_i}). $$  
	Take
	$$\rho\rho^* = \sum_{h(\gamma)=i} \left(V_{\gamma} e^{\left(t(\gamma)\right)}\right)\left(V_{\gamma} e^{\left(t(\gamma)\right)}\right)^*: R_{\vec{n},\vec{d}} \to \End(\C^{d_i}).$$
	
	Then $H=(\rho\rho^*)^{-1}$ is $\GL_{\vec{d}}$-equivariant, and it descends to a metric on $\cV_i$ over $\cM_{\vec{n},\vec{d}}$.
\end{theorem}

\begin{proof}
	$(\rho\rho^*)^{-1}$ is $\GL_{\vec{d}}$-equivariant:
	\begin{align*}
	H_{\left(g_{h(a)} V_a g^{-1}_{t(a)}, g_j e^{(j)}\right)_{a\in Q_1, j\in Q_0}}
	&=\left(\sum_\gamma \left(g_{h(\gamma)}V_{\gamma} e^{\left(t(\gamma)\right)}\right)\left(g_{h(\gamma)}V_{\gamma} e^{\left(t(\gamma)\right)}\right)^*\right)^{-1}\\
	&= (g^*_i)^{-1} \left(\sum_\gamma \left(V_{\gamma} e^{\left(t(\gamma)\right)}\right)\left(V_{\gamma} e^{\left(t(\gamma)\right)}\right)^*\right)^{-1}g_i^{-1}
	\end{align*}
	for all $g\in \GL_{\vec{d}}$.
	By Lemma \ref{lem:equiv}, it descends to the bundle $\cV_i$ of the quotient.	
	
	Then we prove that the matrix-valued function $H_{(V_a,e^{(j)})_{a\in Q_1, j\in Q_0}}=(\rho \rho^*)^{-1}$ defines a positive-definite metric on the trivial bundle $\underline{\C^{d_i}}$ \emph{over the moment map level $\mu^{-1}(I_{\vec{d}})$} (rather than the whole $R_{\vec{n},\vec{d}}$).
	We prove by induction on the vertices that $\rho \rho^* = I_{d_i} + B$ where $B$ is a semi-positive-definite Hermitian matrix, and hence $\rho \rho^*$ is positive definite (and so does $(\rho \rho^*)^{-1}$).
	
	Since the quiver does not have oriented cycle, $Q_0$ can be ordered such that $i<j$ whenever there is an arrow $i\to j$.  Let $i_0$ be the minimal vertex.
	
	At $i_0$, there is no incoming arrow (other than the framing), and the moment-map equation reads
	$$ e^{(i_0)} (e^{(i_0)})^* = I_{d_{i_0}} + \sum_{t(a)=i_0} V_a^* V_a. $$
	$\sum_{t(a)=i_0} V_a^* V_a$ is semi-positive definite:
	$v^* \cdot V_a^* V_a \cdot v = \|V_a \cdot v\|^2\geq 0$ for any column vector $v$.  Thus the statement is true for $\rho \rho^* = e_{i_0} e_{i_0}^*$.
	
	Suppose the statement is true for all vertices less than $i \in Q_0$.  At $i$, the moment-map equation is
	$$ e^{(i)} (e^{(i)})^* = I_{d_i} + \sum_{t(a)=i} V_a^* V_a - \sum_{h(a')=i} V_{a'} V_{a'}^*. $$
	Then
	\begin{align*}
	\rho \rho^* &= e^{(i)} (e^{(i)}) ^* + \sum_{h(a)=i} V_a \rho_{(t(a))}\rho_{(t(a))}^* V_{a}^* \\
	&= I_{d_i} + \sum_{t(a)=i} V_a^* V_a - \sum_{h(a')=i} V_{a'} V_{a'}^* + \sum_{h(a')=i} V_{a'} (I_{d_{t(a')}} + B_{t(a')}) V_{a'}^*\\
	&=I_{d_i} + \sum_{t(a)=i} V_a^* V_a + \sum_{h(a')=i} V_{a'} B_{t(a')} V_{a'}^*
	\end{align*}
	where $\rho_{(t(a))}=\left(V_{\gamma} e^{\left(t(\gamma)\right)}\right)_{h(\gamma)=t(a)}$, which by inductive assumption can be written as $I_{d_{t(a')}} + B_{t(a')}$ where $B_{t(a')}$ is semi-positive definite.  The matrices $V_a^* V_a$ and $V_{a'} B_{t(a')} V_{a'}^*$ are semi-positive definite:
	$$v^* V_{a'} B_{t(a')} V_{a'}^* v = (V_{a'}^* v)^* B_{t(a')} (V_{a'}^* v) \geq 0$$
	for all $v$.  This proves the statement for the vertex $i$.
\end{proof}

The expression $(\rho\rho^*)^{-1}$ can be understood as follows.
$\rho^*$ embeds the dual $V_i^{*}$ into the dual frame which is a trivial bundle equipped with the standard metric.  This gives an induced metric on $V_{i}^{*}$, which is $\rho\rho^*$ written in matrix form.  Taking the dual, we get the metric $H_i=(\rho\rho^*)^{-1}$ on $V_{i}$.

By construction, the metrics on the dual $\cV_i^*$ (still denoted as $H_i$) have the following nice property.  Inductively, it gives nice expressions of $H_i$ in terms of \emph{holomorphic coordinates}.

\begin{prop} \label{prop:relate-H_i}
	Suppose $Q$ has no oriented cycle.
	For $v,w\in (\cV_i)^*$, 
	$$H_i(v,w) = H_0((e^{(i)})^*(v),(e^{(i)})^*(w)) + \sum_{h(a)=i} H_{t(a)}(a^*(v),a^*(w))$$
	where $H_0$ denotes the trivial metric on the trivial bundle, and $e^{(i)}, a$ are denoting the holomorphic bundle maps corresponding to the framing and arrow maps respectively.
\end{prop}

\begin{proof}
	The metric on $\cV_i^*$ is given by the matrix $\rho \rho^*$.  Then the above equation follows from
	$$\rho \rho^* = e^{(i)} (e^{(i)}) ^* + \sum_{h(a)=i} V_a \rho_{(t(a))}\rho_{(t(a))}^* V_{a}^*.$$
\end{proof}

\begin{remark}
	As we have seen, the Grassmannian $\Gr(n,k)$ can be understood as the framed moduli for the quiver which has one vertex and no arrow.  The matrix $e \in \Hom(\C^n,\C^k)$ is the framing map.  Then the moment map equation implies
	$$\rho\rho^* = ee^* = I_k$$ in the above proposition.  This is the standard metric on the trivial bundle $\underline{\C^k}$ that we have used in the last subsection.  In particular $\rho\rho^* = (\rho\rho^*)^{-1}$ in this case.  But this is not true for other quivers.
\end{remark}

\begin{remark}
	Note that the above becomes an infinite sum if the quiver has oriented cycles.  The $\GL_{\vec{d}}$-equivariance still holds.  We should restrict to the open subset of $R^s_{n,d}$ that $(\rho\rho^*)^{-1}$ is convergent.  In the next subsection, we will prove that the same expression defines a metric for any given quiver.
\end{remark}

The GIT description will be important to the proof of Theorem \ref{thm:same-metric}.

There is a residual symmetry $U_{\vec{n}} = \prod_{i\in Q_0} U({n_i})$ acting on $\cM_{\vec{n},\vec{d}}$.  Actually, there is a bigger symmetry by the non-compact group $\GL(\hat{W})$ \cite{Reineke}.  $U_{\vec{n}}$ is considered here since this is the symmetry of the metric $H$ on $\cV_i$ as we shall see.

\begin{defn} \label{def:symmetry}
	The right residual action of $U_{\vec{n}}$ on $\cM$ is defined as
	$$ \left[(V_a,e^{(j)})_{a\in Q_1, j\in Q_0}\right] \cdot g = \left[(V_a,e^{(j)}\circ g_j)_{a\in Q_1, j\in Q_0}\right]$$
	for $g = (g_j \in U(d_i))_{j\in Q_0} \in U_{\vec{n}}$.
\end{defn}
Since the above commutes with the left action of $\GL_{\vec{d}}$, the action is well-defined on $\cM$.

\begin{lemma}
	There is a canonical lift of the action of $U_{\vec{n}}$ on $\cM_{\vec{n},\vec{d}}$ to the universal bundle $\cV_i$, so that the bundle map $\cV_{i}\rightarrow \cM_{\vec{n},\vec{d}}$ is equivariant.
\end{lemma}
\begin{proof}
	$\cV_i$ is the $\GL_{\vec{d}}$-quotient of the trivial bundle $R^s_{n,d}\times V_{i}$.  $U_{\vec{n}}$ acts on this by acting on the component $V_i$ trivially.  This action commutes with the left action of $\GL_{\vec{d}}$ on $R^s_{n,d}\times V_{i}$, and hence descends to act on $\cV_i$.
\end{proof}

\begin{lemma} \label{lem:metric-equiv}
	The metric defined in Theorem \ref{thm:metric} are $U_{\vec{n}}$-invariant.
\end{lemma}

\begin{proof}
	For any $g \in U_{\vec{n}}$, since $g_j g_j^* = I_{n_j}$ for any $j\in Q_0$,
	\begin{align*}
	H_{\left(V_a, e^{(j)}\right)_{a\in Q_1, j\in Q_0} \cdot g}
	&=\left(\sum_\gamma \left(V_{\gamma} e^{\left(t(\gamma)\right) }\cdot g_{t(\gamma)}\right)\left(V_{\gamma} e^{\left(t(\gamma) \right)}\cdot g_{t(\gamma)}\right)^*\right)^{-1}\\
	&= \left(\sum_\gamma \left(V_{\gamma} e^{\left(t(\gamma)\right)}\right)\left(V_{\gamma} e^{\left(t(\gamma)\right)}\right)^*\right)^{-1} = H_{\left(V_a, e^{(j)}\right)_{a\in Q_1, j\in Q_0}}.
	\end{align*}
\end{proof}

Recall from Theorem \ref{thm:Reineke} that $\cM_{\vec{n},\vec{d}}$ is the total space of an iterated Grassmann bundle $M^{(N)}\stackrel{p_N}{\to} M^{(N-1)} \stackrel{p_{N-1}}{\to} \ldots \stackrel{p_{2}}{\to} M^{(1)} \stackrel{p_{1}}{\to} \pt$.  Moreover, $\cV_i$ is the pull-back of the tautological bundle $S_i$ of the Grassmann bundle $M^{(i)} = \Gr_{M^{(i)}}(\underline{\C^{n_i}}\oplus \bigoplus_{j\to i}p_i^*\ldots p_{j+1}^*(S_j),d_i)\to M^{(i-1)}$.  The tautological bundle of the Grassmannian is equipped with a standard metric as illustrated in Section \ref{sec:Herm}.  Inductively, $\cV_i$ is also equipped with a pull-back metric.  We show that this equals to the metric $H$ we defined by an explicit formula.

\begin{theorem} \label{thm:same-metric}
	For all $i\in Q_0$,
	the metric $H=\left(\sum_{h(\gamma)=i} \left(V_{\gamma} e^{\left(t(\gamma)\right)}\right)\left(V_{\gamma} e^{\left(t(\gamma)\right)}\right)^*\right)^{-1}$ equals to the metric on $\cV_i$ constructed from the iterated Grassmann bundle.
\end{theorem}

\begin{proof}
	As the proof of Theorem \ref{thm:metric}, we do induction on the vertices, which are totally ordered such that $i<j$ whenever there is an arrow $i\to j$.
	
	First, we have the GIT fiber-bundle map $\cM \to M^{(j)}$, where $M^{(j)}$ is the framed moduli of the quiver $Q^{(j)}$ (which is obtained by removing all vertices $k>j$ and the corresponding arrows from $Q$).  This map $(V,e) \mapsto (V',e')$ is simply forgetting all the irrelevant arrow maps and frame maps that are not supported on the subquiver $Q^{(j)}$.  Note that the stability condition is preserved: any subrepresentation $R'\subset V'$ can be extended to a subrepresentation $R \subset V$ by assigning the whole $V_k$ to the additional vertices $k>j$  (and the arrow maps just come from restriction).  If $\mathrm{Im}(e') \subset R'$, then $\mathrm{Im}(e) \subset R$.
	
	Note that here we use the GIT description instead of symplectic reduction since $(V',e')$ no longer satisfies the moment-map equation in defining $M^{(j)}$ (even when $(V,e)$ satisfies the moment-map equation for $\cM$).
	
	We start with the minimal vertex $i_0$.  The quiver $Q^{(i_0)}$ is simply a single vertex, and the corresponding framed moduli is $ M^{(i_0)} = \Gr(n_{i_0},d_{i_0}).$
	The universal bundle $\cV_{i_0}^{Q^{(i_0)}}$ is the tautological bundle of $M^{(i_0)} = \Gr(n_{i_0},d_{i_0})$.  From the last subsection, the standard metric of $\cV_{i_0}^{Q^{(i_0)}}$ is descended from $I_{d_{i_0}}$, which is exactly $(e^{Q^{(i_0)}}(e^{Q^{(i_0)}})^*)^{-1}$ by the moment map equation for $Q^{(i_0)}$.  The statement is trivial in this case.
	
	Now consider the vertex $i$.  Denote the vertex right before $i$ by $i-1$.  For the quiver $Q^{(i-1)}$, assume that the two metrics on the universal bundle $\cV^{Q^{(i-1)}}_j$ agree for every $j\in Q^{(i-1)}_0$.  We have the bundle map $\pi:M^{(i)} \to M^{(i-1)}$, and $\cV^{Q^{(i)}}_j = \pi^* \cV^{Q^{(i-1)}}_j$ for all $j<i$.  Moreover, the metric on $\cV^{Q^{(i)}}_j$ is pull-back from $\cV^{Q^{(i-1)}}_j$, which equals to $\left(\sum_{h(\gamma)=j} \left(V_{\gamma} e^{\left(t(\gamma)\right)}\right)\left(V_{\gamma} e^{\left(t(\gamma)\right)}\right)^*\right)^{-1}$ by inductive assumption.  The pull-back map does not change the arrow and framing maps of $Q^{(i-1)}$.  This proves the statement for $S^{Q^{(i)}}_j$ for $j<i$.
	
	Consider $\cV^{Q^{(i)}}_i$, which is the tautological bundle associated to the Grassmannian bundle over $M^{(i-1)}$ parametrizing quotients of $$(e^{(i)},(V_a)_{h(a)=i}): \left.\left(\underline{\C^{d_i}} \oplus \left(\bigoplus_{h(a)=i} \cV^{Q^{(i-1)}}_{t(a)}\right)\right)\right|_{x\in M^{(i-1)}} \to \C^{d_i}.$$
	The metric on $(\cV^{Q^{(i)}}_i)^*$ is induced by the embedding $(e^{(i)},(V_a)_{h(a)=i})^*$ to $\left(\underline{\C^{d_i}} \oplus \left(\bigoplus_{h(a)=i} \cV_{t(a)}\right)\right)^*$, whose metric is given by $$I_{d_i} \oplus \bigoplus_{h(a)=i} \left(\sum_{h(\gamma)=t(a)} \left(V_{\gamma} e^{\left(t(\gamma)\right)}\right)\left(V_{\gamma} e^{\left(t(\gamma)\right)}\right)^*\right)$$
	by inductive assumption.  Thus the induced metric on $(\cV^{Q^{(i)}}_i)^*$ is
	\begin{align*}
	&(e^{(i)},(V_a)_{h(a)=i}) \cdot \left(I_{d_i} \oplus \bigoplus_{h(a)=i} \left(\sum_{h(\gamma)=t(a)} \left(V_{\gamma} e^{\left(t(\gamma)\right)}\right)\left(V_{\gamma} e^{\left(t(\gamma)\right)}\right)^*\right)\right) \cdot (e^{(i)},(V_a)_{h(a)=i})^*\\
	&=e^{(i)}(e^{(i)})^* + \sum_{\substack{h(a)=i\\ h(\gamma)=t(a)}} V_a \cdot \left(V_{\gamma} e^{\left(t(\gamma)\right)}\right)\left(V_{\gamma} e^{\left(t(\gamma)\right)}\right)^* \cdot V_a^* \\
	&=\sum_{h(\gamma)=i} \left(V_{\gamma} e^{\left(t(\gamma)\right)}\right)\left(V_{\gamma} e^{\left(t(\gamma)\right)}\right)^*.
	\end{align*}
	Taking reciprocal gives the metric on $\cV^{Q^{(i)}}_i$.  This proves the metric has the given expression.
\end{proof}

\begin{theorem} \label{thm:Ricci}
	The Ricci curvature of the metric on $\bigotimes_{i\in Q_0} \cV_i$ given in Theorem \ref{thm:metric} defines a K\"ahler metric on $\cM$.
\end{theorem}

\begin{proof}
	As in Theorem \ref{thm:metric}, denote $\rho =\rho ^{\left(i\right)}=\left(V_\gamma e^{\left(t(\gamma)\right)}\right)_{\gamma: h(\gamma)=i}$ which is a matrix-valued function on the vector space $R_{\vec{n},\vec{d}}$.  At each point of $R_{\vec{n},\vec{d}}$, $\rho$ is a linear map from 
	\begin{equation}
	\widehat{W}_i := \bigoplus_{\gamma: h(\gamma)=i} \C^{n_{t(\gamma)}}
	\label{eq:W}
	\end{equation}
	to $V_i$.
	The Ricci curvature of the metric $(\rho\rho^*)^{-1}$ is given by $i \partial \overline{\partial }\log \det \rho \rho ^{\mathrm{*}}$.  We have
	\begin{align*}
	\partial \overline{\partial }\log \det \rho \rho ^{\mathrm{*}}
	&=\partial \left(\mathrm{tr~ }\left(\left(\rho \rho ^{\mathrm{*}}\right)^{-1}\overline{\partial }\left(\rho \rho ^{\mathrm{*}}\right)\right)\right)\\
	&=\mathrm{tr}\left(\partial \left(\left(\rho \rho ^{\mathrm{*}}\right)^{-1}\rho \left(\partial \rho \right)^{\mathrm{*}}\right)\right)\\
	&=\mathrm{tr~ }\left(\left(\rho \rho ^{\mathrm{*}}\right)^{-1}\partial \rho \left(\partial \rho \right)^{\mathrm{*}}+\left(\partial \left(\rho \rho ^{\mathrm{*}}\right)^{-1}\right)\rho \left(\partial \rho \right)^{\mathrm{*}}\right)\\
	&=\mathrm{tr~ }\left(\left(\partial \rho \right)^{\mathrm{*}}\left(\rho \rho ^{\mathrm{*}}\right)^{-1}\partial \rho \right)-\mathrm{tr~ }\left(\left(\rho \rho ^{\mathrm{*}}\right)^{-1}\left(\partial \left(\rho \rho ^{\mathrm{*}}\right)\right)\left(\rho \rho ^{\mathrm{*}}\right)^{-1}\rho \left(\partial \rho \right)^{\mathrm{*}}\right)\\
	&=\mathrm{tr~ }\left(\left(\partial \rho \right)^{\mathrm{*}}\left(\rho \rho ^{\mathrm{*}}\right)^{-1}\partial \rho \right)-\mathrm{tr~ }\left(\left(\rho \rho ^{\mathrm{*}}\right)^{-1}\rho \left(\partial \rho \right)^{\mathrm{*}}\left(\rho \rho ^{\mathrm{*}}\right)^{-1}\left(\left(\partial \rho \right)\rho ^{\mathrm{*}}\right)\right)\\
	&=\mathrm{tr~ }\left(\left(\partial \rho \right)^{\mathrm{*}}\left(\rho \rho ^{\mathrm{*}}\right)^{-1}\partial \rho \right)-\mathrm{tr~ }\left(\left(\partial \rho \cdot \left(\rho ^{\mathrm{*}}\left(\rho \rho ^{\mathrm{*}}\right)^{-\frac{1}{2}}\right)\right)^{\mathrm{*}}\left(\rho \rho ^{\mathrm{*}}\right)^{-1}\left(\partial \rho \cdot \left(\rho ^{\mathrm{*}}\left(\rho \rho ^{\mathrm{*}}\right)^{-\frac{1}{2}}\right)\right)\right)
	\end{align*}
	where $\overline{\partial }\rho =0$ since the matrix $\rho $ has polynomial entries in holomorphic coordinates.
	
	We can take the singular value decomposition
	\begin{equation*}
	\rho =U \cdot \left(\text{diag}\left(\lambda _{1},\ldots ,\lambda _{d_i}\right) \,\,\, 0\right) \cdot V^{\mathrm{*}}
	\end{equation*}
	where $U\in U\left(d_i\right),~ V\in U\left(\dim \widehat{W}_{i}\right)$, and $\lambda _{i}>0$. ($\lambda _{i}\neq 0$ since $\rho$ is surjective.)
	Then
	\begin{align*}
	\rho \rho ^{\mathrm{*}}&=U\left(\text{diag}\left(\lambda _{1}^{2},\ldots ,\lambda _{d_i}^{2}\right)\right)U^{\mathrm{*}}. \\
	\rho ^{\mathrm{*}}\left(\rho \rho ^{\mathrm{*}}\right)^{-\frac{1}{2}}&=V\left(\begin{array}{c}
	\text{diag}\left(\lambda _{1},\ldots ,\lambda _{\alpha \left(i\right)}\right)\\
	0
	\end{array}\right)\left(\text{diag}\left(\lambda _{1}^{-1},\ldots ,\lambda _{\alpha \left(i\right)}^{-1}\right)\right)U^{\mathrm{*}}=V\left(\begin{array}{c}
	I_{\alpha \left(i\right)}\\
	0
	\end{array}\right)U^{\mathrm{*}}. 
	\end{align*}
	In other words, $\rho ^{\mathrm{*}}=\left(\rho ^{\mathrm{*}}\left(\rho \rho ^{\mathrm{*}}\right)^{-\frac{1}{2}}\right)\left(\rho \rho ^{\mathrm{*}}\right)^{\frac{1}{2}}$ is decomposed into the rescaling $\rho \rho ^{\mathrm{*}}$ and the orthogonal embedding $\left(\rho ^{\mathrm{*}}\left(\rho \rho ^{\mathrm{*}}\right)^{-\frac{1}{2}}\right)$ to $\mathrm{Im} \rho ^{\mathrm{*}}\subset \widehat{W}_{i}$.
	
	Now take a vector $v \in T^{1,0} R_{\alpha,d} \cong TR_{\alpha,d}$, and evaluate the above two-form by $(v,\bar{v})$.
	The first term $\mathrm{tr~ }\left(\left(\partial _{v}\rho \right)^{\mathrm{*}}\left(\rho \rho ^{\mathrm{*}}\right)^{-1}\partial _{v}\rho \right)$ is the square norm of the linear map 
	\begin{equation*}
	\partial _{v}\rho \colon \left(\widehat{W}_{i},h_{\mathrm{std}}\right)\rightarrow \left(V_{i},~ h_{{\left(\rho \rho ^{\mathrm{*}}\right)^{-1}}}\right).
	\end{equation*}
	Namely we take the standard basis in $\widehat{W}_{i}$ (which is orthonormal under the standard metric $h_{\mathrm{std}}$), map it to $V_{i}$ by $\partial _{v}\rho $, and take the sum of their square norms with respect to the metric $h_{{\left(\rho \rho ^{\mathrm{*}}\right)^{-1}}}$.
	
	The second term
	\begin{equation*}
	\mathrm{tr~ }\left(\left(\partial _{v}\rho \cdot \left(\rho ^{\mathrm{*}}\left(\rho \rho ^{\mathrm{*}}\right)^{-\frac{1}{2}}\right)\right)^{\mathrm{*}}\left(\rho \rho ^{\mathrm{*}}\right)^{-1}\left(\partial _{v}\rho \cdot \left(\rho ^{\mathrm{*}}\left(\rho \rho ^{\mathrm{*}}\right)^{-\frac{1}{2}}\right)\right)\right)
	\end{equation*}	
	is the square norm of the following component $\left(\partial _{v}\rho \right)_{1}$ of $\partial _{v}\rho $. Namely, we decompose 
	$$\widehat{W}_{i}=\left(\mathrm{Im}~ \rho ^{\mathrm{*}}\right)\oplus \left(\mathrm{Im}~ \rho ^{\mathrm{*}}\right)^{\bot }$$
	and write $\partial _{v}\rho =\left(\left(\partial _{v}\rho \right)_{1},\left(\partial _{v}\rho \right)_{2}\right)$ where $\left(\partial _{v}\rho \right)_{1}\colon \mathrm{Im}~ \rho ^{\mathrm{*}}\rightarrow V_{i}$ and $\left(\partial _{v}\rho \right)_{2}\colon \left(\mathrm{Im}~ \rho ^{\mathrm{*}}\right)^{\bot }\rightarrow V_{i}$.
	
	We have 
	\begin{equation}
	\partial _{v}\overline{\partial _{v}}\log \det \rho \rho ^{\mathrm{*}} = \left\| \partial _{v}\rho \right\|_H ^{2}-\left\| \left(\partial _{v}\rho \right)_{1}\right\|_H ^{2}=\left\| \left(\partial _{v}\rho \right)_{2}\right\|_H ^{2}\geq 0
	\label{eq:Ricci}
	\end{equation}
	for all $v\in T^{1,0}$.  ($H$ stands for the metric $(\rho \rho ^{\mathrm{*}})^{-1}$.)
	This proves that the Ricci curvature of the metric for each $i$ is semi-positive definite.
	
	Now consider 
	$$H_T = \sum_{i\in Q_0} \partial \overline{\partial }\log \det \rho^{(i)} \left(\rho^{(i)}\right)^{\mathrm{*}}.$$
	Suppose it is zero when evaluated at $(v,\bar{v})$.  Then each individual term equals to zero.  This forces $\left(\partial _{v}\rho^{(i)} \right)_{2}=0$ for all $i$, that is, image of $\left(\partial _{v}\rho^{(i)} \right)^{\mathrm{*}} = \partial _{v} (\rho^{(i)})^*$  sits in the image of $(\rho^{(i)})^{\mathrm{*}}$. This exactly means $v$ descends to the zero tangent vector in the quotient $\cM$: $v$ does not alter the subspaces given by $(\rho^{(i)})^*: V_i \to \hat{W}_i$ for all $i$.  By the identification of $\cM$ as a quiver Grassmannian \cite{Reineke}, it means $v$ does not change the position of the point $((\rho^{(i)})^*: i\in Q_0)$ in the quiver Grassmannian, and hence must be the zero tangent vector.  This proves the above expression is positive definite.  
\end{proof}

By Equation \eqref{eq:Ricci}, the metric on $T\cM$ produced from the Ricci curvature of $\cV_i$ is
$$ H_{T}(v,v) = \sum_{i\in Q_0} \|(\partial_v \rho^{(i)})_2 \|_{H_i}^2. $$
We have the tautological exact sequence of vector bundles over $\cM$:
$$ 0 \to \bigoplus_{i\in Q_0} \End(\cV_i) \to \bigoplus_{a\in Q_1} \Hom(\cV_{t_a},\cV_{h_a}) \oplus \bigoplus_{i\in Q_0} \cV_i^{n_i} \to T\cM \to 0  $$
where the second arrow is given by sending $X|_{[\phi,e]\in \cM} \in \bigoplus_{i\in Q_0} \End(\cV_i)$ to $$\left((X_{h_a}\phi_a - \phi_a X_{t_a})_{a\in Q_1}, (X_ie^{(i)})_{i\in Q_0} \right)$$ 
(which is the derivative of the action of $\GL_{\vec{d}}$), and $T\cM$ is obtained as the quotient bundle (of the middle one by the first one).
$H(v,v)$ can be defined for $v \in \bigoplus_{a\in Q_1} \Hom(\cV_{t_a},\cV_{h_a}) \oplus \bigoplus_{i\in Q_0} \cV_i^{n_i}$.  $H_{TM}$ is zero on $\bigoplus_{i\in Q_0} \End(\cV_i)$: the action of $\GL(d_j,\C)$ does not change $\rho^{(i)}$ for $j\not=i$; for $X\in \End(\cV_i)$, $(X\cdot \rho^{(i)} \cdot w)^* v = w^* (\rho^{(i)})^*(X^*\cdot v)=0$ for all $v \in \cV_i, w \in (\Image (\rho^{(i)})^*)^\perp$, and so $(\partial_X \rho^{(i)})_2 = 0$.

\subsection{Quiver with oriented cycles} \label{sec:cycle}
When the quiver $Q$ has an oriented cycle, the framed moduli $\cM_{\vec{n},\vec{d}}$ is no longer projective.  Examples of such quivers were studied algebraically by \cite{Fedotov, Engel-Reineke} along the line of Reineke.

On the other hand, the metric given in Theorem \ref{thm:metric} still makes sense for quiver with oriented cycles, as long as we stay in the domain of convergence and prove that it is positive-definite.  Below we will prove this for any given quiver.

Denote the moment-map level by $R^{\mu=I}_{n,d} = \{\mu = 1\}\subset R_{\vec{n},\vec{d}}$.  Let $\|A\|=\sup_{\|v\|=1} \|A\cdot v\|$ be the operator norm of a matrix $A$.  
We take the following open subset of $R^{\mu=I}_{n,d}$.  

\begin{defn}
	Define
	$$R_{\vec{n},\vec{d}}^{\mu=I,\circ} := \{(V,e) \in R^{\mu=I}_{n,d}: \|V(\gamma)\| < 1 \textrm{ for every oriented cycle } \gamma \}$$
	where $V(\gamma) = V(a_k)\ldots V(a_1)$ for $\gamma = a_k \ldots a_1$.
	$$ R_{\vec{n},\vec{d}}^{\circ} := \GL_{\vec{d}} \cdot R_{\vec{n},\vec{d}}^{\mu=I,\circ}. $$
	and
	$$ \cM_{\vec{n},\vec{d}}^\circ := R_{\vec{n},\vec{d}}^{\mu=I,\circ}/U_{\vec{d}} = R_{\vec{n},\vec{d}}^{\circ} / \GL_{\vec{d}}. $$
\end{defn}
The above definition of $\cM_{\vec{n},\vec{d}}^\circ$ makes sense because of the following.
\begin{lemma}
	$R_{\vec{n},\vec{d}}^{\mu=I,\circ}$ is invariant under $U_{\vec{d}}$.
\end{lemma}

\begin{proof}
	For every oriented cycle $\gamma$ at $i \in Q_0$, $(g\cdot V)(\gamma) = g_i V(\gamma) g_i^{-1}$, and hence the condition $\|V(\gamma)\| < 1$ is respected for $g\in U_{\vec{d}}$.
\end{proof}

The main theorem in this section is the following.

\begin{theorem} \label{thm:metric-with-cycles}
	Let $Q$ be an arbitrary quiver.  As in Theorem \ref{thm:metric}, for each $i \in Q_0$, set
	$$H_i = (\rho_i\rho_i^*)^{-1} = \left(\sum_{h(\gamma)=i} \left(V_{\gamma} e^{\left(t(\gamma)\right)}\right)\left(V_{\gamma} e^{\left(t(\gamma)\right)}\right)^*\right)^{-1}$$
	which is an infinite sum, whose terms are ordered by the length of the path $\gamma$.  (There are just finitely many paths under each fixed length.)  This gives a convergent function $H_i: R_{\vec{n},\vec{d}}^\circ \to \End(\C^{d_i})$.
	$H_i$ is $\GL_{\vec{d}}$-equivariant, and it descends to a metric on $\cV_i$ over $\cM_{\vec{n},\vec{d}}^\circ$.
\end{theorem}

We break into several steps to prove the above theorem.  First, consider the convergence.

\begin{lemma} \label{lem:conv}
	$\rho\rho^*$ is absolutely convergent over  $R_{\vec{n},\vec{d}}^{\mu=I,\circ}$.  Hence $(\rho\rho^*)^{-1}$ is well-defined and $\GL_{\vec{d}}$-equivariant on $R_{\vec{n},\vec{d}}^{\circ}$.
\end{lemma}
\begin{proof} 
	For $(V,e) \in R_{\vec{n},\vec{d}}^{\mu=I,\circ}$, we consider the expression
	$$\sum_{h(\gamma)=i} \left\|V_{\gamma} e^{\left(t(\gamma)\right)}\right\|\left\|\left(V_{\gamma} e^{\left(t(\gamma)\right)}\right)^* \right\| \leq \sum_{h(\gamma)=i} \left\|V_{\gamma} e^{\left(t(\gamma)\right)}\right\|^2.$$  There are only finitely many paths $\gamma_1,\ldots,\gamma_k$ with $h(\gamma_l)=i$ which do not contain any oriented cycle.  Any other path (with $h(\gamma)=i$) can be written as concatenation of one of these $\gamma_l$ and some oriented cycles at some vertices.  Thus 
	$$\sum_{h(\gamma)=i} \left\|V_{\gamma} e^{\left(t(\gamma)\right)}\right\|^2 \leq \sum_{l=1}^k \|\gamma_l e^{(t(\gamma_l))}\|^2 \sum_{p=0}^\infty (1-\epsilon)^p = \sum_{l=1}^k \|\gamma_l e^{(t(\gamma_l))}\|^2 \sum_{p=0}^\infty (1-\epsilon)^p = \sum_{l=1}^k \frac{\|\gamma_l e^{(t(\gamma_l))}\|^2}{\epsilon} < \infty$$
	where given $V$, there is a fixed $\epsilon \in (0,1)$ such that $\|V_\gamma\|^2 < 1-\epsilon$ for all oriented cycles $\gamma$.  Hence $\rho\rho^*$ is absolutely convergent for every $(V,e)\in R_{\vec{n},\vec{d}}^{\mu=I,\circ}$.
	
	Every element in $R_{\vec{n},\vec{d}}^{\circ}$ can be written as $g \cdot (V,e)$ for $g \in \GL_{\vec{d}}$ and $(V,e) \in R_{\vec{n},\vec{d}}^{\mu=I,\circ}$.  $(\rho\rho^*)^{-1}|_{g \cdot (V,e)} = (g_i^*)^{-1} (\rho\rho^*)^{-1}|_{(V,e)} g_i^{-1}$ where $(\rho\rho^*)^{-1}|_{(V,e)}$ is convergent.
\end{proof}

It remains to prove positive definiteness of $H_i$.  First we consider the following specific quiver which is simply a single oriented cycle.

\begin{lemma} \label{lem:single-cycle}
	If $Q$ is a single oriented cycle with $N$ vertices, then for each vertex $i$, $\rho_i\rho_i^*$ is positive definite.  (In particular, when $N=1$, $Q$ consists of one vertex and a self loop.)
\end{lemma}

\begin{proof}
	By symmetry, we just need to prove for $i=N$.
	Let $l = a_N \ldots a_1$, where $a_k$ is the arrow $(k-1) \to k$.
	$$\rho_N\rho_N^* = \sum_{k=1}^N \left(a_N \ldots a_{k+1} e_k e_k^* a_{k+1}^* \ldots a_N^* + \sum_{p>0} l^p a_N \ldots a_{k+1} e_k e_k^* a_{k+1}^* \ldots a_N^* (l^p)^* \right).$$
	The moment map equation at the vertex $k$ is $e_ke_k^* = I - a_{k}a_k^* + a_{k+1}^*a_{k+1}$.  Then the first term gives
	\begin{align*}
	&\sum_{k=1}^N a_N \ldots a_{k+1} (I - a_{k}a_k^* + a_{k+1}^*a_{k+1}) a_{k+1}^* \ldots a_N^* \\
	=& \sum_{k=1}^N a_N \ldots a_{k+1} a_{k+1}^* \ldots a_N^* - \sum_{k=1}^N a_N \ldots a_{k+1} a_{k}a_k^* a_{k+1}^* \ldots a_N^* + \sum_{k=1}^N a_N \ldots a_{k+1} a_{k+1}^*a_{k+1} a_{k+1}^* \ldots a_N^*\\
	=& I - ll^* + \sum_{k=1}^N a_N \ldots a_{k+1} a_{k+1}^*a_{k+1} a_{k+1}^* \ldots a_N^*.
	\end{align*}
	Similarly, the second term gives
	\begin{align*}
	&\sum_{k=1}^N \sum_{p>0} l^p a_N \ldots a_{k+1} (I - a_{k}a_k^* + a_{k+1}^*a_{k+1}) a_{k+1}^* \ldots a_N^* (l^p)^*\\
	=& \sum_{p>0}l^p (l^p)^* - \sum_{p>0}l^pll^*(l^p)^* + \sum_{p>0}\sum_{k=1}^N l^p a_N \ldots a_{k+1} a_{k+1}^*a_{k+1} a_{k+1}^* \ldots a_N^* (l^p)^*\\
	=& ll^* + \sum_{p>0}\sum_{k=1}^N l^p a_N \ldots a_{k+1} a_{k+1}^*a_{k+1} a_{k+1}^* \ldots a_N^* (l^p)^*.
	\end{align*}
	Combining the two terms,
	$$ \rho_N\rho_N^* = I + \sum_{p\geq 0}\sum_{k=1}^N l^p a_N \ldots a_{k+1} a_{k+1}^*a_{k+1} a_{k+1}^* \ldots a_N^* (l^p)^*$$
	and the second term is semi-positive-definite.  Hence $\rho_N\rho_N^*$ is positive definite.
\end{proof}

The following is the key lemma to prove positive-definiteness for a general quiver.

\begin{lemma} \label{lem:add-cycle}
	Suppose $Q$ has the property that for every $i \in Q_0$, restricted to the intersection of the moment map locus and $R^{Q,\circ}_{n,d}$, $\rho^Q_i(\rho^Q_i)^* = I + B_i$ for some semi-positive definite matrix $B_i$.  Let $Q'$ be obtained by concatenating to $Q$ a chain $x_1 \stackrel{a_1'}{\to} \ldots \stackrel{a_{k-1}'}{\to} x_k$ where $x_1$ and $x_k$ are certain vertices in $Q$.  ($x_1$ can be equal to $x_k$, meaning what we have added is an oriented cycle.  When $k=1$, we have added a loop; $a'_{0}=a'_1$.  When $k=2$, there is no intermediate vertex in the chain.)  Then $Q'$ has the same property.  Namely, for every $i \in Q'_0$, restricted to the intersection of the moment map locus and $R^{Q',\circ}_{n,d}$, $\rho^{Q'}_i(\rho^{Q'}_i)^* = I + B_i'$ for some semi-positive-definite matrix $B_i'$.
\end{lemma}

\begin{proof}	
	First, consider the case that the given vertex $i\in Q_0'$ belongs to $Q_0$.  For the original quiver $Q$, $\rho^Q_i (\rho^Q_i)^* = I + B_i$ where $B_i$ is semi-positive definite.  After concatenating the chain, the terms in $\rho^Q_i$,
	$$ \gamma^{x_1 \to i}_Q e^{(1)} (e^{(1)})^* (\gamma^{x_1 \to i}_Q)^*   \textrm{ and } \gamma^{x_k \to i}_Q e^{(k)} (e^{(k)})^* (\gamma^{x_k \to i}_Q)^*$$
	where $\gamma^{x_1 \to i}_Q$ ($\gamma^{x_k \to i}_Q$ resp.) is a path from $x_1$ (from $x_k$ resp.) to $i$,
	get affected.  Namely, the moment map equation for $e^{(1)} (e^{(1)})^*$ (or $e^{(k)} (e^{(k)})^*$) gets an extra term $(a_1')^* a_1'
	$ ($-a_{k-1}'(a_{k-1}')^*$ resp.).   (If $k=1$, $x_1=x_k$ and $a'_1 = a'_{0}$, and the moment map equation for $e^{(1)} (e^{(1)})^*$ gets both the extra terms $(a_0')^* a_0'
	$ and $-a_0'(a_0')^*$.) As a result, we have an extra negative term 
	\begin{equation}
	- \gamma^{x_k \to i}_Q a_{k-1}'(a_{k-1}')^* (\gamma^{x_k \to i}_Q)^* \label{eq:extra-neg-term}
	\end{equation}
	for each path $\gamma^{x_k \to i}_Q$.  We shall show that these negative terms can be canceled.
	
	We also have additional paths in $Q'$ heading to $i$, which can be divided into the following types:
	\begin{enumerate}
		\item $\gamma_Q^{x_k \to i} a_{k-1}'\ldots a_l'e^{(x_l)}$ for $l=2,\ldots,k-1$.  (This is an empty case when $k=1,2$.)
		\item $\gamma_Q^{x_k \to i} \gamma_{\mathrm{new}} \gamma_Q^{j \to x_1} e_j$ where $\gamma_{\mathrm{new}} := a_{k-1}'\ldots a_1'$ and $\gamma_Q^{j \to x_1}$ is any path in $Q$ from $j \in Q_0$ to $x_1$.  ($\gamma_Q^{x_1 \to x_1}$ can be the trivial path.  $\gamma_{\mathrm{new}}=a_0'$ when $k=1$.)
		\item $\gamma_Q^{x_k \to i} \left(\prod_{r=1}^p \gamma_{\mathrm{new}} \gamma_{Q,r}^{x_k \to x_1}\right) a_{k-1}'\ldots a_l'e^{(x_l)}$ for some $p>0$ and $l=2,\ldots,k-1$.
		\item $\gamma_Q^{x_k \to i} \left(\prod_{r=1}^p \gamma_{\mathrm{new}} \gamma_{Q,r}^{x_k \to x_1}\right) \gamma_{\mathrm{new}} \gamma_Q^{j \to x_1}e_j$ for some $p>0$ and $j \in Q_0$.
	\end{enumerate}
	For (1), by the moment-map equation $e^{(x_l)}(e^{(x_l)})^* = I + (a'_{l})^*a'_{l} - a'_{l-1}(a'_{l-1})^*$, we have
	\begin{align*}
	&\gamma_Q^{x_k \to i} a_{k-1}'\ldots a_l'\cdot e^{(x_l)}(e^{(x_l)})^* \cdot (\gamma_Q^{x_k \to i} a_{k-1}'\ldots a_l')^* \\
	=& \gamma_Q^{x_k \to i} a_{k-1}'\ldots a_l'(\gamma_Q^{x_k \to i} a_{k-1}'\ldots a_l')^* + \gamma_Q^{x_k \to i} a_{k-1}'\ldots a_l'(a'_{l})^*a'_{l} (\gamma_Q^{x_k \to i} a_{k-1}'\ldots a_l')^* \\
	&- \gamma_Q^{x_k \to i} a_{k-1}'\ldots a_l' a'_{l-1}(a'_{l-1})^* (\gamma_Q^{x_k \to i} a_{k-1}'\ldots a_l')^*.
	\end{align*}
	The first term above for $l=k-1$ cancel with the extra negative term \eqref{eq:extra-neg-term} for $\rho_i^Q$.  The third term (which is negative) for $l \in \{3,\ldots,k-1\}$ cancel with the first term of $l-1$.  As a result, after combining (1) with the modified $\rho_i^Q$, the remaining negative terms are
	\begin{equation}
	- \gamma_Q^{x_k \to i} a_{k-1}'\ldots a'_{1} (\gamma_Q^{x_k \to i} a_{k-1}'\ldots a'_{1})^* = - \gamma_Q^{x_k \to i} \gamma_{\mathrm{new}} (\gamma_Q^{x_k \to i} \gamma_{\mathrm{new}})^*. \label{eq:extra-1}
	\end{equation}
	(For the case $k=1$, this trivially holds since $\gamma_{\mathrm{new}}=a'_0$, and the above equals to \eqref{eq:extra-neg-term}.)
	
	Now consider (2): $\gamma_Q^{x_k \to i} \gamma_{\mathrm{new}} \gamma_Q^{j \to x_1} e_j e_j^* (\gamma_Q^{x_k \to i} \gamma_{\mathrm{new}} \gamma_Q^{j \to x_1})^*$.  The moment map equation for $e_j e_j^*$ when $j \not= x_1, x_k$ are the same for $Q$ and $Q'$.  
	Summing $\gamma_Q^{x_k \to i} \gamma_{\mathrm{new}} \gamma_Q^{j \to x_1} e_j e_j^* (\gamma_Q^{x_k \to i} \gamma_{\mathrm{new}} \gamma_Q^{j \to x_1})^*$ over arbitrary $\gamma_Q^{j \to x_1}$ and $j\in Q_0$, we obtain $$(\gamma_Q^{x_k \to i} \gamma_{\mathrm{new}})\cdot \rho^Q_{x_1} \cdot (\gamma_Q^{x_k \to i} \gamma_{\mathrm{new}})^* = (\gamma_Q^{x_k \to i} \gamma_{\mathrm{new}})\cdot (I+B_1) \cdot (\gamma_Q^{x_k \to i} \gamma_{\mathrm{new}})^*$$
	plus
	$$ \sum_{\gamma_Q^{x_1 \to x_1}} \gamma_Q^{x_k \to i} \gamma_{\mathrm{new}} \gamma_Q^{x_1 \to x_1} (a_1')^*a_1' (\gamma_Q^{x_k \to i} \gamma_{\mathrm{new}} \gamma_Q^{x_1 \to x_1})^* - \sum_{\gamma_Q^{x_k \to x_1}} \gamma_Q^{x_k \to i} \gamma_{\mathrm{new}} \gamma_Q^{x_k \to x_1} a_{k-1}'(a_{k-1}')^* (\gamma_Q^{x_k \to i} \gamma_{\mathrm{new}} \gamma_Q^{x_k \to x_1})^*$$
	which is due to the additional terms in the moment map equations for $e^{(x_1)}(e^{(x_1)})^*$ and $e^{(x_k)}(e^{(x_k)})^*$.  (When $k=1$, $\gamma_Q^{x_0 \to x_0}$ being the trivial path at $x_0$ is one of the possibilities.) In above, 	$B_1$ is semi-positive-definite by the assumption on $Q$.  Then the negative terms \eqref{eq:extra-1} cancel with the first term.  After combining the paths in $Q$ and (1) and (2), the remaining negative terms are
	\begin{equation}
	- \sum_{\gamma_Q^{x_k \to x_1}} \gamma_Q^{x_k \to i} \gamma_{\mathrm{new}} \gamma_Q^{x_k \to x_1} a_{k-1}'(a_{k-1}')^* (\gamma_Q^{x_k \to i} \gamma_{\mathrm{new}} \gamma_Q^{x_k \to x_1})^*. \label{eq:extra-2}
	\end{equation}
	(If there is no path $\gamma_Q^{x_k\to x_1}$ in $Q$ from $x_k$ to $x_1$, then this is zero, and we do not have (3) nor (4).  We stop here and get that $\rho_i^{Q'}(\rho_i^{Q'})^* = I + B_i'$ for a semi-positive-definite matrix $B_i'$.)
	
	The terms in (3) for $p=1$ and $l=k-1$ cancel with the above \eqref{eq:extra-2}:
	\begin{align*}
	&\gamma_Q^{x_k \to i}\left(\prod_{r=1}^p \gamma_{\mathrm{new}} \gamma_{Q,r}^{x_k \to x_1}\right) a_{k-1}'\ldots a_l'\cdot e^{(x_l)}(e^{(x_l)})^* \cdot \left(\gamma_Q^{x_k \to i}\left(\prod_{r=1}^p \gamma_{\mathrm{new}} \gamma_{Q,r}^{x_k \to x_1}\right) a_{k-1}'\ldots a_l'\right)^* \\
	=& \gamma_Q^{x_k \to i} \left(\prod_{r=1}^p \gamma_{\mathrm{new}} \gamma_{Q,r}^{x_k \to x_1}\right) a_{k-1}'\ldots a_l'\left(\gamma_Q^{x_k \to i} \left(\prod_{r=1}^p \gamma_{\mathrm{new}} \gamma_{Q,r}^{x_k \to x_1}\right) a_{k-1}'\ldots a_l'\right)^* \\
	&+ \gamma_Q^{x_k \to i} \left(\prod_{r=1}^p \gamma_{\mathrm{new}} \gamma_{Q,r}^{x_k \to x_1}\right) a_{k-1}'\ldots a_l'(a'_{l})^*a'_{l} \left(\gamma_Q^{x_k \to i} \left(\prod_{r=1}^p \gamma_{\mathrm{new}} \gamma_{Q,r}^{x_k \to x_1}\right) a_{k-1}'\ldots a_l'\right)^* \\
	&- \gamma_Q^{x_k \to i}\left(\prod_{r=1}^p \gamma_{\mathrm{new}} \gamma_{Q,r}^{x_k \to x_1}\right) a_{k-1}'\ldots a_l' a'_{l-1}(a'_{l-1})^* \left(\gamma_Q^{x_k \to i} \left(\prod_{r=1}^p \gamma_{\mathrm{new}} \gamma_{Q,r}^{x_k \to x_1}\right) a_{k-1}'\ldots a_l'\right)^*.
	\end{align*}
	Like in (1), for each $p$, the third term (which is negative) for $l \in \{3,\ldots,k-1\}$ cancel with the first term for $l-1$.  (When $p=0$ and $l=k-1$, the third term is exactly \eqref{eq:extra-2}.)  Then the remaining negative terms are \eqref{eq:extra-1} modified by inserting the loops $\left(\prod_{r=1}^p \gamma_{\mathrm{new}} \gamma_{Q,r}^{x_k \to x_1}\right)$ for $p>0$, that is, $	- \gamma_Q^{x_k \to i}\left(\prod_{r=1}^p \gamma_{\mathrm{new}} \gamma_{Q,r}^{x_k \to x_1}\right) \gamma_{\mathrm{new}} (\gamma_Q^{x_k \to i}\left(\prod_{r=1}^p \gamma_{\mathrm{new}} \gamma_{Q,r}^{x_k \to x_1}\right) \gamma_{\mathrm{new}})^*$.  Then like in (2), these negative terms cancel with terms in (4).  Summing up to finite $p$, the only negative terms left are \eqref{eq:extra-2} modified by inserting the loops $\left(\prod_{r=1}^p \gamma_{\mathrm{new}} \gamma_{Q,r}^{x_k \to x_1}\right)$:
	$$- \sum_{\gamma_Q^{x_k \to x_1}} \gamma_Q^{x_k \to i} \left(\prod_{r=1}^p \gamma_{\mathrm{new}} \gamma_{Q,r}^{x_k \to x_1}\right) \gamma_{\mathrm{new}} \gamma_Q^{x_k \to x_1} a_{k-1}'(a_{k-1}')^* \left(\gamma_Q^{x_k \to i} \left(\prod_{r=1}^p \gamma_{\mathrm{new}} \gamma_{Q,r}^{x_k \to x_1}\right) \gamma_{\mathrm{new}} \gamma_Q^{x_k \to x_1}\right)^*$$
	which cancel with terms in (3) for $(p+1)$.  As $p \to \infty$, $\|\left(\prod_{r=1}^p \gamma_{\mathrm{new}} \gamma_{Q,r}^{x_k \to x_1}\right)\|\to 0$.  This finishes the proof that $\rho^{Q'}_i(\rho^{Q'}_i)^* = I + B_i$, for $i\in Q_0$.
	
	For the case that $i=x_j$ for $j=2,\ldots,k-1$, the proof is similar.  (We do not need to consider this case when $k=1,2$.)  The paths in $Q'$ heading to $x_j$ are divided into the following types:
	\begin{enumerate}
		\item $ a_{j-1}'\ldots a_l'e^{(x_l)}$ for $l=2,\ldots,j$.
		\item $a_{j-1}'\ldots a_1'
		\left(\prod_{r=1}^p\gamma_{Q,r}^{x_k \to x_1} \gamma_{\mathrm{new}}\right)\gamma_Q^{j \to x_1} e_j$ for some $p\geq 0$, $j \in Q_0$.
		\item $a_{j-1}'\ldots a_1'\left(\prod_{r=1}^p\gamma_{Q,r}^{x_k \to x_1} \gamma_{\mathrm{new}}\right)
		\gamma_Q^{x_k \to x_1} a_{k-1}' \ldots a_l'e^{(x_l)}$ for some $p\geq 0$, $l=2,\ldots,k-1$.
	\end{enumerate}
	The cancellation is similar and we do not repeat here.
\end{proof}

Similarly, adding a chain at a single vertex of $Q$ preserves the positive-definiteness property.

\begin{lemma} \label{lem:add-chain}
	Suppose $Q$ as in Lemma \ref{lem:add-cycle}.
	Let $Q'$ be obtained by concatenating to $Q$ a chain $x_1 \stackrel{a_1'}{\to} \ldots \stackrel{a_{k-1}'}{\to} x_k$ at either $x_1$ or $x_k$ in $Q_0$.  Then for every $i \in Q'_0$,  $\rho^{Q'}_i(\rho^{Q'}_i)^* = I + B_i'$ for some semi-positive-definite matrix $B_i'$.
\end{lemma}
\begin{proof}
	The proof in this case is simpler than that of Lemma \ref{lem:add-cycle}, since there is no new oriented cycle.
	
	Consider the case that $x_k \in Q_0$.  If $i$ belongs to the chain, then the only paths that head to $i$ are contained in the chain.  Since no oriented cycle is involved in all such paths, Theorem \ref{thm:metric} already gives the result.
	
	If $i \in Q_0$, then
	\begin{align*}
	\rho^{Q'}_i(\rho^{Q'}_i)^* =&\sum_{j\in Q_0-\{x_k\}} \sum_{\gamma_Q^{j\to i}} (\gamma_Q^{j\to i} e^{(j)})(\gamma_Q^{j\to i} e^{(j)})^* + \sum_{\gamma_Q^{x_k \to i}} \gamma_Q^{x_k\to i}\cdot e^{(x_k)}(e^{(x_k)})^*\cdot(\gamma_Q^{x_k\to i})^* \\
	&+ \sum_{\gamma_Q^{x_k \to i}}\sum_{r=1}^{k-1} (\gamma_Q^{x_k\to i} a_{k-1}\ldots a_r)e_{x_r} e_{x_r}^*  \cdot(\gamma_Q^{x_k\to i} a'_{k-1}\ldots a'_r)^*\\
	=& \rho_i^Q(\rho_i^Q)^* - \sum_{\gamma_Q^{x_k \to i}} \gamma_Q^{x_k\to i}(a'_{k-1}(a'_{k-1})^*)(\gamma_Q^{x_k\to i})^* \\
	&+ \sum_{\gamma_Q^{x_k \to i}}\sum_{r=1}^{k-1} (\gamma_Q^{x_k\to i} a'_{k-1}\ldots a'_r)(I-a'_{r-1}(a'_{r-1})^*+(a'_r)^*a'_r) \cdot(\gamma_Q^{x_k\to i} a'_{k-1}\ldots a'_r)^*\\
	=& \rho_i^Q(\rho_i^Q)^* + \sum_{\gamma_Q^{x_k \to i}}\sum_{r=1}^{k-1}(\gamma_Q^{x_k\to i} a'_{k-1}\ldots a'_r)(a'_r)^*a'_r\cdot(\gamma_Q^{x_k\to i} a'_{k-1}\ldots a'_r)^*.
	\end{align*}
	Since $\rho_i^Q(\rho_i^Q)^* = I + B_i$ for some semi-positive-definite matrix $B_i$, and the second term is semi-positive-definite, $\rho^{Q'}_i(\rho^{Q'}_i)^*$ satisfies the requirement.
	
	The case that $x_1 \in Q_0$ is similar and the proof is omitted.
\end{proof}

\begin{proof}[Proof of Theorem \ref{thm:metric-with-cycles}]
	Without loss of generality, suppose $Q$ is connected.  (Otherwise $\cM_{\vec{n},\vec{d}}$ and $\cV_i$ decompose into products coming from the connected components, and we just need to study each component.)
	The case without oriented cycle is given in Theorem \ref{thm:metric}.  Suppose $Q$ has at least one oriented cycle.  By Lemma \ref{lem:single-cycle}, the statement is true for this oriented cycle as a quiver.  There must be additional arrows if this single oriented cycle is not yet the whole $Q$.  Then we can either add a chain as in Lemma \ref{lem:add-cycle} or \ref{lem:add-chain}, and the statement still holds.  (Both the cases of loop at a vertex or multiple edge are covered by Lemma \ref{lem:add-cycle}.)  Inductively the statement holds for $Q$.
\end{proof}

\begin{example}
	For the $A_2$-quiver, $\rho_i\rho_i^*$ gives a metric for all $i$.  By Lemma \ref{lem:add-cycle}, this is still true if we add an oriented cycle with arrows $l_1,\ldots,l_p$.  See Figure \ref{fig:A2loop}.  
\end{example}

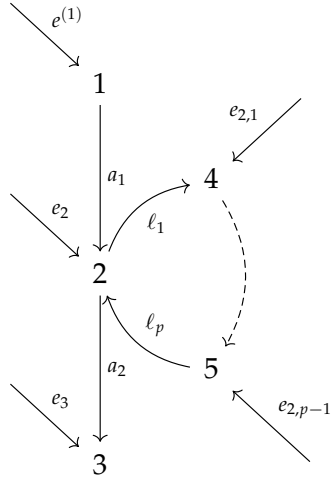
\begin{figure}[h]
	\begin{tikzcd}
	\arrow[rd, "e^{(1)}"]& & &\\	
	& 1 \arrow[dd, "a_1"]& & \arrow[ld, "e_{2, 1}"']\\
	\arrow[rd, "e_2"]& & 4 \arrow[dd, bend left, dashrightarrow]&\\
	& 2 \arrow[dd, "a_2"] \arrow[ru, bend left, "\ell_1"']& &\\
	\arrow[rd, "e_3"]& & 5 \arrow[lu, bend left, "\ell_p"']&\\
	& 3 & & \arrow[lu, "e_{2, p-1}"']\\
	\end{tikzcd}
	\caption{$A_2$ modified by adding an oriented cycle.}
	\label{fig:A2loop}
\end{figure}

Note that the following equality still holds over $R^\circ_{n,d}$:
	$$\rho \rho^* = e^{(i)} (e^{(i)}) ^* + \sum_{h(a)=i} V_a \rho_{(t(a))}\rho_{(t(a))}^* V_{a}^*.$$
Thus Proposition \ref{prop:relate-H_i} still holds for quivers with oriented cycles.

\begin{prop}
	For any quiver $Q$ and every $v,w\in (\cV_i)^*$, 
	$$H_i(v,w) = H_0((e^{(i)})^*(v),(e^{(i)})^*(w)) + \sum_{h(a)=i} H_{t(a)}(a^*(v),a^*(w)).$$
\end{prop}

Now we consider a version of Theorem \ref{thm:Ricci} in this case.  We define $\hat{W}_i$ by Equation \eqref{eq:W}.  But this time, it is an infinite direct sum of Hilbert spaces (meaning that it consists of infinite sequence $w=(w_\gamma : h(\gamma)=i)$ with $\|w\|^2=\sum_\gamma \|w_\gamma\|^2 < \infty$).

\begin{lemma}
	For each $(V,e)\in R^\circ_{n,d}$, $\rho_i(V,e)$ defines a bounded linear map $\hat{W}_i \to V_i$.  Its adjoint $\rho_i(V,e)^*: V_i \to \hat{W}_i$ has a singular-value decomposition.
\end{lemma}

\begin{proof}
	For $w=(w_\gamma : h(\gamma)=i) \in \hat{W}_i$, $\rho_i(V,e)$ maps it to 
	$$\sum_{\gamma: h(\gamma)=i} V_\gamma e^{(t(\gamma))} w_\gamma. $$
	Like in the proof of Lemma \ref{lem:conv}, consider $\sum_{\gamma: h(\gamma)=i} \|V_\gamma e^{(t(\gamma))}\| \|w_\gamma\| \leq \|w\| \sum_{\gamma: h(\gamma)=i} \|V_\gamma e^{(t(\gamma))}\| < +\infty.$
	This also shows that if $\|w\|=1$, then the image of $w$ is also bounded.  $\rho_i(V,e)^*$ has image being finite-dimensional (since $V_i$ is finite-dimensional), and hence is a compact operator.  Thus it has a singular-value decomposition.
\end{proof}

\begin{prop}
	The Ricci curvature of the metric given by $(\rho_i \rho_i^*)^{-1}$ is semi-positive definite on $\cM^\circ$.
\end{prop}

\begin{proof}
	By the previous lemma, the proof of Theorem \ref{thm:Ricci} on semi-positive definiteness still works.  Namely, 
	$$\partial \overline{\partial }\log \det \rho \rho ^{\mathrm{*}}
	=\mathrm{tr~ }\left(\left(\partial \rho \right)^{\mathrm{*}}\left(\rho \rho ^{\mathrm{*}}\right)^{-1}\partial \rho \right)-\mathrm{tr~ }\left(\left(\partial \rho \cdot \left(\rho ^{\mathrm{*}}\left(\rho \rho ^{\mathrm{*}}\right)^{-\frac{1}{2}}\right)\right)^{\mathrm{*}}\left(\rho \rho ^{\mathrm{*}}\right)^{-1}\left(\partial \rho \cdot \left(\rho ^{\mathrm{*}}\left(\rho \rho ^{\mathrm{*}}\right)^{-\frac{1}{2}}\right)\right)\right).$$
	Note that the two terms on the RHS are finite: $\mathrm{tr~ }\left(\left(\partial \rho \right)^{\mathrm{*}}\left(\rho \rho ^{\mathrm{*}}\right)^{-1}\partial \rho \right) = \mathrm{tr~ }\left(\rho \rho ^{\mathrm{*}}\right)^{-1}\partial \rho \left(\left(\partial \rho \right)^{\mathrm{*}} \right) = \sum_{j=1}^{d_i} \langle \partial \rho \left(\partial \rho \right)^{\mathrm{*}} \epsilon_j, \left(\rho \rho ^{\mathrm{*}}\right)^{-1} \epsilon_j \rangle_{V_i}$ which is a finite sum, and similar for the second term.  $\rho ^{\mathrm{*}}=\left(\rho ^{\mathrm{*}}\left(\rho \rho ^{\mathrm{*}}\right)^{-\frac{1}{2}}\right)\left(\rho \rho ^{\mathrm{*}}\right)^{\frac{1}{2}}$ is decomposed into the rescaling $(\rho \rho ^{\mathrm{*}})^{\frac{1}{2}}$ and the orthogonal embedding $\left(\rho ^{\mathrm{*}}\left(\rho \rho ^{\mathrm{*}}\right)^{-\frac{1}{2}}\right)$ to $\mathrm{Im}\, \rho^{\mathrm{*}}\subset \widehat{W}_{i}$.  Then the above equals to $\|(\partial \rho)_2\|_H^2 \geq 0$ as in Theorem \ref{thm:Ricci}.
\end{proof}


	

\section{Fiberwise Nonlinearity} \label{sec:nonlinear}
In the mathematical study of quivers, we mostly focused on linear representations.  In particular, the morphisms between universal vector bundles are linear along fibers.  
On the other hand, nonlinear `activation functions' play a key role in machine learning.  In this section, we construct some natural non-linear fiber-bundle endomorphisms of the universal bundles $\cV_i$ over $\cM_{\vec{n},\vec{d}}$ by using fiberwise symplectomorphisms.


For simplicity, we shall take $n_i=d_i+1$ for all $i\in Q_0$ in this section.


\subsection{Activation functions arising from toric moment maps and symplectomorphisms} \label{sec:toric}

In this section, we make the observation that several activation functions commonly used in machine learning actually belong to a much bigger class, namely the $T$-equivariant symplectomorphisms on open subsets of a symplectic toric variety.  

First, let's recall the basic setup for toric varieties.
Let's equip $\C^m$ with the standard K\"ahler structure.  We obtain a symplectic toric variety $(X,\omega_X)$ as a symplectic quotient by the real torus $T^{m-d}$.  We assume $X$ is smooth.  The $T^{m-d}$-action can be specified by an injective homomorphism $\Z^{m-d} \to \Z^m$, which induces a map $T^{m-d} \to T^m$, and $T^m$ acts on $\C^m$ by coordinate-wise multiplication.  We assume that the quotient of $\Z^m$ by the image of $\Z^{m-d}$ is again a lattice, which we identify as $\Z^d$.  We denote by $v_i \in \Z^d$ the images of the standard basic vectors of $\Z^m$ under the quotient map $\Z^m \to \Z^d$.

The residual action of $T^d$ on $X$ gives a moment-map fibration over a polytope $P$, which is given by the intersection of $m$ half-spaces in $\R^d$: 
$$ \{x \in \R^d: \ell_j(x) := v_j \cdot x - c_j \geq 0\} $$
where the constants $c_j \in \R^d$ are determined by the level taken in the symplectic quotient.  We assume that the level is chosen such that for all $j$, $\{\ell_j(x) = 0\} \cap P$ is a (non-empty) codimension-one boundary of the polytope $P$.

Consider the open toric orbit of $X$, which can be identified as $(\C^\times)^d$ by fixing a basis of $\Z^d$.  Denote by $\omega_X$ the K\"ahler form induced on the symplectic quotient.  Let $\omega_{\mathrm{std}} = \sum_{i=1}^d dx_i \wedge d\theta_i$ be the standard symplectic form on $\R^d \times T^d$ (where $T^d$ denotes the real $d$-torus).  
The symplectic form $\omega_X|_{(\C^\times)^d}$ has an explicit description by the following beautiful formula.

\begin{theorem}[\cite{Guillemin,Abreu}] \label{thm:toric}
$$ \left(\frac{1}{2} \left(\sum_{i=1}^m v_i \log \ell_i(x)\right), \mathrm{Id}\right): (P^\circ \times T^d, \omega_{\mathrm{std}}|_{P^\circ \times T^d}) \to \R^d \times T^d \stackrel{\exp}{\cong} ((\C^\times)^d,\omega_X|_{(\C^\times)^d})$$ 
is a symplectomorphism.
\end{theorem}

Taking the universal cover $\R^d \to T^d$ and lifting the above, one obtains the following.

\begin{corollary} \label{cor:sympl}
	The inverse of $ \left(\frac{1}{2} \left(\sum_{i=1}^m v_i \log \ell_i(x)\right), \,\mathrm{Id}\right)$ gives a symplectomorphism 
	$$\sigma_\C = (\sigma(\mathrm{Re}(\vec{z})),\mathrm{Im}(\vec{z})): (\C^d,\exp^* \omega_X) \to (P^\circ \times \R^d,\omega_{\mathrm{std}})$$
	where $\exp: \C^d \to (\C^\times)^d \subset X$, and $\sigma: \R^d \to P^\circ$ is the inverse of $\left(\frac{1}{2} \left(\sum_{i=1}^m v_i \log \ell_i(x)\right), \,\mathrm{Id}\right)$.
\end{corollary}

\begin{example}	\label{ex:C}
	For $\C^d$, the moment polytope $P$ is $\R_{\geq 0}^d = \{x_i \geq 0: i=1,\ldots,d\}$.  $v=(1,\ldots,1)$.  The above map is simply $\left(\frac{\log x_i}{2}\right)_{i=1}^d: \R_{>0}^d \to \R^d$.  The symplectomorphism $((\C^\times)^d,\omega_{\C^n}|_{\C^\times)^d}) \stackrel{\cong}{\to} (\R_{>0}^d \times T^d, \omega_{\mathrm{std}}|_{P^\circ \times T^d})$ is $\left(|z_i|^2, \frac{\log z_i - \log \bar{z_i}}{2i} \right)_{i=1}^d$.
\end{example}

\begin{example}	\label{ex:Pn}
For the complex projective space $\bP^d$, the corresponding moment polytope $P$ (for a chosen level) is the $d$-simplex given by $\ell_i \geq 0$ where $\ell_i(x)=x_i$ for $i=1,\ldots,d$, and $\ell_{d+1}(x)=1-x_1-x_2-\dots-x_d$.  The generators are $v_i = \epsilon_i$ for $i=1,\ldots,d$ (the standard basis) and $v_{d+1}=-\sum_{i=1}^d v_i$.  We have
$$d\left(\frac{1}{2}\sum_{i=1}^{d+1} \ell_i(x) \log \ell_i(x)\right) = \frac{1}{2}\sum_{i=1}^d \log\left(\frac{x_i}{1-\sum_{j=1}^d x_j}\right) dx_i$$
as a map $P^\circ \to \R^d$.  By direct computation, the inverse of this map equals to
\begin{equation} \label{eq:expsigma}
\sigma(\vec{r})=\left(\frac{e^{2r_i}}{1+\sum_{j=1}^d e^{2r_j}}\right)_{i=1}^d: \R^d \to P^\circ.
\end{equation}
Written in terms of the complex coordinates $\vec{z} \in (\C^{\times})^d$, the symplectomorphism $((\C^{\times})^d, \omega_{\bP^n}|_{(\C^\times)^d})\to(P^\circ \times T^d, \omega_{\mathrm{std}}|_{P^\circ \times T^d})$ is given by
$$\left(\frac{|z_i|^2}{1+\sum_{j=1}^d |z_j|^2}, \frac{\log z_i - \log \bar{z_i}}{2i} \right)_{i=1}^d.$$
Pulling back by $\C^d \to (\C^{\times})^d$, we have the symplectomorphism $\sigma_\C = \left(\left(\frac{e^{2r_i}}{1+\sum_{j=1}^d e^{2r_j}}\right)_{i=1}^d,\mathrm{Id}\right):(\C^d,\exp^* \omega_X) \to (P^\circ \times \R^d,\omega_{\mathrm{std}})$.
	
When $d=1$, $\frac{e^{2r}}{1+ e^{2r}}$ is a commonly-used activation function.  See Figure \ref{fig:strip}.  By taking a direct product, $\left(\frac{e^{2r_i}}{1+ e^{2r_i}}\right)_{i=1}^d: \R^d \to [0,1]^d$ corresponds to $(\bP^1)^d$.
	
\end{example}

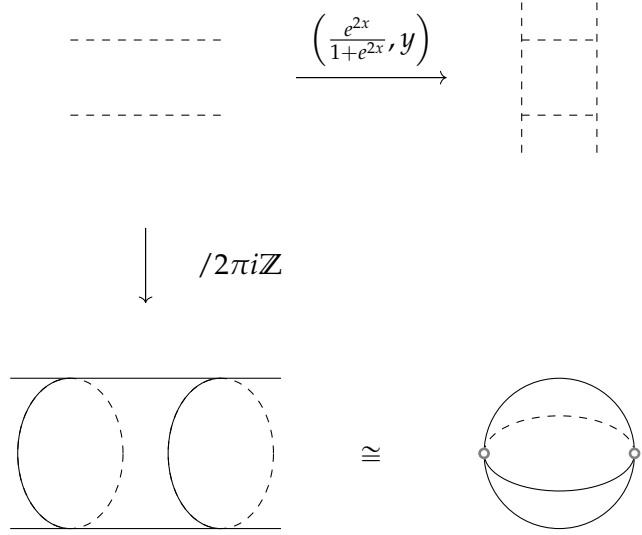
\begin{figure}
	\begin{tikzpicture}
	\draw[dashed] (-1, 0.5) -- (1, 0.5);
	\draw[dashed] (-1, -0.5) -- (1, -0.5);
	\draw(0, -2) -- (0, -3)
	(1.25, -2.5) node{$/2\pi i\Z$};
	\draw[->](0, -2) -- (0, -3);
	\draw (-1.8, -4) -- (1.8, -4);
	\draw[dashed] (1, -5) ellipse (0.7cm and 1cm);
	\draw (1, -4) arc (90:270:0.7cm and 1cm);
	\draw[dashed] (-1, -5) ellipse (0.7cm and 1cm);
	\draw (-1, -4) arc (90:270:0.7cm and 1cm);
	\draw (-1.8, -6) -- (1.8, -6);
	
	\node[] at (3, -5) {$\cong$};
	\draw (4.5, -5) arc (180:360:1cm and 0.5cm);
	\draw[dashed] (4.5, -5) arc (180:0:1cm and 0.5cm);
	\draw (5.5, -5) circle (1cm);
	\filldraw [white] (4.5, -5) circle (2pt);
	\filldraw [gray] (4.5, -5) circle (2pt);
	\filldraw [white] (4.5, -5) circle (1pt);
	\filldraw [white] (6.5, -5) circle (2pt);
	\filldraw [gray] (6.5, -5) circle (2pt);
	\filldraw [white] (6.5, -5) circle (1pt);
	\draw(2, 0) -- (4, 0)
	(3, 0.5) node{$\left(\frac{e^{2x}}{1+e^{2x}}, y\right)$};
	\draw[->](2,0) -- (4,0);
	\draw[dashed] (5, -1) -- (5, 1);
	\draw[dashed] (6, -1) -- (6, 1);
	\draw[dashed] (5, 0.5) -- (6, 0.5);
	\draw[dashed] (5, -0.5) -- (6, -0.5);

	\end{tikzpicture}
	\caption{$\C$ as a Covering Space Mapped to an Open Strip}
	\label{fig:strip}
\end{figure}

\begin{remark} \label{rmk:toric}
	In above, we have taken the quotient K\"ahler structure from $\C^m$.  For a general toric K\"ahler structure, the symplectomorphism in Theorem \ref{thm:toric} is given by $$\left(d\left(\frac{1}{2}  \left(\sum_{i=1}^m \ell_i \log \ell_i(x)\right)+h\right), \mathrm{Id}\right)$$
	where $h$ is a smooth function on the closed polytope $P$ such that the Hessian of $\left(\frac{1}{2}  \left(\sum_{i=1}^m \ell_i \log \ell_i(x)\right)+h\right)$ is positive definite in $P^\circ$ \cite{Abreu}.
	
	In particular, for a general projective toric variety $X$, we can take an embedding of $X$ to $\bP^N$ by toric holomorphic sections of a very ample line bundle $L$, and use the induced toric K\"ahler structure from $\bP^N$.  Then the symplectomorphism $\sigma_\C$ is given by $(\sigma, \mathrm{Id})$ where
	$$ \sigma(\vec{r})=\frac{\sum_{i=1}^d e^{2(\vec{u_i},\vec{r})} \vec{u_i} }{\sum_{j=1}^d e^{2(\vec{u_j},\vec{r})}}: \R^d \to P^\circ, $$
	$\vec{u_i}$ are points such that their convex hull equal to $P$, and $(\vec{u},\vec{r})$ is the standard dot product on $\R^d$.  See \cite[Section 4.2]{Fulton}.
\end{remark}

Now we have the symplectomorphisms $\phi_X:(P^\circ \times T^d, \omega_{\mathrm{std}}|_{P^\circ \times T^d}) \stackrel{\cong}{\to} ((\C^\times)^d,\omega_X|_{(\C^\times)^d})$ and $\phi_{\C^d}:\R_{>0}^d \times T^d \stackrel{\cong}{\to} ((\C^\times)^d,\omega_{\C^d}|_{(\C^\times)^d})$ (Example \ref{ex:C}).  For the toric structure of $X$, let's arrange the order of the indices such that the first $d$ vectors $v_i$ for $i=1,\ldots,d$ form a basis of $\Z^d$.  (We assume $m \geq d$.)  Moreover, we take the first $d$ constants $c_j=0$ for $j=1,\ldots,d$.  Then $P^\circ \subset \R_{>0}^d$.

Consider the composition $\phi_{\C^d} \circ \phi_X^{-1}: ((\C^\times)^d,\omega_X|_{(\C^\times)^d})\to ((\C^\times)^d,\omega_{\C^d}|_{(\C^\times)^d})$.  It is a symplectomorphism onto the image $\phi_{\C^d}(P^\circ \times T^d)$.

\begin{prop} \label{prop:sympl}
	$\phi_{\C^d} \circ \phi_X^{-1}$ extends to a $T$-equivariant symplectomorphism 
	$$\psi: (\C^d,\omega_X|_{\C^d}) \stackrel{\cong}{\to} (\pi_{\C^d}^{-1}(P-B),\omega_{\C^d}|_{\pi_{\C^d}^{-1}(P-B)})$$
	where $B = \bigcup_{i=d+1}^m \{\ell_i(x) = 0\}$, and $\pi_{\C^d} = (|z_i|^2)_{i=1}^d: \C^d \to \R_{\geq 0}^d$ is the moment map for $\C^d$.
\end{prop}

\begin{proof}
	$\phi_X$ is given by $2r^X_i = \log x_i + \sum_{j=d+1}^m (v_j^{(i)} \log \ell_j(x))$, where $v_j=(v_j^{(1)},\ldots,v_j^{(d)})$.  $\phi_{\C^d}$ is given by $2r^{\C^d}_i =\log x_i$.  Hence $e^{2r^X_i}=e^{2r^{\C^d}_i} \prod_{j=d+1}^m \ell_j^{v_j^{(i)}}(e^{2r^{\C^d}_1},\ldots,e^{2r^{\C^d}_d})$.  In terms of the complex coordinates, this gives $$z^X_i=z^{\C^d}_i \left(\prod_{j=d+1}^m \ell_j^{v_j^{(i)}}\left(|z^{\C^d}_1|^2,\ldots,|z^{\C^d}_d|^2\right)\right)^{1/2}.$$  
	It is obviously well-defined over $\pi_{\C^d}^{-1}(P-B)$.  We need to show that it has inverse, which gives the required extension of $\phi_{\C^n} \circ \phi_X^{-1}$.  Since $(\phi_{\C^n} \circ \phi_X^{-1})^*(\omega_{\C^n})=\omega_X$ on $(\C^\times)^d$, this still holds over $\C^d$ as the equality is a closed condition.
	
	Consider the Jacobian of $\vec{z}^X(\vec{z}^{\C^d},\overline{\vec{z}^{\C^d}})$.  We shall show it is positive definite, and hence invertible.  To simplify, we write $z = z^{\C^d}$.  Denote $G=\frac{1}{2}\left(\sum_{i=1}^m \ell_i(x) \log \ell_i(x) - v\cdot x\right)$.  For any non-zero vector $(a_1,\ldots,a_d)$,
	\begin{align*}
	&\sum_{i,j} \overline{a_j} a_i \partial_{z_i} z_j^X \\
	=& \sum_{i,j} \overline{a_j} a_i \partial_{z_i} \left(\exp (\partial_{x_j}|_{x_p = z_p\bar{z_p}} G) \cdot \frac{z_j}{|z_j|} \right) \\
	=& \sum_{i,j} \overline{a_j} a_i \exp (\partial_{x_j}|_{x_p = z_p\bar{z_p}} G) \cdot \partial_{z_i} \left(\frac{z_j}{|z_j|} \right) + \sum_{i,j} \overline{a_j} a_i \frac{z_j}{|z_j|} \cdot \exp (\partial_{x_j}|_{x_p = z_p\bar{z_p}} G) \cdot \frac{\partial (z_i\bar{z_i})}{\partial z_i} \cdot \frac{\partial^2 G}{\partial x_i \partial x_j} \\
	=& \sum_{i} |a_i|^2 |z_i^X| \cdot \left(\frac{1}{2|z_i|}\right) + \sum_{i,j} |z_j|^{-1} \cdot |z_j^X| \cdot \overline{a_j}z_j \cdot a_i  \bar{z_i} \cdot \frac{\partial^2 G}{\partial x_i \partial x_j}.
	\end{align*}
	Note that $|z_j|^{-1}|z_j^X| = \left(\prod_{k=d+1}^m \ell_k^{v_k^{(j)}}\left(|z^{\C^d}_1|^2,\ldots,|z^{\C^d}_d|^2\right)\right)^{1/2}$ which is positive.  Let $c = \min \{|z_j|^{-1}|z_j^X|: j=1,\ldots,d\}$.  Then the second term is no less than $c a_i  \bar{z_i} \cdot \frac{\partial^2 G}{\partial x_i \partial x_j}$.  Since $\frac{\partial^2 G}{\partial x_i \partial x_j}$ is positive definite on $P^\circ$, it is semi-positive definite on $P-B$.  Thus this term is non-negative.  The first term is positive.  Thus $\sum_{i,j} \overline{a_j} a_i \partial_{z_i} z_j^X>0$.  Similarly $\sum_{i,j} a_j \overline{a_i} \overline{\partial_{z_i}} z_j^X$.  Hence the Jacobian is positive-definite and hence invertible.
\end{proof}

\begin{example}
We continue to consider $\bP^d$.  From Example \ref{ex:Pn}, $x_i = \frac{|z^{\bP^d}_i|^2}{1+\sum_{j=1}^d |z^{\bP^d}_j|^2}$.  $\pi^{-1}_{\C^d}(P-B) = \{\|\vec{z}^{\C^d}\| < 1\}$.  From Example \ref{ex:C}, $x_i = |z^{\C^d}_i|^2$.  Hence the symplectomorphism $(\C^d,\omega_{\bP^d}|_{\C^d}) \stackrel{\cong}{\to} (\{\|\vec{z}^{\C^d}\| < 1\},\omega_{\C^n}|_{\pi_{\C^d}^{-1}(P-B)})$ is
\begin{equation} \label{eq:multi-sigmoid}
z^{\C^d}_i = \frac{z^{\bP^d}_i}{\sqrt{1+ \|\vec{z}^{\bP^d}\|^2}}.
\end{equation}
When $d=1$, this gives
$\frac{z}{\sqrt{1+|z|^2}}$
which is another activation function used in machine learning.  ($z$ is restricted in $\R$ in most algorithms.)  The symplectomorphism can be easily checked in this case: ($z=z^{\bP^1}$ for simplicity)
\begin{align*}
d z^{\C} \wedge \overline{d z^{\C}} =&  d \frac{z}{\sqrt{1+ |z|^2}} \wedge d \frac{\overline{z}}{\sqrt{1+|z|^2}} \\
=& \frac{(1+z\bar{z})dz - (|z|^2 dz + z^2 \overline{dz})/2}{(1+z\bar{z})^{3/2}} \wedge \frac{(1+z\bar{z})\overline{dz} - (|z|^2 \overline{dz} + \bar{z}^2 dz)/2}{(1+z\bar{z})^{3/2}} = \frac{dz\wedge \overline{dz}}{(1+z\bar{z})^2}
\end{align*}
giving the Fubini-Study metric.  See Figure \ref{fig:disk}.

By taking the direct product, the symplectomorphism for the case $X=(\bP^1)^d$ is $\left(\frac{z_j}{\sqrt{1+|z_j|^2}}\right)_{i=1}^d$.
\end{example}

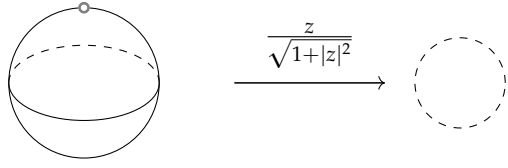
\begin{figure}
\begin{tikzpicture}
	\draw (-1, 0) arc (180:360:1cm and 0.5cm);
	\draw[dashed] (-1, 0) arc (180:0:1cm and 0.5cm);
	\draw (0, 0) circle (1cm);
	\filldraw [white] (0, 1) circle (2pt);
	\filldraw [gray] (0, 1) circle (2pt);
	\filldraw [white] (0, 1) circle (1pt);
    \draw(2, 0) -- (4, 0)
        (3, 0.5) node{$\frac{z}{\sqrt{1+|z|^2}}$};
        \draw[->](2,0) -- (4,0);
    \draw[dashed] (5, 0) circle (0.6cm);
    
\end{tikzpicture}
\caption{$\C$ as a Chart in $\mathbb{P}^1$ Mapped to an Open Disk}
\label{fig:disk}
\end{figure}

Due to the nice fact that the Fubini-Study metric on $\bP^d$ is $U(d)$-invariant (and so does the standard metric on $\C^d$), we have the following (which is not true for $(\bP^1)^d$ nor general toric manifolds).

\begin{lemma} \label{lem:psi-equiv}
For $X=\bP^d$, the symplectomorphism $\psi: (\C^d,\omega_{\bP^d}|_{\C^d}) \stackrel{\cong}{\to} \left(\pi_{\C^d}^{-1}(P-B),\omega_{\C^n}|_{\pi_{\C^d}^{-1}(P-B)}\right)$ in Proposition \ref{prop:sympl} is $U(d)$-equivariant.
\end{lemma}
\begin{proof}
Using the explicit expression \eqref{eq:multi-sigmoid}, 
$$\psi(U \cdot \vec{z}^{\bP^d})=\frac{U \cdot \vec{z}^{\bP^d}}{\sqrt{1+ \|U\cdot \vec{z}^{\bP^d}\|^2}} = \frac{U \cdot \vec{z}^{\bP^d}}{\sqrt{1+ \|\vec{z}^{\bP^d}\|^2}} = U \cdot \psi(\vec{z}^{\bP^d})$$
for all $U \in U(d)$.
\end{proof}

As explained in Remark \ref{rmk:toric}, we can also equip $X$ with another $T$-invariant K\"ahler form (that do not come from the standard K\"ahler structure on $\C^m$).  The symplectomorphism is given by
$$ \left(\frac{1}{2} d\left(\sum_{i=1}^m \ell_i(x) \log \ell_i(x)+ h(x)\right), \mathrm{Id}\right): (P^\circ \times T^d, \omega_{\mathrm{std}}|_{P^\circ \times T^d}) \to \R^d \times T^d \stackrel{\exp}{\cong} ((\C^\times)^d,\omega^h|_{(\C^\times)^d}).$$ 
Thus the toric construction is rather flexible.

\begin{example}
	The `softplus' function $x=\log (1+e^{2y}): \R \to \R_{>0}$ gives an example of such a K\"ahler structure on $\C$ (by identifying it with $\R_{>0} \times \R$ with the standard symplectic structure).  The inverse is $y =\frac{1}{2}\log (e^x - 1)$, whose difference with $\frac{1}{2}\log x$ is $h' = \frac{1}{2}\log \frac{e^x - 1}{x} = \frac{1}{2}\log \left(1+\sum_{k=1}^\infty \frac{x^k}{(k+1)!}\right)$ which is indeed a smooth function on $\R_{\geq 0}$.  Moreover, $y' = \frac{e^x}{2(e^x-1)} > 0$ on $\R_{>0}$.
\end{example}

\subsection{Symplectomorphisms of fiber bundles over the moduli}
In the last subsection, we have exhibited various symplectic embeddings for $\C^n$.  Now we want to make a family version of these maps over the framed quiver moduli $\cM_{\vec{n},\vec{d}}$.  The last subsection can be understood as constructing self-maps on a fiber of a vector bundle over $\cM_{\vec{n},\vec{d}}$.

To globalize \eqref{eq:multi-sigmoid}, we consider the universal bundle $\cV_i$ equipped with a Hermitian metric $H_i$.  (We have constructed a nice Hermitian metric on $\cV_i$ in Section \ref{sec:Herm}.)  We have a fiberwise symplectic structure $\omega_{\cV_i}$ induced from the Hermitian metric.  Moreover, we have the projective bundle $\mathbb{P}(\cV_i\oplus\mathcal{O}_\cM)$ which is a fiberwise compactification of $\cV_i$.  Then the fiber bundle $\mathbb{P}(\cV_i\oplus\mathcal{O}_\cM)$ is equipped with a fiberwise K\"ahler metric $\omega_{\mathbb{P}(\cV_i\oplus\mathcal{O}_\cM)}$ induced from $H_i$ (namely, $\frac{i}{2}\partial\bar{\partial}\log (H_i \oplus H_0)$ where $H_0$ is the trivial metric on $\cO_M$).

\begin{prop}
 There is a fiberwise symplectomorphism $$\psi_{\cV_i}: (\cV_i,\omega_{\mathbb{P}(\cV_i\oplus\mathcal{O}_\cM)}|_{\cV_i}) \stackrel{\cong}{\to} \left(\{v \in \cV_i: H_i(v,v) < 1\},\omega_{\cV_i}|_{\{H_i(v,v) < 1\}}\right).$$
\end{prop}

\begin{proof}
	For each $p \in \cM$, we have computed the symplectomorphism $$(\cV_i|_p,\omega_{\mathbb{P}(\cV_i\oplus\mathcal{O}_\cM)}|_{\cV_i|_p}) \stackrel{\cong}{\to} \left(\{v \in \cV_i|_p: H_i(v,v) < 1\},\omega_{\cV_i}|_{\{H_i(v,v) < 1\}}\right)$$
	in \eqref{eq:multi-sigmoid}, with the metric given by $H_i|_p$ here.  Thus
	\begin{equation}
	\psi_{\cV_i}(v) = \frac{v}{\sqrt{1+H_i(v, v)}}
	\label{eq:psi}
	\end{equation}
	gives a fiberwise symplectomorphism whose image is ${\{H_i(v,v) < 1\}}$.
\end{proof}

Recall that the universal bundle $\cV_i \to \cM$ admits an action of $U_{\vec{n}}$ coming from framing (Definition \ref{def:symmetry}).  One advantage of $\psi_{\cV_i}$ is that it is equivariant under this action.

\begin{lemma} \label{lem:action}
	For $g \in U_{\vec{n}}$,
	$$\psi_{\cV_i}\circ g = g \circ \psi_{\cV_i}$$
	if we used the metric $H_i$ given in Theorem \ref{thm:metric}.
\end{lemma}

\begin{proof}
	This follows from \eqref{eq:psi} and $H_i(g\cdot v,g \cdot v)=H_i(v,v)$ by Lemma \ref{lem:metric-equiv}.
\end{proof}

\begin{remark}
	Equation \eqref{eq:psi} has an alternative derivation using the framing.
	Namely, we have the surjective morphism $\rho: \underline{\hat{W}_i} \to \cV_i$ (see Equation \eqref{eq:W}), whose dual give a fiberwise-linear embedding $\rho^*: \cV_i \stackrel{H_i}{\cong} \cV_i^* \to \underline{\hat{W}_i}^*$.  (The underline means the trivial bundle over $\cM$ associated with the vector space.)  Then $\bP(\underline{\hat{W}_i}^*\oplus \cO)$ (with the standard metric) induces a fiberwise K\"ahler form on $\cV_i^*$, and we have a fiberwise symplectic embedding $(\cV_i,\omega_{\bP(\underline{\hat{W}_i}^*\oplus \cO)}) \hookrightarrow (\cV_i,\omega_{\cV_i})$.  This gives
	$$ \frac{\rho\rho^*\cdot H_i\cdot v}{\sqrt{1+H_0(\rho^*\cdot H_i\cdot v, \rho^*\cdot H_i\cdot v)}} = \frac{\rho\rho^*\cdot H_i\cdot v}{\sqrt{1+v^* \cdot H_i^* \cdot \rho \rho^*\cdot H_i\cdot v}}.$$
	Now if we use the metric $H_i = (\rho \rho^*)^{-1}$ given by Theorem \ref{thm:metric}, then the above equals to the expression in \eqref{eq:psi}.
\end{remark}

We also have a more flexible construction using the framing, which globalize \emph{any given non-linear continuous map} $\sigma_\C: \C^{n_i} \to \C^{n_i}$.  Namely, $\sigma_\C$ can be regarded as a fiberwise non-linear self-map on the trivial bundle $\underline{\C^{n_i}} \to \underline{\C^{n_i}}$ (still denoted by $\sigma_\C$).  Then we take the composition
$$ \cV_i \stackrel{H_i}{\cong} \cV_i^* \stackrel{(e^{(i)})^*}{\to} \underline{\C^{n_i}} \stackrel{\sigma_\C}{\to} \underline{\C^{n_i}} \stackrel{e^{(i)}}{\to} \cV_i $$
and denote it by $\sigma_{\cV_i}$.  See Figure \ref{fig:framing}.

\begin{figure}[h]
	\begin{tikzcd}
	\underline{\C^n} \arrow[d, "e^{(i)}"] \arrow[loop left, "\sigma_\C"] \arrow[r, "\cong"]& \underline{\C^n}^*\\
	\cV_i \arrow[r, "\cong"]& \cV_i^* \arrow[u, "(e^{(i)})^*"]
	\end{tikzcd}
	\caption{Using the framing to globalize an activation function $\sigma_\C:\C^n \to \C^n$ over $\cM$.}
	\label{fig:framing}
\end{figure}

It is easy to get the following explicit expression in terms of $\sigma_\C$.

\begin{lemma}
	The above fiber-bundle map $\sigma_{\cV_i}: \cV_i \to \cV_i$ equals to
	\begin{equation}
	\sigma_{\cV_i}(v) = \sum_{k=1}^{n_i}  (\sigma_\C)_k\left(H_i(e^{(i)}_1,v),\ldots, H_i(e^{(i)}_{n_i},v) \right) \cdot e^{(i)}_k
	\label{eq:sigma}
	\end{equation}
	where we write $\sigma_\C = ((\sigma_\C)_1,\ldots,(\sigma_\C)_{n_i})$ and $e^{(i)} = (e^{(i)}_1 \ldots e^{(i)}_{n_i})$.
\end{lemma}

For instance, we can take $\sigma$ to be the one coming from the symplectomorphism in Corollary \ref{cor:sympl}.  We shall prove the universal approximation theorem for such $\sigma$ in Section \ref{sec:approx}.

\subsection{A machine learning program using the framed quiver moduli}

Let $Q$ be a digraph and denote by $\C\cdot Q$ its path algebra over $\C$.  
Let $\vec{d} \in \Z_{\geq 0}^{Q_0}$ be a dimension vector.  We take the framing dimension vector to be $\vec{n} = \vec{d} + \vec{1}$, where $\vec{1}_i := 1$ for all $i\in Q_0$.  The one additional framing vector is used for translation (called a `bias' vector).

We fix a collection of input vertices and a collection of output vertices $I_{\mathrm{in}}, I_{\mathrm{out}} \subset Q_0$, and 
$$\gamma \in I_{\mathrm{out}} \cdot (\C\cdot Q) \cdot I_{\mathrm{in}} = \bigoplus_{i \in I_{\mathrm{in}}, j \in I_{\mathrm{out}}} j \cdot (\C\cdot Q) \cdot i.$$
(The trivial path at a vertex $i$ is again denoted by $i$.)

\subsubsection{Machine learning using a flat space.} \label{sec:flat}
Let's first formulate a typical machine learning program in the quiver setup.  The following flat space 
$$U \cong \prod_{a\in Q_1} \Hom(V_{t(a)},V_{h(a)}) \times \prod_{i \in Q_0} V_i$$ 
is used frequently in the subject.  

\begin{lemma}
	The open subset 
	$$U=\{[V,e] \in \cM_{\vec{d}+\vec{1},\vec{d}}: (e^{(i)}_1,\ldots,e^{(i)}_{d_i}) = I_{d_i} \textrm{ for all } i\in Q_0\} \subset \cM$$ 
	gives a coordinate chart of $\cM$. ($I_{d_i}$ denotes the identity matrix of rank $d_i$.)
\end{lemma}
\begin{proof}
	First, such $(V,e)$ are stable: $\mathrm{Im}(e)$ is the whole $V$.  Second, for distinct $(V,e),(V',e')$ satisfying the above condition, $[V,e]\not=[V',e']$: since they are stable, their orbits are closed.  Suppose $g\cdot (V,e)=(V',e')$ for some $g \in \GL_{\vec{d}}$.  Then $g_i \cdot e^{(i)} = (e')^{(i)}$.  But since $(e^{(i)}_1,\ldots,e^{(i)}_{d_i}) = I_{d_i} = ((e')^{(i)}_1,\ldots,(e')^{(i)}_{d_i})$, this forces $g = \mathrm{Id}$ and so $(V,e)=(V',e')$, contradicting that they are distinct.
	
	Then we have the chart map $U \stackrel{\cong}{\to} \prod_{a\in Q_1} \Hom(V_{t(a)},V_{h(a)}) \times \prod_{i \in Q_0} V_i$ defined by 
	$$ W_a:=(e^{(i)}_1,\ldots,e^{(i)}_{d_i})^{-1}V(a),\,\, b_i:=(e^{(i)}_1,\ldots,e^{(i)}_{d_i})^{-1}e^{(i)}_{d_i+1}. $$
\end{proof}

Now we fix a path $\gamma$ from the input vertices $I_{\mathrm{in}}$ to the output vertices $I_{\mathrm{out}}$.
For each element $[V,e]\in U$, by composing the affine linear maps $V_a(\cdot)+e^{(t(a))}_{d_{t(a)}+1}$ attached to arrows $a$ in the path $\gamma$, together with some non-linear functions $\sigma_i: V_i \to V_i$ that are called `activation functions', one obtains a non-linear function 
$$f^U_{[V,e]}: \bigoplus_{i\in I_{\mathrm{in}}} \C^{d_{i}} \stackrel{(e^{(i)}_1,\ldots,e^{(i)}_{d_i})}{\cong} \bigoplus_{i\in I_{\mathrm{in}}} V_i \to \bigoplus_{j\in I_{\mathrm{out}}}V_j \stackrel{(e^{(j)}_1,\ldots,e^{(j)}_{d_i})}{\cong} \bigoplus_{j\in I_{\mathrm{out}}}\C^{d_{j}}  $$ 
which is used to approximate a non-explicitly given function $f$.  A stochastic gradient flow of the error function on $U$ is employed to find the optimal point in $U$.

Let's write down the symmetry in Lemma \ref{lem:action} in the chart $U$.

\begin{lemma} \label{lem:inv-chart}
	The chart $U$ is invariant under the right action by 
	$$\prod_{i\in Q_0} U(d_i) \subset \prod_{i\in Q_0} U(d_i+1) = U_{\vec{n}}$$
	where $U(d_i) \subset U(d_i+1)$ is embedded as $\left(\begin{array}{cc} U(d_i) & 0 \\ 0 & 1 \end{array}\right)$.
	
	Consider the trivialization $\cV_i|_U \cong \left(\prod_{a\in Q_1} \Hom(V_{t(a)},V_{h(a)}) \times \prod_{j \in Q_0} V_j\right) \times V_i$.  The right action of $g \in U(d_i)$ on $\cV_i|_U$ is given by 
	$$(W_a,b_j,v)_{\substack{a \in Q_1\\j \in Q_0, b_j \in V_j\\v \in V_i}}\cdot g = (g^{-1}\cdot W_a, g^{-1}\cdot b_j, g^{-1}v)_{\substack{a \in Q_1\\j \in Q_0, b_j \in V_j\\v \in V_i}}$$
	where $g^{-1}\cdot W_a$ equals to $g^{-1}W_a$ if $h(a)=i$, $W_ag$ if $t(a)=i$, and $W_a$ otherwise; $g^{-1}\cdot b_j$ equals to $g^{-1}b_j$ if $j=i$, and $b_j$ if $j\not=i$.
\end{lemma}

\begin{proof}
	$U$ consists of points $[V,e]$ where $(e^{(j)}_1,\ldots,e^{(j)}_{d_i})$ are invertible for all $j\in Q_0$.  This property is invariant under the action of $\prod_{i\in Q_0} U(d_i)$.  Hence $U$ is an invariant subset.
	
	$(W_a,b_j,v)_{\substack{a \in Q_1\\j \in Q_0, b_j \in V_j\\v \in V_i}}$ corresponds to the point $[W,e,v] \in \cV_i = (R^s_{n,d} \times V_i)/\GL_{\vec{d}}$. where $e^{(j)} = (I_{d_j} \,\, b_j)$.
	\begin{align*}
	&[W,e,v]\cdot g = [W,e\cdot g, v] = [W, (g \,\, b_j)_{j\in Q_1}, v] \\
	=& [g^{-1}\cdot W, (I_{d_j} \,\, g^{-1}\cdot b_j)_{j\in Q_1}, g^{-1}\cdot v]
	\end{align*}
	where the left action by $g^{-1}$ (as an element in $\GL_{\vec{d}}$) is as specified by definition.
\end{proof}

Now we prove an important symmetric property that the activation function \eqref{eq:multi-sigmoid} enjoys, which can be used to reduce the dimensions.

\begin{prop} \label{prop:sym}
	Suppose the activation functions $\sigma_i:V_i \to V_i$ are taken to be the one given in Equation \eqref{eq:multi-sigmoid}.  Then $f^U$ is invariant under $\prod_{j\not\in I_{\mathrm{in}} \cup I_{\mathrm{out}}} U(d_j) \subset U_{\vec{n}}$.  (See the embedding in Lemma \ref{lem:inv-chart}.)
\end{prop}

\begin{proof}
	The terms of $f^U$ are of the form $\sigma_{h(a_k)} \left(W_{a_k} \ldots \left(\sigma_{h(a_1)}\left(W_{a_1}(v)+b_{h(a_1)}\right)\ldots \right) + b_{h(a_k)}\right)$.  By the above lemma, for $g \in U(d_{h(a_i)})$ where $h(a_i)\not\in I_{\mathrm{in}} \cup I_{\mathrm{out}}$, the action of $g$ results in $\sigma_{h(a_i)} \mapsto g\cdot \sigma_{h(a_i)} g^{-1}$ in the above expression and does not affect any other part.  
	By Lemma \ref{lem:psi-equiv}, $\sigma_i$ is $U(d_i)$-equivariant, and hence $f^U$ remains invariant. 
\end{proof}

By the above proposition, we can descend $f^U$ to the orbit space $U/\prod_{j\not\in I_{\mathrm{in}} \cup I_{\mathrm{out}}} U(d_j)$ to reduce the dimensions.  However, the quotient space will be highly singular.  Instead, we can realize the dimension reduction by restricting $f^U$ to a submanifold $U'$ of $U$ whose orbit occupies the whole $U$, and do stochastic gradient flow on $U'$ instead of $U$. The following gives one simple possibility.

\begin{prop} \label{prop:sym2}
	Consider a subset $\{a_1,\ldots,a_p\}$ of arrows whose heads and tails do not belong to $I_{\mathrm{in}}\cup I_{\mathrm{out}}$, and for any two distinct arrows $a_1,a_2$ in the subset, $t(a_1) \not= h(a_2)$.  Then the vector subspace
	$$U' := \{(W_a,b_j)_{\substack{a \in Q_1\\j \in Q_0, b_j \in V_j}} \in U: W_{a_i}  \textrm{ is of the form } \left(D \,\, W_{a_i}'\right) \,\, \forall i=1,\ldots,p\}$$
	has its orbit being the whole $U$, that is, $\left(\prod_{j\not\in I_{\mathrm{in}} \cup I_{\mathrm{out}}} U(d_j)\right) \cdot U' = U$.  In above, $D$ is a diagonal matrix (of the maximum possible size) and $W_{a_i}'$ is any matrix occupying the rest.
\end{prop}
\begin{proof}
	This follows from the singular-value decomposition of a matrix $W$ as $A \cdot \left(D \,\, W'\right) \cdot B$ where $A$ and $B$ are unitary matrices of appropriate sizes.
\end{proof}

\subsubsection{Machine learning using the quiver moduli.} \label{sec:ML}
The quiver moduli $\cM = \cM_{\vec{n},\vec{d}}$ gives a compactification of $U$.  Compactness is important for the formulation of Morse theory and convergence of a gradient flow.  We would like to use the whole $\cM$ in application of machine learning.  Non-trivial metrics over the moduli will play a crucial role.

First, consider the situation before adding in activation functions.  
Each arrow $a\in Q_1$ is associated with a vector-bundle morphism $a_\cM: \cV_{t(a)} \to \cV_{h(a)}$.  For each path $a_k\ldots a_1$, we take the map
\begin{equation}
(a_k)_\cM\left(\ldots \left((a_2)_\cM\left((a_1)_\cM(v) + e^{(h(a_1))}_{d_{h(a_1)}+1}\right) + e^{(h(a_2))}_{d_{h(a_2)}+1}\right)\ldots\right)+ e^{(h(a_k))}_{d_{h(a_k)}+1}
\label{eq:compose}
\end{equation}
which is fiberwise affine linear.
Thus a path $\gamma$ gives an affine bundle morphism
$$\gamma_\cM: \cV_{I_{\mathrm{in}}} := \bigoplus_{i \in I_{\mathrm{in}}} \cV_i \to \cV_{I_{\mathrm{out}}} := \bigoplus_{j \in I_{\mathrm{out}}} \cV_j.$$
Then we have
\begin{equation}
L_\gamma \left(\left(s^{(i)}_{k}\right)_{\substack{i\in I_{\mathrm{in}}\\k\in\{1,\ldots,d_i\} }}\right)=\left(H_j \left(e^{(j)}_p,\sum_{i\in I_{\mathrm{in}}} \gamma_\cM \cdot \sum_{k=1}^{d_i} s^{(i)}_{k} e^{(i)}_{k} \right)\right)_{\substack{j \in I_{\mathrm{out}},\\ p\in\{1,\ldots,d_j\}}}: \bigoplus_{i\in I_{\mathrm{in}}} \underline{\C^{d_i}} \to \bigoplus_{j\in I_{\mathrm{out}}} \underline{\C^{d_j}}.
\label{eq:L}
\end{equation}

Suppose a continuous function $f = K \to \bigoplus_{j\in I_{\mathrm{out}}} \C^{d_j}$ is given, where $K$ is a compact subset of $\bigoplus_{i\in I_{\mathrm{in}}} \C^{d_i}$.  
Then the fiberwise integral
\begin{equation}
\cE := \int_K \left\|f - L_\gamma \right\|^2_{\cV_{I_{\mathrm{out}}}} d\mu_K
\label{eq:E}
\end{equation}
gives a smooth function on $\cM$.  

\begin{remark}
	Alternatively, we can define the fiber-bundle morphism
	$$ f_\cM: \sum_{j\in I_{\mathrm{out}}}\sum_{l=1}^{d_j} f^{(j)}_l \cdot e^{(j)}_l:\underline{K} \to \cV_{I_{\mathrm{out}}} $$
	and take
	$$ \int_K \left\|f_\cM\left(\left(s^{(i)}_{k}\right)_{\substack{i\in I_{\mathrm{in}}\\k\in\{1,\ldots,d_i\} }}\right) - \sum_{i\in I_{\mathrm{in}}} \gamma_\cM \cdot \sum_{k=1}^{d_i} s^{(i)}_{k} e^{(i)}_{k} \right\|^2_{\cV_{I_{\mathrm{out}}}} d\mu_K. $$
	However, with such a definition, we need to worry that $(e^{(j)}_1,\ldots,e^{(j)}_{d_j})$ for some output $j$ degenerates (as a frame), in which case approximating $f$ and approximating $f_\cM$ are different.
\end{remark}

We have a gradient flow $r: \R \to \cM$ which can be used to minimize $\cE$:
$$ \frac{dr}{dt} = -(\nabla \cE)(r(t))$$
where $\nabla \cE = (d \cE)^{\#_g}$ where $(\cdot)^\#_g: T^*\cM \stackrel{g}{\cong} T\cM$ is the identification by a metric $g$ on $\cM$.  ($(\nabla \cE)^p = g^{pq} \partial_q \cE$ in local coordinates.)  $g$ can be taken to be the induced metric from the trivial metric on the vector space $R_{\vec{n},\vec{d}}$ via symplectic reduction.  Alternatively, $g$ can be taken to be the metric given by the Ricci curvature in Theorem \ref{thm:Ricci} (when $Q$ has no oriented cycle), which has a better expression in homogeneous coordinates.

Note that $(L_\gamma)|_{[V,e]}$ is affine linear on $\bigoplus_{i\in I_{\mathrm{in}}} \underline{\C^{d_i}}$ for every $[V,e]\in\cM$, which is not good enough for the purpose of approximating $f$.  We introduce fiberwise non-linearity below.

\begin{defn}
	Let $A$ be a finite set whose every element is associated with two vertices (head $h$ and tail $t$) in $Q_0$.  Elements in $A$ are called activation arrows.  (These are not arrows in $Q_1$.)  The semiring generated by $Q_1$ and $A$, denoted by $\Gamma(Q,A)$, has the underlying vector space spanned by the independent set $\coprod_{p=0}^\infty S_p$ where $S_p$ is defined inductively as follows.
	
	\begin{enumerate}
		\item $S_0$ consists of all paths of $Q$.
		\item Suppose $S_p$ has been defined, and each element in $S_p$ has a head $h$ and a tail $t$.  $S_{p+1}$ consists of $\gamma \cdot \alpha \cdot \tilde{\gamma}$, where $\alpha \in A$, $\gamma$ is any path of $Q$ with $h(\alpha) = t(\gamma)$, and $\tilde{\gamma} \in t(\alpha)\cdot (\C\cdot S_p)$.  ($\C\cdot S_p$ denotes the vector space generated by $S_p$;  $t(\alpha)\cdot (\C \cdot S_p)$ is the subspace generated by elements of $S_p$ with head being $t(\alpha)$.)  The above element has the head $h(\gamma)$ and the tail $t(\tilde{\gamma})$.	
	\end{enumerate}
	The above vector space $\Gamma(Q,A)$ has an obvious product by concatenation.  ($s_1 \cdot s_2 = 0$ if $h(s_2)\not=t(s_1)$.)  Note that for $\alpha \in A$, $(s_1+c s_2)\cdot \alpha = s_1 \cdot \alpha + c s_2 \cdot \alpha$, but $\alpha \cdot (s_1+c s_2) \not= \alpha \cdot s_1 + c \alpha \cdot s_2$.
\end{defn}

Now, suppose each element $\alpha \in A$ is associated with a fiber-bundle morphism $$\alpha_\cM: \cV_{t(\alpha)} \to \cV_{h(\alpha)}.$$  
Then by composing the corresponding affine linear morphisms (as in Equation \eqref{eq:compose}) and $\alpha_\cM$, an element $\tilde{\gamma} \in \Gamma(Q,A)$ induces a fiber-bundle morphism
$$ \tilde{\gamma}_\cM:  \cV_{t(\tilde{\gamma})} \to \cV_{h(\tilde{\gamma})}.$$
In particular, if we fix $\tilde{\gamma} \in I_{\mathrm{out}} \cdot \Gamma(Q,A) \cdot I_{\mathrm{in}}$, then we define $$f_{\tilde{\gamma}}: \bigoplus_{i\in I_{\mathrm{in}}} \underline{\C^{d_i}} \to \bigoplus_{j\in I_{\mathrm{out}}} \underline{\C^{d_j}}$$ 
like in the definition of $L_\gamma$ in \eqref{eq:L} by replacing $\gamma_\cM$ by $\tilde{\gamma}_\cM$.  Then a stochastic gradient flow for the corresponding error function (by replacing $L_\gamma$ by $f_{\tilde{\gamma}}$ in \eqref{eq:E}) can be carried out.

For now, we set $A$ to be $Q_0$ as a set, and each element $\mathfrak{o}_i$ has the head and tail being $i \in Q_0$.  We can associate $\mathfrak{o}_i$ with the fiber-bundle morphism $\psi_{\cV_i}$ given in \eqref{eq:psi}, or $\sigma_{\cV_i}$ in \eqref{eq:sigma}.  Then we obtain $f_{\tilde{\gamma}}$ above which is non-linear along fibers for the purpose of machine learning.

If we associate $\mathfrak{o}_i$ with $\psi_{\cV_i}$ given in \eqref{eq:psi}, then the symmetry of $\prod_{j\not\in I_{\mathrm{in}} \cup I_{\mathrm{out}}} U(d_j) \subset U_{\vec{n}}$ is respected, by Lemma \ref{lem:action}.  The proof is similar to that for Proposition \ref{prop:sym} and is omitted here.

\begin{prop}
	Suppose $\mathfrak{o}_i$ is assigned as $\psi_{\cV_i}$ given in \eqref{eq:psi}, and $H_i$ is taken to be the metric in Theorem \ref{thm:metric}.  Then $f_{\tilde{\gamma}}$ is invariant under $\prod_{j\not\in I_{\mathrm{in}} \cup I_{\mathrm{out}}} U(d_j) \subset U_{\vec{n}}$.  (See the embedding in Lemma \ref{lem:inv-chart}.)
\end{prop}

\begin{remark}
	Since we are taking affine linear morphisms $(a_i)_\cM(v) + e^{(h(a_i))}_{d_{h(a_i)}+1}$ for the arrows which involves the term $e^{(h(a_i))}_{d_{h(a_i)}+1}$, only the symmetry $U(d_j)$ rather than $U(d_j+1)=U(n_j)$ is respected.  If the bias vector $e^{(h(a_i))}_{d_{h(a_i)}+1}$ is not used in the program, then $f_{\tilde{\gamma}}$ will be invariant under the bigger group $\prod_{j\not\in I_{\mathrm{in}} \cup I_{\mathrm{out}}} U(n_j)$.
\end{remark}

\subsubsection{A simple example} \label{sec:eg}
Recall the quiver in Example \ref{ex:A3}.  Let's take $n_1=d_1, n_2=d_2+1, n_3=d_3$ instead of $n_i=d_i+1 \,\forall i$, since we do not need to use bias vectors at the input and output vertices.  The path $\gamma$ is simply $a_2a_1$, and $\tilde{\gamma} = a_2 \cdot \mathfrak{o}_1 \cdot a_1$.

We have the universal bundles $\cV_1 \cong \underline{\C}^{d_1}$, $\cV_2$ and $\cV_3$.  We have $\rho^{(1)} = e^{(1)}$, $\rho^{(2)} = (e^{(2)} \,\, a_1e^{(1)})$, $\rho^{(3)} = (e^{(3)} \,\, a_2e^{(2)} \,\, a_2a_1e^{(1)})$.
In terms of homogeneous coordinates (namely the coordinates $(e^{(1)},e^{(2)},e^{(3)},a_1,a_2)$ on the vector space $R_{\vec{n},\vec{d}}$, where each entry is a matrix of suitable size), 
the metrics given in Theorem \ref{thm:metric} are 
\begin{align*}
H_1 &= \left(e^{(1)} (e^{(1)})^*\right)^{-1}, \, H_2 = \left(e^{(2)} (e^{(2)})^* + a_1 e^{(1)} (e^{(1)})^* a_1^*\right)^{-1}, \,\\ H_3 &= \left(e^{(3)} (e^{(3)})^* + a_2 e^{(2)} (e^{(2)})^* a_2^* + a_2 a_1 e^{(1)} (e^{(1)})^* a_1^* a_2^*\right)^{-1}.
\end{align*}
on $\cV_i, i=1,2,3$ respectively.  The activation functions we constructed in the previous subsection are
$$ \psi_i(v) = \frac{v}{\sqrt{1+v^* H_i v}} \textrm{ and } \sigma_i(v) = \sum_{k=1}^{d_i+1} (\sigma_\C)_k\left((e^{(i)}_1)^*H_iv,\ldots,(e^{(i)}_{d_i+1})^*H_iv\right)\cdot e^{(i)}_k,$$
where we can take $(\sigma_\C)_k (\vec{z}) = \frac{e^{2 \mathrm{Re}(z_k)}}{1+\sum_{j=1}^{d_i+1} e^{2\mathrm{Re}(z_j)}} + \bi \, \mathrm{Im}(z_k)$ for instance.  Both have the $\GL$-equivariance property $\psi_i(g \cdot v) = g \cdot \psi_i(v)$ and $\sigma_i(g\cdot v) = g\cdot \sigma_i(v)$.

The function (over $\cM$) cooked up from this quiver is $f_{\tilde{\gamma}}: \C^{d_1} \to \C^{d_3}$,
$$f_{\tilde{\gamma}}(s_1,\ldots,s_{d_1}) = \left(H_3 \left(e^{(3)}_p, a_2 \cdot \sigma_2\left(a_1 \cdot \sum_{k=1}^{d_1} s_k e^{(1)}_{k}  + e^{(2)}_{d_2+1}\right)\right)\right)_{p=1}^{d_3} $$
if we use $\sigma_2$ as the activation function, or the same expression with $\sigma_2$ replaced by $\psi_2$.  Then we run a stochastic gradient flow on $\cM$ (or on the vector space $R_{\vec{n},\vec{d}}$ upstairs) to minimize the distance of $f_{\tilde{\gamma}}$ and a function $f$ coming from reality.

To run the gradient flow, we need to take a metric on $\cM$.  Recall that we have the metric on the tangent bundle of $\cM$ coming from the Ricci curvatures of $H_i$:
{\small
	\begin{equation}
	H_T = \sum_{i=1}^3 \mathrm{tr~ }\left(\left(\partial_v \rho^{(i)} \right)^{\mathrm{*}}\cdot H_i\cdot\partial_w \rho^{(i)} \right)
	-\sum_{i=1}^3\mathrm{tr~ }\left(\left(\partial_v \rho^{(i)} \cdot (\rho^{(i)}) ^{\mathrm{*}}\cdot H_i^{\frac{1}{2}}\right)^{\mathrm{*}}\cdot H_i \cdot \left(\partial_w \rho^{(i)} \cdot (\rho^{(i)}) ^{\mathrm{*}}\cdot H_i^{\frac{1}{2}}\right)\right).
	\label{eq:H_T2}
	\end{equation}
}
Consider the open subset of the vector space $R_{\vec{n},\vec{d}}$ in which $(e_1^{(i)},\ldots,e_{d_i}^{(i)})$ is invertible.  (This is the preimage of the chart $U\subset \cM$.)  A tangent vector $v$ of $\cM$ is lifted as $(\delta a_1,\delta a_2, \delta e_{d_2+1}^{(2)})$ (and all other components are set to be zero).  
Since $\rho^{(1)} = e^{(1)}$, $\rho^{(2)} = (e^{(2)}, a_1 e^{(1)})$, $\rho^{(3)} = (e^{(3)}, a_2 e^{(2)}, a_2a_1e^{(1)})$, we have
$\partial_v \rho^{(1)} = 0$, $\partial_v \rho^{(2)} = \left((0\,\, \delta e_{d_2+1}^{(2)}),\,\, (\delta a_1)e^{(1)}\right)$, and 
$$\partial_v \rho^{(3)} = \left(0,\,\, (\delta a_2)e^{(2)} + (0\,\,a_2 \delta e_{d_2+1}^{(2)}),\,\,\, (\delta a_2)a_1e^{(1)} + a_2 (\delta a_1)e^{(1)} \right).$$
Then the above metric $H_T$ can be computed explicitly in terms of the homogeneous coordinates.

\begin{remark}
	In above, if we use trivial metrics over the vector space $R_{\vec{n},\vec{d}}$ instead of $H_i$ and $H_T$, the expressions will get simpler; however they will only be $U_{\vec{d}}$-equivariant rather than $\GL_{\vec{d}}$-equivariant.  Then we need to restrict to the moment-map level $\mu^{-1}(I) \subset R_{\vec{n},\vec{d}}$ and its tangent bundle, in order to stay in the same moduli $\cM$ downstairs.  This would increase the computational complexity.
\end{remark}

We can also write in inhomogeneous coordinates $(W_1,W_2,b)$ in the chart $U \subset \cM$, where $(e_1^{(i)},\ldots,e_{d_i}^{(i)})=I_{d_i}$ for all $i=1,2,3$.  $W_1,W_2$ are the matrices of the arrows $a_1,a_2$ and $e_{d_2+1}^{(2)}=b$ is the bias vector. Then $H_1 = I$, $H_2 = (I+bb^* + W_1 W_1^*)^{-1}$, and $H_3 = (I + W_2W_2^*+W_2 bb^* W_2^* +W_2W_1W_1^*W_2^*)^{-1}$.  Then we can run the gradient flow in $U \subset \cM$ (which has lower dimensions than $R_{\vec{n},\vec{d}}$).

Note that when $W_1, W_2, b$ are close to zero, $H_i$ are close to the identity matrix, and $\sigma_i$ is close to $\sigma_\C$.  Moreover, the second term of \eqref{eq:H_T2} is close to zero, and the first term is close to the standard metric.  Thus when $W_1,W_2,b$ are small, the function $f_{\tilde{\gamma}}$ is close to the commonly used one 
\begin{equation} f^U_{W_1,W_2,b}(v) = W_2\cdot ( \sigma_\C(W_1 \cdot v + b)),
\label{eq:f^U-A3}
\end{equation}
and the gradient flow is close to the usual one on
 the flat space $U$ (see Section \ref{sec:flat}).  The additional terms can be understood as modifications to ensure the flow converges in $\cM$.




\subsection{A discussion on Morse inequalities} \label{sec:topo}

By the work of Reineke \cite{Reineke}, the framed quiver moduli $\cM$ is a tower of Grassmannians (Theorem \ref{thm:Reineke}), and hence its Poincar\'e polynomial is a product of that of Grassmannians (Corollary \ref{cor:Reineke}).  Such topological invariants give important information about a gradient flow on $\cM$.

In particular the Morse inequalities for a Morse function $\cE$ on a compact manifold $\cM$ state as follows.  Let $c^j(\cE)$ be the number of critical points of index $j$ for $\cE$.  Then for every $j$, $$c^j(\cE) \geq h^j(\cM)$$ where $h^j(\cM)$ denotes the cohomological numbers (which are coefficients of the Poincar\'e polynomial).

Given a gradient flow, which is a path $\gamma: \R\to \cM$ satisfying the gradient flow equation, $\lim_{t\to \pm \infty} \gamma(t)$ are critical points.  Moreover, critical points carry important effect to the rate of the gradient flow.  Namely, when the flow $\gamma$ gets close to a critical point with index being $1,\ldots, \dim \cM - 1$, $\|\gamma'(t)\| = \| \mathrm{grad}\, \cE(\gamma(t))\|$ becomes small.  In other words the flow slows down when it passes through a neighborhood of a critical point.  Such a slowing-down effect of saddle points was studied in machine learning in \cite{PDGB,DPGCGB}.

The cohomological numbers $h^j(\cM)$ give the minimum number of critical points and hence are important invariants of a neural network (which simply means a directed graph $Q$ together with a dimension vector $\vec{d}$ here).  Over $\C$, $h^j(\cM) = 0$ when $j$ is odd.  Thus the Euler characteristic $\chi_{Q,\vec{d}}$ equals to $\sum_j h^j(\cM)$, which is the minimal total number of critical points.  It is computed by simply setting $q=1$ in Corollary \ref{cor:Reineke}.   Another important invariant is $\dim \cM$ (that is, the number of training parameters of the network), which is simply
$$ \cD_{Q,\vec{d}} = \sum_{i \in Q_0} d_i \left(n_i + \sum_{\substack{j \to i\\j \not= i}} d_j\right) $$
using the notation of Corollary \ref{cor:Reineke}.  (We take $n_i = d_i+1$ in this section.)
These are illustrated in the two practical examples below.

\begin{example}
	Consider $Q$ being the $A_{k+2}$ quiver, which has $k+2$ vertices  labeled by $0,\ldots,k+1$, and there is exactly one arrow from $i$ to $i+1$ for $i=0,\ldots,k$, and no arrow otherwise.  Set $d_{-1}:=0$.  Then the minimal total number of critical points is
	$$ \chi_{Q,\vec{d}} = \prod_{i=-1}^{k} \binom{d_i+d_{i+1}+1}{d_{i+1}} $$
	and
	$$ \cD_{Q,\vec{d}} = \sum_{i=-1}^k d_{i+1}(d_i +1).$$
\end{example}

\begin{example} \label{ex:Res}
	Now consider the following quiver $A_{k+2}'$, which has vertices labeled by $0,\ldots,k+1$, and there is one arrow from vertex $i$ to vertex $j$ for every $i<j$.  Set $d_{-1}:=0$.  Then
	$$ \chi_{Q,\vec{d}} = \prod_{i=-1}^{k} \binom{\sum_{j=0}^{i+1} d_j+1}{d_{i+1}} $$
	and
	$$ \cD_{Q,\vec{d}} = \sum_{i=-1}^k d_{i+1}\left(\sum_{j=0}^i d_j +1\right).$$	
\end{example}
Figure \ref{fig:chi} shows the graph of $\log \chi_{Q,\vec{d}}$ versus $\cD_{Q,\vec{d}}$ for the two examples, where we set $k= 3$, $d_1 = 600, d_5 = 10$, and $d_2=d_3=d_4$.    

Note that the quiver denoted by $A_{k+2}'$ in Example \ref{ex:Res} is a simple analog of the network known as ResNet, which adds arrows to the $A_{k+2}$-quiver that skip the middle vertices to get around with the `gradient-vanishing problem'.   Namely, in the $A_{k+2}$ case, the derivatives of $\cE$ with respect to matrix entries for arrows in the early stage are typically very small by chain rule, which is not good for the flow rate.  Arrows that skip the middle vertices are added, so that there are short paths which involve the early arrows.  

From Figure \ref{fig:chi}, we see that in the same dimensions, the minimal number of critical points in $\cM$ is smaller for $A_5'$ than that for $A_5$.  (We have numerically verified this for general $k$.)  This gives a supporting evidence that $\chi_{Q,\vec{d}}$ is an important invariant in applications to machine learning.

\begin{figure}[htb!]
	\includegraphics[scale=0.8]{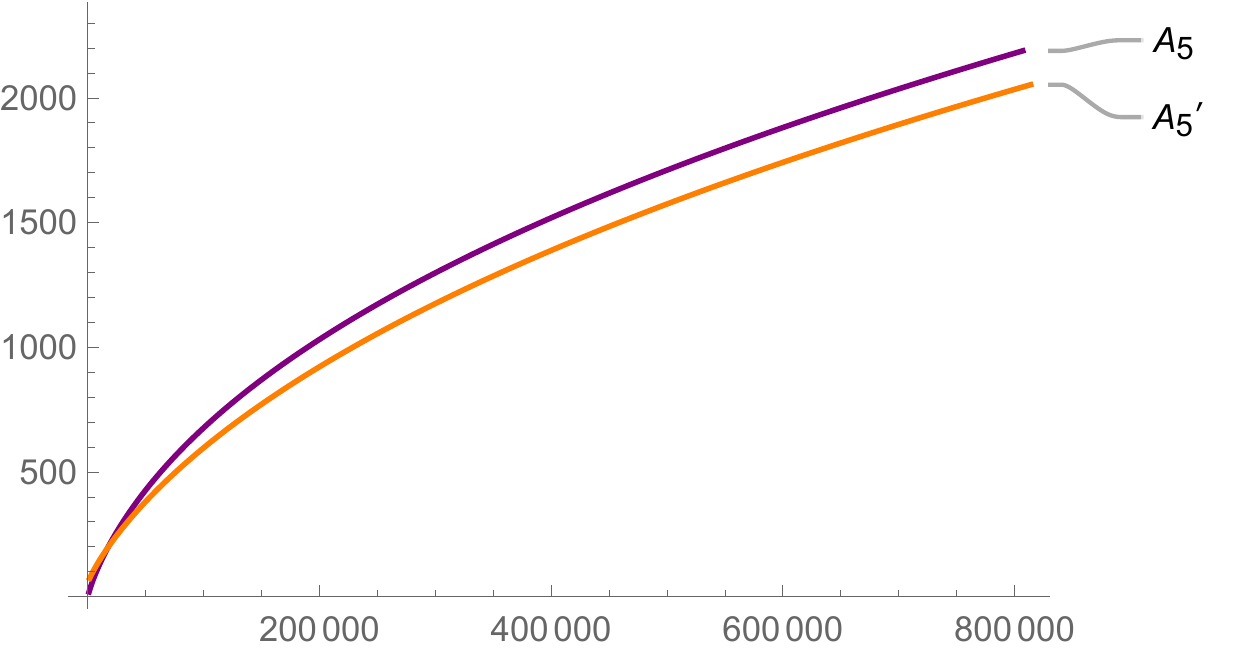}
	\caption{A plot of $\log \chi_{Q,\vec{d}}$ (y-axis) versus $\cD_{Q,\vec{d}}$ (x-axis) for $A_5$ and $A_5'$.}
	\label{fig:chi}
\end{figure}

\subsection{A remark on Abelianization}
	In many basic neural networks, each vertex of $Q$ is associated with a vector space of only dimension one.  When $\vec{d}=\vec{1}$, that is, all entries of the dimension vector equal to one, $\cM_{\vec{n},\vec{1}}$ is a quotient by the Abelian group $(\C^\times)^{\Sigma \vec{d}}$, and hence a toric variety.  Indeed, by Theorem \ref{thm:Reineke}, $\cM_{\vec{n},\vec{1}}$ is a tower of projective spaces $\bP^k$ for a sequence of $k$.

	Given $Q$ and $\vec{d}$, we can always construct a bigger quiver $Q^{\textrm{Ab},\vec{d}}$as follows.  For each vertex $i \in Q_0$, we make $d_i$ copies indexed by $(i,p)$ for $p=1,\ldots,d_i$.  For each arrow of $Q$ from $i$ to $j$, we make a corresponding arrow for $Q^{\textrm{Ab},\vec{d}}$ from $(i,p)$ to $(j,q)$ for every $p=1,\ldots,d_i$ and $q=1,\ldots,d_j$.  See Figure \ref{fig:Abel} for an example.
	
	\begin{figure}[htb!]
		\includegraphics[scale=0.5]{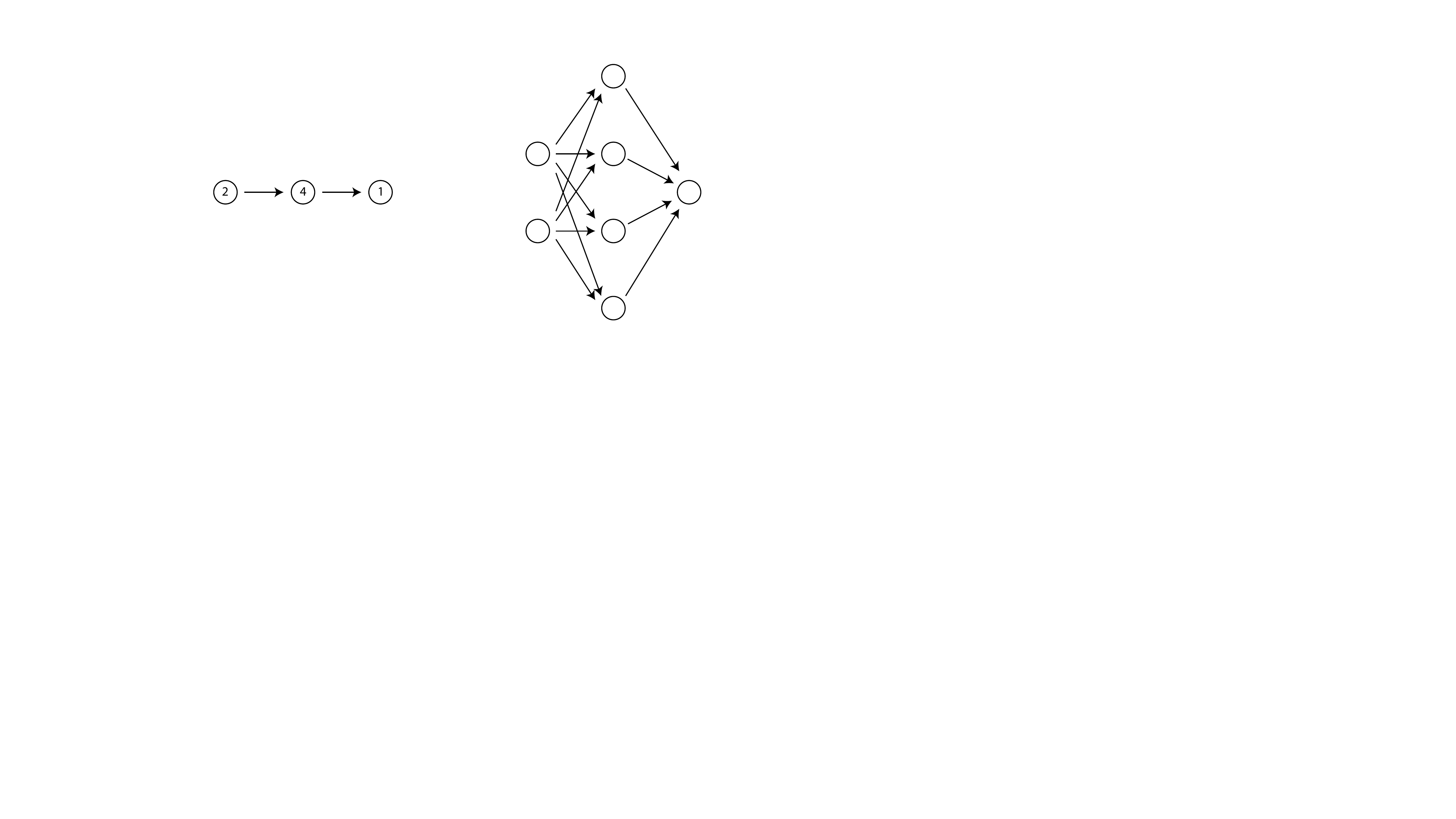}
		\caption{An example of Abelianization.  The LHS shows a quiver $Q$ together with the dimension vector $\vec{d}$.  The RHS shows $Q^{\mathrm{Ab},\vec{d}}$.}
		\label{fig:Abel}
	\end{figure}
	
	Given a dimension vector $\vec{n} \in Q_0$, define $\vec{n}^{\textrm{Ab}} \in Q^{\textrm{Ab},\vec{d}}_0$ by $\vec{n}^{\textrm{Ab}}_{(i,p)} = \vec{n}_i$ for all $p$.  The relation between $\cM^Q_{\vec{n},\vec{d}}$ and the toric variety $\cM^{Q^{\textrm{Ab},\vec{d}}}_{\vec{n}^{\textrm{Ab}},\vec{1}}$ is known as Abelianization and is well-studied in \cite{Martin}.  The basic example is $\Gr(n,d)$ (which is the framed moduli for the quiver with a single vertex), whose Abelianization is $(\bP^n)^d$ (the disconnected quiver with $d$ vertices and no arrow).
	
	Namely, the moduli spaces $\cM^Q_{\vec{n},\vec{d}}$ and $\cM^{Q^{\textrm{Ab},\vec{d}}}_{\vec{n}^{\textrm{Ab}},\vec{1}}$ are GIT quotients of the same vector space $R^Q_{\vec{n},\vec{d}} = R^{Q^{\textrm{Ab},\vec{d}}}_{\vec{n}^{\textrm{Ab}},\vec{1}}$ by $\GL_{\vec{d}}$ and $(\C^\times)^{\Sigma \vec{d}}$ respectively.  More precisely, we have the fiber bundle $\mu_{U_{\vec{d}}}^{-1}(\{I\}) / T \to \cM^Q_{\vec{n},\vec{d}}$ with fibers being a product of complete flags $U_{\vec{d}}/T^{\Sigma \vec{d}}$, and the inclusion $\mu_{U_{\vec{d}}}^{-1}(\{I\}) / T \subset \cM^{Q^{\textrm{Ab},\vec{d}}}_{\vec{n}^{\textrm{Ab}},\vec{1}}$.  The universal bundles $\cV_i$ over $\cM^Q_{\vec{n},\vec{d}}$ is descended from the direct sum of universal line bundles $\bigoplus_{p=1}^{d_i} \cV_{(i,p)}$ over $\cM^{Q^{\textrm{Ab},\vec{d}}}_{\vec{n}^{\textrm{Ab}},\vec{1}}$ (restricted to the above subset).  The cohomology of $\cM^Q_{\vec{n},\vec{d}}$ is generated by the Chern classes $c_k(\cV_i)$, which can be written as the $k$-th elementary symmetric polynomials in $c_1(\cV_{(i,p)})$ for $p=1,\ldots,d_i$.  On the other side, the cohomology of $\cM^{Q^{\textrm{Ab},\vec{d}}}_{\vec{n}^{\textrm{Ab}},\vec{1}}$ is generated by $c_1(\cV_{(i,p)})$.
	
	Note that the functions $f^Q_{\tilde{\gamma}}$ and $\cE^Q$ over  $\cM^Q_{\vec{n},\vec{d}}$ constructed in Section \ref{sec:ML} \emph{cannot} be lifted to $\cM^{Q^{\textrm{Ab},\vec{d}}}_{\vec{n}^{\textrm{Ab}},\vec{1}}$.  The reason is that, $f^Q_{\tilde{\gamma}}$ and $\cE^Q$ are $\GL_{\vec{d}}$-equivariant functions on $R^{Q,s}_{n,d}$, the subset of stable representations, rather than the whole $R^{Q}_{n,d}$.  The definition of $f^Q_{\tilde{\gamma}}$ involves the metrics on the universal bundles $\cV_i$, which take the expression $(\rho_i \rho_i^*)^{-1}$, and it is only defined over $R^{Q,s}_{n,d}$ where $\rho_i$ is surjective.  Rather, we have the functions $f^{Q^{\textrm{Ab},\vec{d}}}_{\tilde{\gamma}}$ and $\cE^{Q^{\textrm{Ab},\vec{d}}}$ on $\cM^{Q^{\textrm{Ab},\vec{d}}}_{\vec{n}^{\textrm{Ab}},\vec{1}}$, which uses the metrics $(\rho_{i,p} \rho_{i,p}^*)^{-1}$ on the line bundles $\cV_{i,p}$.
	
	Previously we have taken $\vec{n} = \vec{d} + \vec{1}$ for $\cM^Q_{n,d}$.  After Abelianization,  the dimension vectors $\vec{n}^{\mathrm{Ab}}$ and $\vec{1}$ for  $\cM^{Q^{\textrm{Ab},\vec{d}}}_{\vec{n}^{\textrm{Ab}},\vec{1}}$ no longer satisfy such equality.  This is actually not a problem, since the function $f^{Q^{\textrm{Ab},\vec{d}}}_{\tilde{\gamma}}$ defined on $\cM^{Q^{\textrm{Ab},\vec{d}}}_{\vec{2},\vec{1}}$ can be lifted to $\cM^{Q^{\textrm{Ab},\vec{d}}}_{\vec{n}^{\textrm{Ab}},\vec{1}}$.  Alternatively, we can set the $k$-th framing vectors to be zero for all $k=2,\ldots, d_i+1$ and for all vertices $(i,p)\in Q^{\textrm{Ab},\vec{d}}_0$.  This gives a subvariety of  $\cM^{Q^{\textrm{Ab},\vec{d}}}_{\vec{n}^{\textrm{Ab}},\vec{1}}$ which is isomorphic to $\cM^{Q^{\textrm{Ab},\vec{d}}}_{\vec{2},\vec{1}}$.

\section{Universal Approximation Theorem} \label{sec:approx}

In Section \ref{sec:toric}, we have introduced the multi-variable functions $\sigma$ coming from moment maps of toric varieties.  For instance, $\sigma_k(x)=\frac{e^{2 x_k}}{1+\sum_{j=1}^{d} e^{2 x_j}}$ for $X=\bP^d$.  In this section, we will give a theoretical basis for using this as an activation function, by proving the universal approximation theorem for this function.

The universal approximation theorem ensures that in theory, any given function on a compact set can be approximated (as close as you want) by the functions produced from directed graphs (denoted by $f^U_{[V,e]}$ in Section \ref{sec:flat}).  There are several different versions of this theorem \cite{Cybenko, LESHNO1993861, NEURIPS2018_03bfc1d4, lu2017expressive}.  To the authors' knowledge, the past works have focused on proving the theorem for single-variable activation functions.

In the work of Cybenko in proving the theorem below, rescaling on the domain of the activation function plays a key role.  The rescaling technique will also be very useful in our situation.

\begin{theorem}[\cite{Cybenko}]
	Let $\phi: \R \to \R$ be any continuous function with $\lim_{x\to\infty}\phi(x)=1$ and $\lim_{x\to-\infty}\phi(x)=0$. Let $K$ be a compact set in $\R^d$.
	Then the collection of functions $G: K \to \R$ of the form
	$$G(x)=\sum_{j=1}^N\alpha_j\phi(y_j^Tx+\theta_j)$$
	where $N \in \Z_{>0}$,$y_j \in \R^d$, and $\theta_j, \alpha_j \in \R$,
	are dense in the space of continuous functions $C(K)$.
\end{theorem}

The above function $G$ can be understood as $f^U_{W_1,W_2,b}$ \eqref{eq:f^U-A3} produced from the graph $A_3$, when the dimension at the output vertex is $d_3 = 1$, and $\sigma: \R^{d_2} \to \R^{d_2}$ is taken to be $\sigma(\vec{x})=(\phi(x_1),\ldots,\phi(x_{d_2}))$.  (Take $d_2 = N$, $W_1 = (y_j^T)_{j=1,\ldots,N}$, $b=(\theta_j)_{j=1,\ldots,N}$, and $W_2 = (\alpha_j)_{j=1,\ldots,N}$.) 
For general dimension $d_3$, we simply have $(G_1,\ldots,G_{d_3})$, where $G_i$ are of the same form as above (with different $\alpha_{j,i}$) which can be used to approximate any given continuous function $K\to \R^{d_3}$.

We will prove the following theorem.
Consider the quiver with three vertices as in Section \ref{sec:eg}, and the function $f^U_{W_1,W_2,b}(v)=W_2(\sigma(W_1\cdot v + b))$ in \eqref{eq:f^U-A3}, where $\sigma$ is the multi-variable activation function on $\R^{d_2}$ made from $\bP^{d_2}$.

\begin{theorem} \label{thm:app}
	Let $K$ be a compact set of $\R^{d_1}$, and $f: K \to \R^{d_3}$ a continuous function.  For any $\epsilon>0$, there exists $d_2 > 0$ and $W_1 \in \Mat(d_2,d_1)$, $W_2 \in \Mat(d_3,d_2)$, $b \in \R^{d_2}$ such that $\|f^U_{W_1,W_2,b} - f\|_{L^2(K)} < \epsilon$.
\end{theorem}

The compact set is given as a subset in $\R^n$.  Thus from now on we restrict to the real field, which will suffice for the theorem.  This means we take real-valued matrices and the real part $\sigma$ of $\sigma_\C$.




\subsection{Tropical limit}
A crucial idea in the work of \cite{Cybenko} is to compose $\phi$ with a rescaling, so that it tends to a step function in the limit.  We can apply such a rescaling to the multi-variable function 
$\sigma: \R^d \to P^\circ$.  This is well-known in toric geometry and is called the tropical limit. 

Let $\Sigma$ be the dual fan of the moment polytope $P$.  $\Sigma$ is the collection of cones that are dual to the boundary strata of the polytope $P$.  In particular, maximal cones of $\Sigma$ are one-to-one corresponding to corners of $P$.  

We assume that $|\Sigma| = \R^d$.  $\R^d$ is stratified into the relative interiors of cones in $\Sigma$.
We recall the following interesting fact from toric geometry.  It plays an important role in the study of holomorphic discs and Lagrangian Floer theory for toric varieties.

\begin{lemma} \label{lem:sigma_infty}
	Let $X_\Sigma$ be a toric variety (equipped with any toric K\"ahler form).  Let $C$ be a cone in $\Sigma$.  For any $x\not=0$ which lies in the relative interior of $C$, $p_C=\lim_{t\to -\infty} \sigma(t x)$ exists and equal to a point (which is independent of $x$) in the boundary stratum of $P$ that is dual to $C$.
	
	In other words, the family of functions $\sigma_t(x) = \sigma(tx)$ converges (as $t \to +\infty$) to the discontinuous function $\sigma_\infty$, where $\sigma_\infty (x) = p_C$ if $x$ belongs to the relative interior of $C$.
\end{lemma}

\begin{proof}
	The cone $C$ corresponds to a complex torus orbit of the toric variety $X$.  To be more explicit, consider a maximal cone $C^{\mathrm{max}}$ that contains $C$.  Without loss of generality, let $C^{\mathrm{max}} = \R_{\geq 0} \cdot\{v_1,\ldots,v_d\}$, and $C = \R_{\geq 0} \cdot\{v_1,\ldots,v_k\}$.  $C^{\mathrm{max}}$ gives a local chart $\C^d$ of the toric variety, and the complex torus orbit corresponding to $C$ is given by $z_1=\ldots=z_k=0$.
	
	We have a special point given by $z_1=\ldots=z_k=0,\, z_{k+1}=\ldots=z_{d}=1$ in the orbit.  We assert that $p_C=\lim_{t\to -\infty} \sigma(t x) \in P$ (for any $x$ in the relative interior of $C$) is the moment-map image of this point.
	
	To see this, we write $x = \sum_{i=1}^k x_i v_i$ where $x_i\not=0$ for all $i=1,\ldots,k$.  Consider the lifting of $tx = (tx_1,\ldots,tx_k,0,\ldots,0)$: $(e^{tx_1},\ldots,e^{tx_k},1,\ldots,1)$ in the chart $\C^d \subset X$.  Then $\sigma(tx)$ is the moment-map image of $(e^{tx_1},\ldots,e^{tx_k},1,\ldots,1)$.  Taking $t\to -\infty$, $(e^{tx_1},\ldots,e^{tx_k},1,\ldots,1) \to (z_1=\ldots=z_k=0, z_{k+1}=\ldots=z_{d}=1)$.  Thus $\sigma(tx)$ converges to the above special point $p_C$.
\end{proof}

In terms of solving equations, $p_C$ is the solution of the simultaneous equations $x_1=\ldots=x_k=0$ and $x_i \prod_{j=d+1}^m \ell_j^{v_{j,i}}\left(0,\ldots,0,x_{k+1},\ldots,x_d\right) = 1$ for $i=k+1,\ldots,m$.  (We have used the dual basis of $\{v_1,\ldots,v_d\}$ to write the coordinates of $P$, and $v_j = \sum_{i=1}^d v_{j,i} v_i$.) By above, the solution exists and is unique.

\begin{example}  \label{ex:Pd}
	Consider $X=\bP^d$.  Denote the coordinates of $\R^d$ by $(x_1,\ldots,x_d)$, and set $x_0 \equiv 0$.  The $(d+1-l)$-cones $C$ of $\Sigma$ are given by $\left\{x_{i_1}=\ldots=x_{i_l} > x_k \textrm{ for all } k\in \{0,\ldots,d\}-\{i_1,\ldots,i_l\}\right\}$, where $i_1,\ldots,i_l \in \{0,\ldots,d\}$ are fixed, $l=1,\ldots,d+1$.  For $\sigma = \left( \frac{e^{2 x_p}}{1+\sum_{j=1}^{d} e^{2 x_j}} \right)_{p=1}^d$, the point $p_C=\lim_{t\to -\infty} \sigma(t x)$ has coordinates $(p_C)_{i_r} = 1/l$ for $r=1,\ldots,l$ and $i_r \not=0$, and $(p_C)_j = 0$ for all other $j\not=i_1,\ldots,i_l,0$.  In particular, for the maximal cones $S_i = \{x_i > x_k \textrm{ for all } k\in \{0,\ldots,d\}-\{i\}\}$, $p_{S_i} = \epsilon_i$ for $i=0,\ldots,d$ where $\epsilon_0 := 0$ and $\{\epsilon_1,\ldots,\epsilon_d\}$ denotes the standard basis.
\end{example}

\begin{figure}[htb!]
	\includegraphics[scale=0.3]{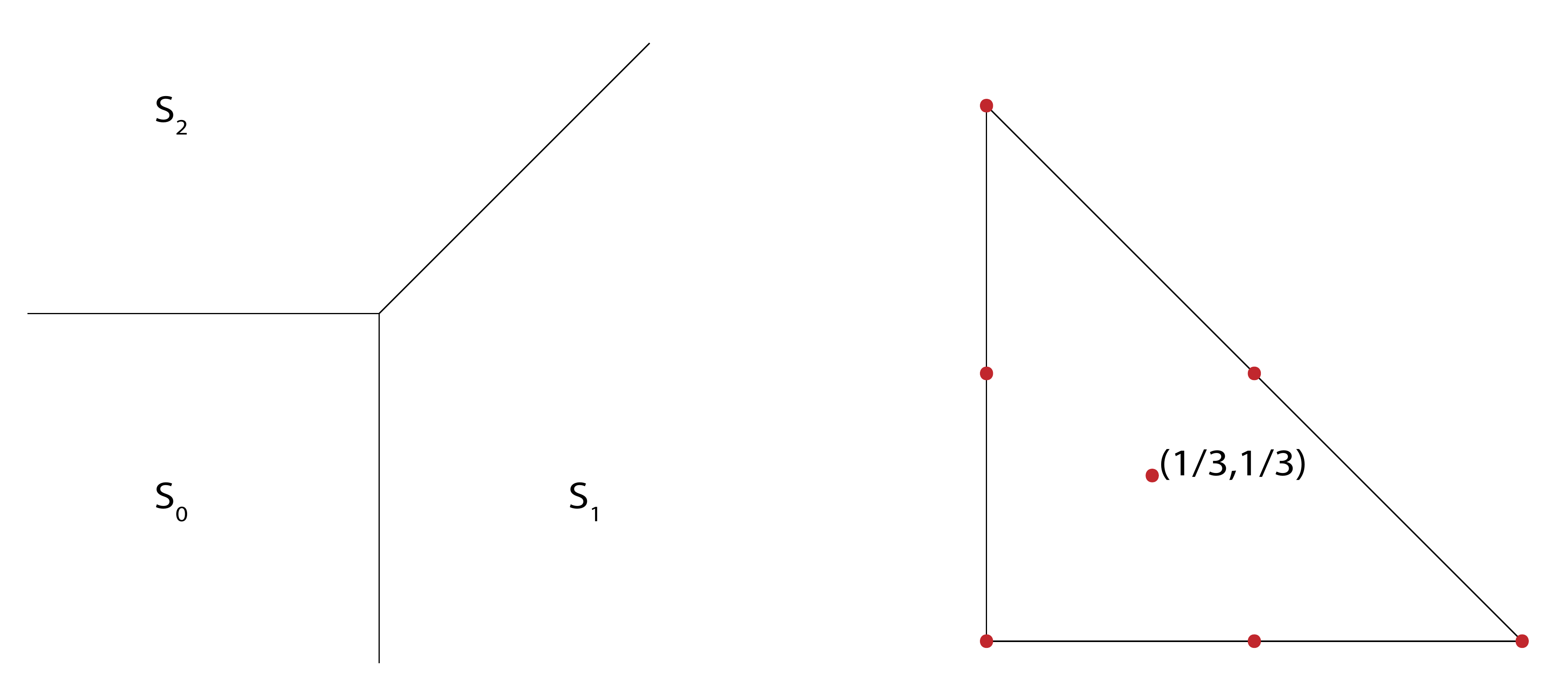}
	\caption{The left shows the fan picture of $\bP^2$, and the right shows the moment-map polytope.  The dots show the limit points $p_C$ for each cone $C$ of the fan.}
	\label{fig:P2}
\end{figure}

\begin{corollary} \label{cor:sigma_inf}
	Let $K$ be any compact set in $\R^d$.  For any $\epsilon>0$ and an open neighborhood of the union $U$ of codimension-one strata of $\Sigma$, there exists $t\gg 0$ such that $|\sigma_t(x) - \sigma_\infty(x)| < \epsilon$ for all $x \in K-U$.  ($\sigma_\infty$ is defined in Lemma \ref{lem:sigma_infty}.)
\end{corollary}
\begin{proof}
	Any $x \in K-U$ belongs to one of the maximal cones $C$.
	By Lemma \ref{lem:sigma_infty}, $\sigma_t(x)$ converges to $p_C$.  Moreover, both $\sigma_t$ and $\sigma_\infty$ are continuous on $K-U$.  Then the result follows from the compactness of $K-U$.
\end{proof}

In order to prove Theorem \ref{thm:app}, we consider a particular type of polyhedral decompositions of $\R^n$, which we call to be a centered simplicial web.

\subsection{Centered polyhedral web}

\begin{defn} \label{def:web}
	A centered simplicial web with $N$ ordered compact chambers in $\R^n$ is a polyhedral decomposition of $\R^n$ whose vertices are all trivalent, defined inductively on the number of compact chambers as follows.
	
	A centered simplicial web with zero compact chamber is the polyhedral decomposition given by the fan of $\bP^n$, up to an affine linear isomorphism in $\GL(n,\R) \ltimes \R^n$.
	
	Now suppose the notion of a centered simplicial web with $N$ ordered compact chambers has been defined, which has exactly $(n+1)$ non-compact rays (which we call the outer rays), whose corresponding infinite lines intersect at exactly one point called the $N$-th center that lies in the union of the $N$ compact chambers.   Moreover, the web is required to have $(n+1)$ non-compact chambers; each non-compact chamber is adjacent to $n$ outer rays and opposite to the remaining one outer ray.  (`Opposite' here means that the non-compact chamber is disjoint from the corresponding outer ray, whose infinite line intersects the chamber at a half-line.)  The outer rays are one-to-one corresponding to their opposite non-compact chambers.
	
	A centered simplicial web with $(N+1)$ ordered compact chambers is defined as follows.  First, take a centered simplicial web with $N$ ordered compact chambers.  Second, we choose a non-compact chamber, and denotes the direction of its opposite ray by a non-zero vector $v$.  Third, we take an affine hyperplane which intersects all the relative interior of the $n$ adjacent rays of the non-compact chamber.  This bounds a new compact chamber and the intersection points $V_i$ are the new vertices.
	Finally, we choose the $(N+1)$-th center to be $c_{N+1}=c_N-tv$, where $c_N$ is the $N$-th center, and $t\in \R_{>0}$ is taken such that $c_{N+1}$ lies in the union of the compact chambers (including the new one).  Then we have $n$ new rays emanated from the vertices $V_i$ whose infinite lines pass through $c_{N+1}$.  This gives a new web with $(N+1)$ ordered compact chambers, and it still has $(n+1)$ non-compact chambers, each of which is adjacent to $n$ outer rays and opposite to one outer ray.
\end{defn}

See Figure \ref{fig:trop-eg} for some examples of centered simplicial webs in $\R^2$.

\begin{figure}[htb!]
	\includegraphics[scale=0.3]{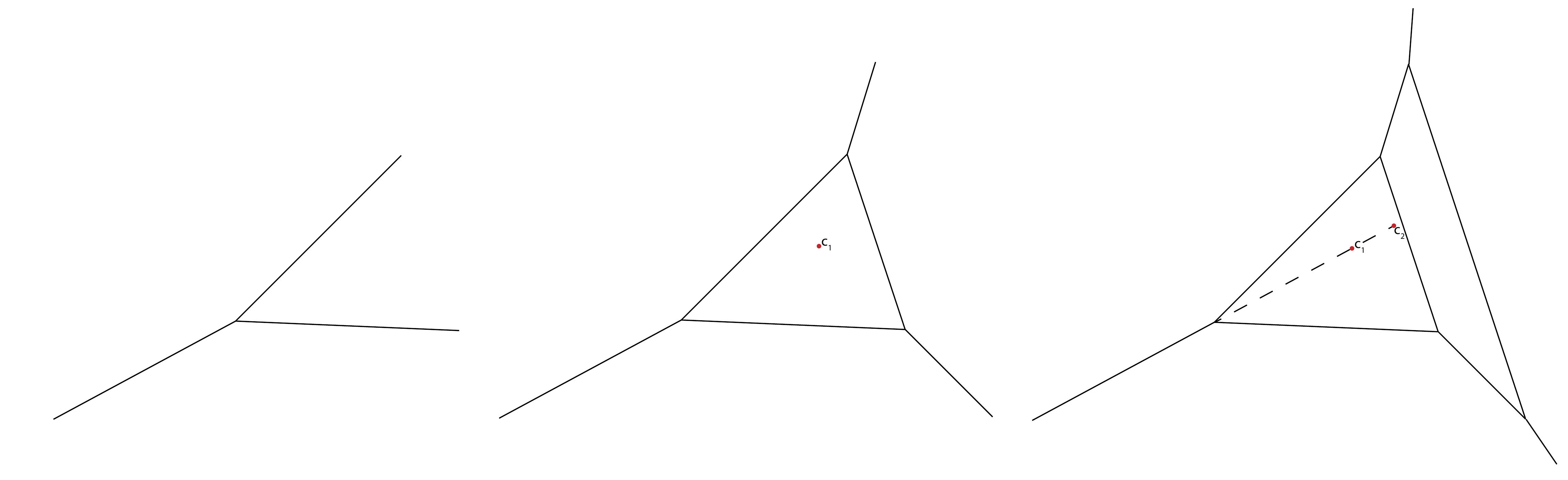}
	\caption{Examples of centered simplicial web in $\R^2$.  They have zero,one, and two compact chambers respectively.}
	\label{fig:trop-eg}
\end{figure}

From the above definition, there is a one-to-one correspondence between the centers and compact chambers.  Moreover, the $(k+1)$-th compact chamber $C_{k+1}$ (for $k=1,\ldots,N-1$) is associated with a one-strata of the web, which is a subset of the opposite ray of the non-compact chamber containing $C_{k+1}$ in the $(k+1)$-th inductive step.  Furthermore, both the $k$-th and $(k+1)$-th centers lie in the infinite line of the associated 1-strata of $C_{k+1}$.

\begin{remark}
	The above notion is closely related to tropical subvarieties.  In the tropical context, there is an integral structure on the ambient space and the balancing condition (whose definition requires the integral structure) is imposed on each vertex of a tropical variety.  However, we do not have an integral structure here, since $\sigma_\C$ is defined on the universal cover $\C^d$ rather than $(\C^\times)^d$ (see Corollary \ref{cor:sympl} and Example \ref{ex:Pn}).  It means affine linear maps are taken over $\R$ rather than over $\Z$.  Instead of the balancing condition, we impose the notion of centers in the above definition.
\end{remark}

\begin{theorem} \label{thm:web}
	Given a centered simplicial web $A$ with $N$ ordered compact chambers in $\R^n$, there exists an affine-linear embedding $L:\R^n \to \R^d$, where $d=n+N$, such that the $L$-preimage of the fan of $\bP^d$ in $\R^d$ equals to the web $A$.
\end{theorem}

\begin{proof}
	We shall prove the following statement: given a centered simplicial web $A$ with $N \geq 1$ compact chambers in $\R^n$, there exists a centered simplicial web $B$ with $N-1$ compact chambers in $\R^{n+1}$ such that the intersection of $\R^n \cong \R^n\times\{0\}$ with $B$ equals to $A$.  Then by applying this statement $N$ times, we obtain a web $B^{(N)}$ with zero compact chamber in $\R^{n+N}$ whose intersection with $\R^n \times \{0\}$ gives $A$.  By an affine linear isomorphism on $\R^{n+N}$, $B^{(N)}$ is identified with the fan of $\bP^{n+N}$.  The required map $L$ is given by the composition of the inclusion $\R^n \times \{0\} \subset \R^{n+N}$ with this linear isomorphism.
	
	The above statement is proved by induction on $N$.  First consider the case $N=1$.  We take a point $V$ away from $\R^n \times \{0\} \subset \R^{n+1}$.  Then we take a cone at $V$ over the compact simplicial chamber of $A$.  Moreover,
	the line joining $V$ with the given center of $A$ intersects with the complement of the cone and produces a ray emanated from $V$.  This gives a simplicial web in $\R^{n+1}$ with no compact chamber, whose intersection with $\R^n \times \{0\}$ is exactly $A$.  (See Figure \ref{fig:trop-01}.)
	
	Now suppose it is true for $N$.  Consider a centered simplicial web $A$ in $\R^n \cong \R^n\times\{0\}$ with $N+1$ ordered compact chambers.  We can take away the $(N+1)$-th compact chamber $C=C_{N+1}$ (and forget the corresponding center $c_{N+1}$) and obtain a centered simplicial web $A'$ with $N$ compact chambers.  By inductive hypothesis, there exists a centered simplicial web $B'$ in $\R^{n+1}$ with $N-1$ compact chambers whose intersection with $\R^n\times\{0\}$ gives $A'$.
	
	The compact chamber $C$ of $A$ is contained in a non-compact chamber $C'$ of $A'$, which is the intersection of $\R^n\times\{0\}$ with a non-compact chamber $D'$ of $B'$.  
	$C'$ is opposite to an outer ray $R$ of $A'$, which is the intersection of $\R^n\times\{0\}$ with a non-compact 2-plane $P$ of $B'$.
	Note that the last two centers $c^A_N$ and $c^A_{N+1}$ of $A$ are contained in the line of $R$, and hence contained in the infinite 2-plane of $P$.  Consider the two rays of $B$ that are adjacent to $P$.  Denote the one which is opposite to the chamber $D'$ by $L$.  The other one is denoted by $L'$, which must be adjacent to $D'$.
	
	Now we construct a web $B$ in $\R^{n+1}$ whose intersection with $\R^n\times\{0\}$ gives $A$.  A point $V$ in the relative interior of the ray $L'$ is taken to be a new vertex.  Consider the line passing through $V$ and the last center $c^A_{N+1}$.  This line lies in the infinite 2-plane of $P$.  Thus for a generic choice of $V$, it must intersect with
	the infinite line of $L$ at a point, which we shall define as the new center $c^B_N$ for $B$.  $V$ is taken far away enough in the ray $L'$ so that the intersection point $c^B_N$ equals to $c^B_{N-1} - v$ for some vector $v$ in the direction of $L$.
	
	Consider the $n$ outer rays of $A'$ that are adjacent to the chamber $C'$.  They are the intersections of $\R^n\times\{0\}$ with the corresponding $n$ outer-2-planes of $B'$ that are adjacent to $L'$.  The last chamber $C$ of $A$ is formed by the hyperplane through the vertices taken in the relative interior of the $n$ outer rays of $A'$.	
	The lines joining $V$ to these vertices in $A$ lie in the outer-2-planes of $B'$, and hence intersect with the corresponding outer rays of $B'$ at certain points, which we take to be new vertices of $B$.  The hyperplane through $V$ and these new vertices bounds a new chamber.
	The lines joining $c^B_N$ with the new vertices produce the outer rays of the new web $B$.  This gives $B$ whose intersection with $\R^n\times\{0\}$ equals to $A$.
\end{proof}

\begin{figure}[htb!]
	\includegraphics[scale=0.5]{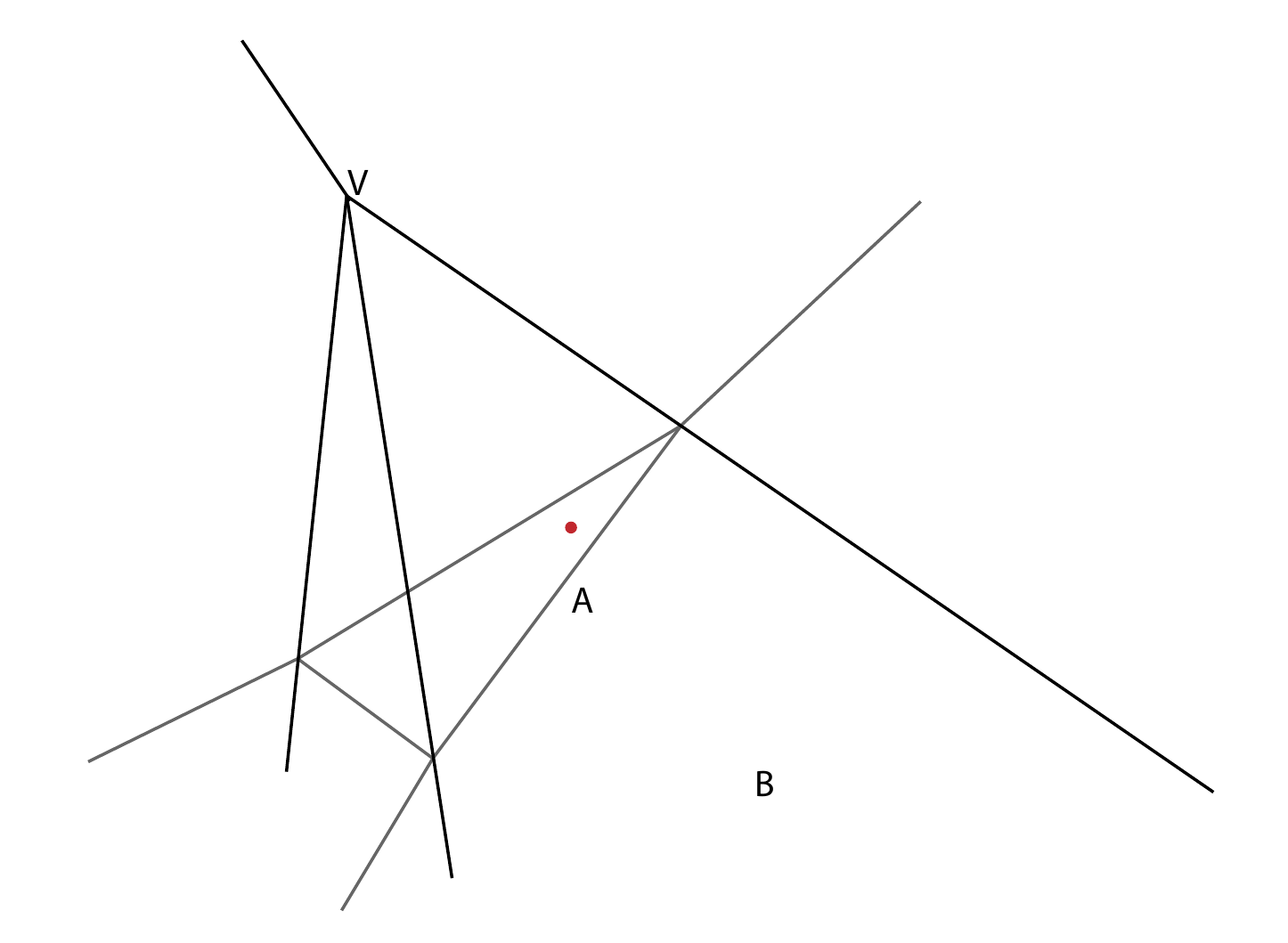}
	\caption{Construction of a tropical web with zero compact chamber whose intersection with a hyperplane equals to a given tropical web with one compact chamber.}
	\label{fig:trop-01}
\end{figure}

The inductive step in the above proof is illustrated by Figure \ref{fig:trop-N+1}.

\begin{figure}[htb!]
	\includegraphics[scale=0.3]{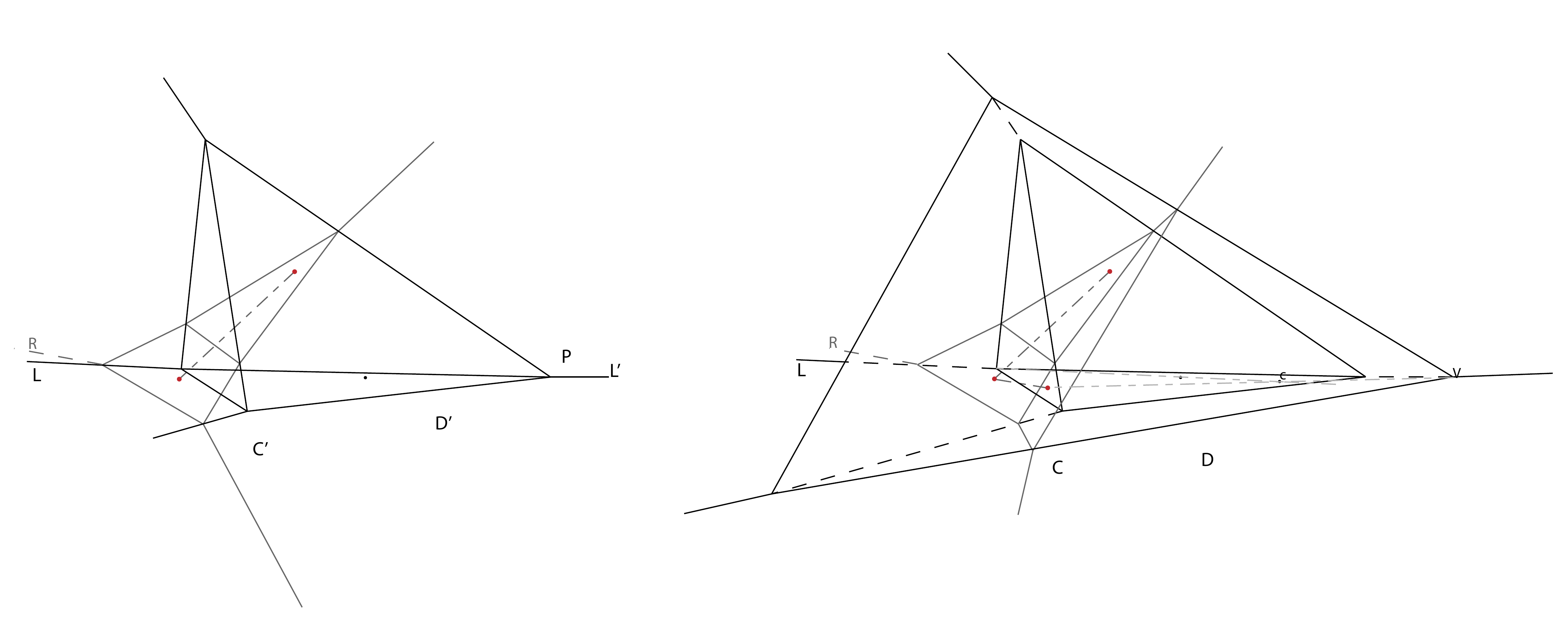}
	\caption{Construction of a tropical web with $3$ compact chambers whose intersection with a hyperplane equals to a given tropical web with $2$ compact chambers.}
	\label{fig:trop-N+1}
\end{figure}

Next, we consider polytopes rather than simplices and the corresponding webs formed from polytopes.  Motivated from the well-known fact below, we define a centered polyhedral web to be the intersection of a centered simplicial web with an affine subspace.

\begin{prop} \label{prop:polytope-simplex}
	For a polytope $P$ with $m$ facets in $\R^n$ where $m > n+1$, there exists a simplex $S$ in $\R^{m-1}$ 
	such that $S \cap (\R^{n} \times \{0\}) = P$ (where $\R^n$ is identified with $\R^{n} \times \{0\}$).
\end{prop}

The simplex $S$ in Proposition \ref{prop:polytope-simplex} can be constructed as follows.  Without loss of generality suppose $0 \in P$. Consider the dual polytope $P^\vee = \{\nu \in (\R^n)^*: (\nu,v) \leq 1 \textrm{ for all } v \in P \}$, which is the convex hull of its vertices $\nu_i$ for $i=1,\ldots,m$.  Then we have a surjective map from the standard simplex $S^\vee = \{\sum_i a_i \epsilon_i^* \in (\R_{\geq 0}^m)^*: \sum_i a_i = 1\}$ (where $\{\epsilon_i^*: i=1,\ldots,m\}$ is the standard basis) to $P^\vee$ by sending $\epsilon_i^*$ to $\nu_i$.  $S^\vee$ in the affine subspace $\{\sum_i a_i = 1\}$ can be identified as a simplex in $(\R_{\geq 0}^{m-1})^*$ by the projection along the direction $-\sum_i \epsilon_i$.  Then the dual linear map gives the desired linear injection $\R^n \to \R^{m-1}$ which sends $P$ into $S$.  By composing with a linear isomorphism, the image of $\R^n$ can be made to be $\R^n \times \{0\}$.

As a result, a centered polygonal web with one compact chamber (which is constructed by taking a polygon with a chosen center $c$ and outer rays at vertices whose lines pass through $c$) can be obtained as an intersection with $\R^n \times \{0\}$ of a centered simplicial web with one compact chamber in $\R^{m-1}$ (where $m$ is the number of non-compact chambers).  See the left of Figure \ref{fig:poly-web-1-chamber} for an example.

\begin{figure}[htb!]
	\includegraphics[scale=0.4]{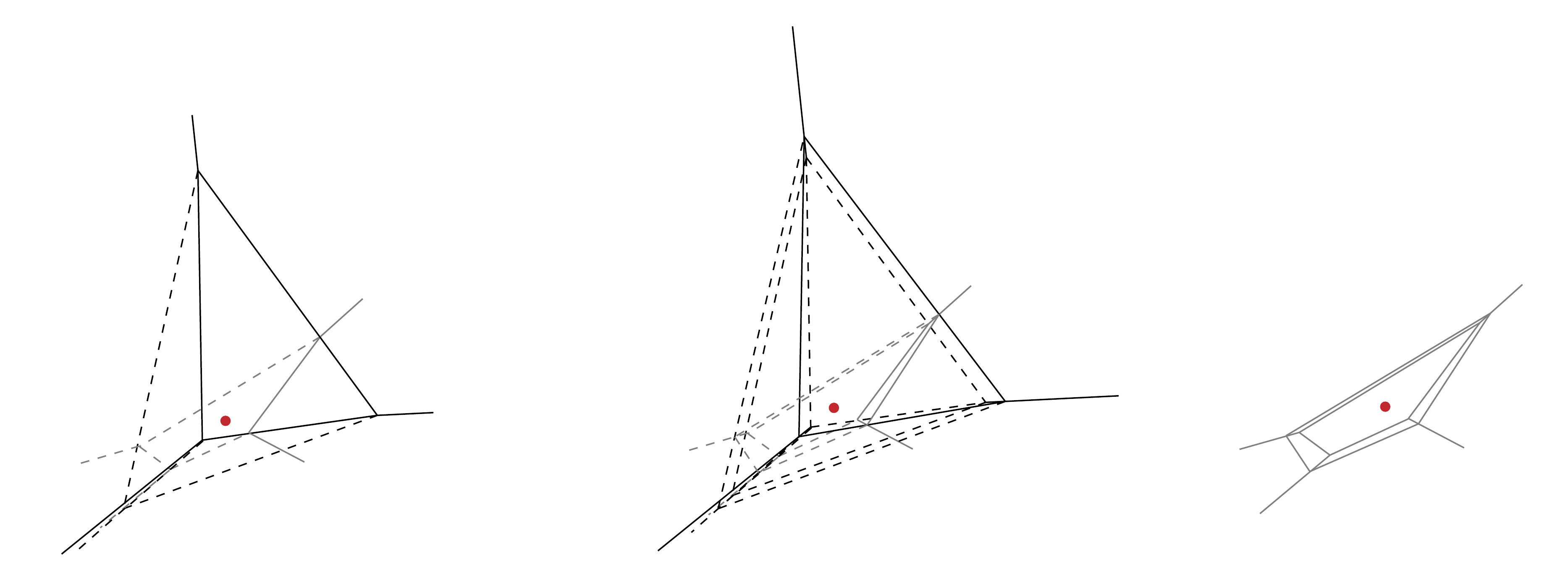}
	\caption{The left shows an example of a polyhedral web with one compact chamber, which is given as an intersection of a simplicial web with a subspace.  The right two figures show a concentric simplicial and a polygonal web.}
	\label{fig:poly-web-1-chamber}
\end{figure}

The following degenerate configuration will be helpful.  In Definition \ref{def:web}, suppose we take all the centers to be the same.  Moreover, suppose the new hyperplane introduced to bound a new chamber is allowed to intersect the outer rays at the original vertices (rather than their relative interior).   Then we can construct the following configuration.

\begin{defn}
	Let's take $(n+1)$ rays emanated from $0 \in \R^n$, such that any $n$ of them are linearly independent.  For each ray, we take a sequence of distinct points $V_{i,k}$ (where $i=1,\ldots,n+1$ is indexing the ray) such that $V_{i,k+1}-V_{i,k}$ is pointing in the ray direction.  Then for each $k>0$, we take a simplex with vertices at $V_{i,k}$ for $i=1,\ldots,n+1$.  This gives a polyhedral decomposition of $\R^n$.  This is called a concentric simplicial web.
\end{defn}

By taking an intersection of the above degenerate simplicial configuration with a subspace (that passes through the center $0$), we get a configuration made from a sequence of polytopes whose vertices lie in a fixed collection of $p$ rays, where $p$ is the number of vertices of each polytope.
We call this a concentric polyhedral web.  See the right of Figure \ref{fig:poly-web-1-chamber}.


\subsection{Proof of the approximation theorem}

We are now ready to prove Theorem \ref{thm:app}.

\begin{proof}[Proof of Theorem \ref{thm:app}]
	For any $\delta > 0$, we can take a concentric polyhedral web $B$ in $\R^{d_1}$, such that for every chamber $C$ of $B$, $C \cap K$ is contained in a $\delta$-ball.  $B$ is constructed as follows.  Without loss of generality, let $K \ni 0$.  First, we take a polytope $P$ that lies in a $\delta$-ball centered at $0 \in \R^{d_1}$.  $P$ induces a subdivision on the unit sphere $\bS^{d_1-1}$ by projecting its boundary strata onto $\bS^{d_1-1}$.  $P$ is taken with sufficiently many vertices such that the induced subdivision on the unit sphere $\bS^{d_1-1}$ lies in a $\delta'$-ball for a chosen $\delta'$.
	By Proposition \ref{prop:polytope-simplex}, $P = S \cap (\R^{d_1} \times \{0\})$ for some simplex $S$ in $\R^{m-1}$ where $m$ is the number of facets of $P$.  We take $0 \in S$ to be the center.  Then we take rays from $0$ through the vertices of $S$ and construct a concentric simplicial web.  Since $K$ is compact, by taking $\delta'$ sufficiently small, and the sequences of vertices in the rays sufficiently close to each other, the resulting concentric polyhedral web $B$ can be made such that every chamber intersects $K$ in a $\delta$-ball.
	
	Next, we take a centered simplicial web $A$ in $\R^{m-1}$ whose centers are chosen sufficiently close to each other, and the vertices in the inductive steps are taken such that $A \cap K$ is sufficiently close $B \cap K$. Namely, for every chamber $C$ of $A$, $C \cap K$ lies in the $\delta$-neighborhood of $C' \cap K$ for the corresponding chamber $C'$ of $B$.  In particular, $C \cap K$ is contained in a $2\delta$-ball.
	
	By Theorem \ref{thm:web}, there exists $L': \R^{m-1} \to \R^{d_2}$ such that $A$ is the $L'$-preimage of the fan $\Sigma_{\bP^{d_2}}$.  By composing $L'$ with $\R^{d_1} \times \{0\} \subset \R^{m-1}$, we obtain $L: \R^{d_1} \to \R^{d_2}$ such that for every maximal cone $S_i$ of $\Sigma_{\bP^{d_2}}$, $L^{-1}(S_i) \cap K$ lies in a $(2\delta)$-ball.
	
	Since $f$ is uniformly continuous in $K$, for every $\epsilon > 0$, $\delta$ can be taken such that $|f(x)-f(y)|<\epsilon$ for every $x,y \in K$ lying in a $2\delta$-ball.  In particular, we have a step function $s = \sum_{C} r_C \delta_C$ supported over $A$ (where $\delta_C(x) = 1$ for $x \in C$ and $0$ otherwise, and $C$ are chambers of $A$) such that $\|f - s\|_{L^2(K)} < \epsilon \sqrt{\mathrm{Vol}(K)}$.
	
	We have the step function $\sigma_\infty$ which sends the interior of the maximal cones $S_i$ of $\Sigma_{\bP^{d_2}}$ to $e_i$ for $i=0,\ldots,d_2$, where $e_0 = 0$.  (See Example \ref{ex:Pd}.)   The cone $S_i$ corresponds to chambers $L^{-1}(S_i)$ of $A$ under $L$.  Since $\{e_i-e_0: i=1,\ldots,d_2\}$ forms a basis, there exists a unique affine linear map $W_2: \R^{d_2} \to \R^{d_3}$ which sends $e_i$ to $r_{L^{-1}(S_i)}$ for all $i=0,\ldots,d_2$.  Thus $s = W_2 \circ \sigma_\infty \circ L$.
	
	Finally, by Corollary \ref{cor:sigma_inf}, there exists $t \gg 0$ such that $|W_2 \circ \sigma_t - W_2 \circ \sigma_\infty|<\epsilon / 2$ on $L(K)-U$, where $U$ is an arbitrary open neighborhood of the codimension-one strata of $\Sigma_{\bP^d_2}$.  Moreover, $|W_2 \circ \sigma_t - W_2 \circ \sigma_\infty|$ is bounded.  Hence by taking $\mathrm{Vol}(L^{-1}(U) \cap K)$ sufficiently small, we have $\|W_2 \circ \sigma_t \circ L - W_2 \circ \sigma_\infty \circ L\|_{L^2(K)} < \epsilon$.  In conclusion, $\|W_2 \circ \sigma \circ (tL) - f\|_{L^2(K)} \leq \| W_2 \circ \sigma \circ (tL) - W_2 \circ \sigma_\infty \circ L \|_{L^2(K)} + \|W_2 \circ \sigma_\infty \circ L - f\|_{L^2(K)}$ can be made arbitrarily small.
\end{proof}

\bibliographystyle{amsalpha}
\bibliography{geometry}	
\end{document}